\documentclass[a4paper,11pt]{amsart}
\usepackage[utf8]{inputenc}
\usepackage[T1]{fontenc}
\usepackage[UKenglish]{babel}
\usepackage[margin=18mm]{geometry}

\usepackage{color}%cyan,red;magenta,green;yellow,blue

\usepackage{graphicx}

\usepackage{amsmath,amssymb,amsfonts,amsthm}
\usepackage{mathrsfs,eucal,dsfont}
\usepackage{verbatim,enumitem}

\usepackage{hyperref,url}

\newcommand{\R}{\mathds R}

\newcommand{\N}{\mathds N}

\newcommand{\I}{\mathds 1}

\newcommand{\sL}{\mathcal{L}}

\def\d{{\rm d}}
\def\<{\langle}
\def\>{\rangle}

\newcommand{\supp}{\operatorname{supp}}

\def\R{\mathbb R}    
\def\N{\mathbb N}  
   
\def\<{\langle} \def\>{\rangle}  
    
\def\d{\text{\rm{d}}}   
   \def\ess{\text{\rm{ess}}}
 \def\beq{\begin{equation}}

\def\to{\rightarrow}
\def\8{\infty}\def\3{\triangle}
\def\1{\lesssim}

\renewcommand{\bar}{\overline}
\renewcommand{\hat}{\widehat}
\renewcommand{\tilde}{\widetilde}

\newtheorem{theorem}{Theorem}[section]
\newtheorem{lemma}[theorem]{Lemma}
\newtheorem{proposition}[theorem]{Proposition}

\theoremstyle{definition}

\newtheorem{remark}[theorem]{Remark}

\numberwithin{equation}{section}
\begin{document}
\allowdisplaybreaks
\title[
H\"{o}lder regularity of parabolic fractional $p$-Laplacian type equations] {
H\"{o}lder regularity of solutions to parabolic equations associated with fractional $p$-Laplacian type operators}

\author{Takashi Kumagai\qquad Jian Wang  \qquad \hbox{ and }\qquad
Meng-ge Zhang}

 \date{}

\maketitle

\begin{abstract}
In this paper, we study
H\"{o}lder regularity of local weak super-solutions to the
parabolic equation associated with fractional $p$-Laplacian type operators
$$
\partial_t u(t,x)-\text{p.v.}\int_{\R^d}|u(t,y)-u(t,x)|^{p-2}(u(t,y)-u(t,x))J(t;x,y)\,dy=0,\quad (t,x)\in \R\times \R^d,
$$
where $\text{p.v.}$ denotes the principal value,
$p\in(1,\infty)$ and
the kernel $J(t;x,y)$ is comparable (uniformly with respect to
 $(t;x,y)\in\R\times\R^d\times\R^d$) to the fractional $p$-Laplacian kernel $|x-y|^{-d-sp}$ with $s\in(0,1)$.
Our contribution extends the fundamental work of Kassmann and Weidner to the
nonlinear
case $p\ne 2$.
In particular, we provide an explicit nonlocal parabolic tail condition
in terms of an $L^{1+\varepsilon}$-norm in the time variable (for a suitable $\varepsilon>0$) for all $p\in(1,\infty)$,
under which H\"older regularity holds.
Our proof is based on the method of
expansion of positivity, which is the key ingredient
in the reduction-of-oscillation argument.

\medskip


\medskip

\end{abstract}
\allowdisplaybreaks

\section{Introduction and main result}
In this paper, we prove the local H\"{o}lder regularity of
weak
super-solutions
to the following equation
\begin{equation}\label{EQ}
\frac{\partial u(t,x)}{\partial t}-\sL_tu(t,x)=0,\quad (t,x)\in I\times D\subseteq \R\times \R^d,
\end{equation}
where $I$ is a bounded interval of $\R$, $D$ is a bounded domain of $\R^d$, and $\sL_t$ is a nonlocal (nonlinear) operator of the fractional $p$-Laplacian type with $p>1$; that is,
$$\sL_tu(t,x)=\text{p.v.}\int_{\R^d}|u(t,y)-u(t,x)|^{p-2}(u(t,y)-u(t,x))J(t;x,y)\,dy,\quad (t,x)\in\R\times\R^d.$$
Here, the symbol $\text{p.v.}$ denotes the principal value, and the kernel $J(t;x,y)$ satisfies that for any $t\in \R$ and $x,y\in \R^d$, $J(t;x,y)=
J(t;y,x)$, and that there are $\Lambda\ge1$ and $s\in (0,1)$ so that for all $t\in \R$ and $x,y\in \R^d$,
\begin{equation}\label{J0}
 \frac{\Lambda^{-1}}{|x-y|^{d+sp}}\le J(t;x,y)\le\frac{\Lambda}{|x-y|^{d+sp}}. \end{equation} When $p=2$ and
$J(t;x,y)=|x-y|^{-d-2s}$, the nonlocal operator $\sL_t$ boils down to the well-known fractional Laplacian $-(-\Delta)^s$  multiplied by a proper constant.

Note that the operator $\sL_t$ corresponds to the following nonlocal (nonlinear) $p$-form
$$\mathcal{E}_t(f,g)=\frac{1}{2}\iint_{\R^d\times\R^d}
|f(t,y)-f(t,x)|^{p-2}(f(t,y)-f(t,x))(g(t,y)-g(t,x))J(t;x,y)\,dx\,dy.$$
A function
$$u\in L^p_{{\rm loc}}(I;W^{s,p}_{{\rm loc}}(D))\cap
L^2_{{\rm loc}}(I;L^2_{{\rm loc}}(D))\cap L^\infty_{{\rm loc}}(I; L_{p-1,sp}(\R^d))$$
is called a local weak sub-solution
(resp. super-solution) to \eqref{EQ} in $I\times D$, if for any sub-interval $[t_1,t_2]\subset I$ and any compact set $D'\subset D$,
\begin{equation}\label{PE}
\int_{D'}u(t,x)\eta(t,x)\,dx\Big|^{t_2}_{t_1}-\int^{t_2}_{t_1}\int_{D'}u(t,x)\partial_t\eta(t,x)\,dx\,dt+\int^{t_2}_{t_1}\mathcal{E}_t(u,\eta)\,dt\le
0\,(\mbox{resp.}\ge 0),
\end{equation}
for all non-negative test functions
$\eta\in L^p_{{\rm loc}}(I;W^{s,p}_0(D'))\cap W^{1,2}_{{\rm loc}}(I;L_{{\rm loc}}^2(D))$.
A function $u$, which is both a local weak sub-solution and a local weak super-solution to \eqref{EQ}, is a local weak solution.
(The definitions of these function spaces above can be found in the Appendix.)
The existence and the
uniqueness of the local weak solution associated with the operator $\sL_t$
in $I\times D\subseteq \R\times \R^d$ is proved in \cite[Section 2.1.1]{St1}.

\subsection{Main result} For $n\ge1$
and any function $f$ defined on $D\subset\R^n$,
set $$\|f\|_{L^\infty(D)}=\sup_{x\in D}|f(x)|,\quad \oint_D f(x)\,dx=|D|^{-1}\int_D f(x)\,dx$$ and $f_-:=\max\{-f,0\}$.
We can now
state the main result of this paper, which shows the local H\"{o}lder continuity of weak super-solutions to the parabolic equation associated with the fractional $p$-Laplacian type operator $\sL_t$.
\begin{theorem}\label{main1}
Let
$p\in(1,\infty)$, and
$u$ be a locally bounded weak
super-solution to
\eqref{EQ} in $I\times D$. Then, $u$ is local H\"{o}lder continuous in $I\times D$.
More explicitly,
the following statements hold.

\begin{itemize}
\item[{\rm (i)}] Suppose that
$p\in
(2,\infty)$.
Then there exist constants $C_1\ge1$, $C_0>1$,
    $\varepsilon>0$
    and $\beta\in(0,sp/(p-2))$,
\begin{equation}\label{Holder}
|u(t_1,x_1)-u(t_2,x_2)|\le\frac{ C_1L}{R^\beta}\left(|x_1-x_2|+L^{({p-2})/({sp})}|t_1-t_2|^{{1}/({sp})}\right)^\beta
\end{equation}
for all $R>0$ and every pair of points $(t_1,x_1),(t_2,x_2)\in
(t_0-(C_0R)^{sp}/2,t_0)\times B_{C_0R/2}(x_0)$ so that
$(t_0-(C_0R)^{sp},t_0)\times B_{C_0R}(x_0)\subset I\times D$
with some $t_0\in \R$ and $x_0\in \R^d$, where
\begin{align*}
L
:=&\|u\|_{L^\infty ((t_0-2(C_0R)^{sp},t_0)\times B_{2C_0R}(x_0))}+ \left(C_0^{\beta(p-2)-sp}\right)^{1/(p-2)}\\
&
+\left(\oint_{t_0-2(C_0R)^{sp}}^{t_0}\left((C_0R)^{sp}\int_{\R^d\setminus B_{C_0R/2}(x_0)}\frac{
u_-(t,x)^{p-1}}{|x-x_0|^{d+sp}}\,dx\right)^{1+\varepsilon}\,dt\right)^{1/(1+(p-1)\varepsilon)}.
\end{align*}

\item[{\rm(ii)}] Suppose that
$p\in (1,2)$.
Then there exist constants $C_1\ge1$, $C_0>1$,
    $\varepsilon>0$
    and $\beta\in(0,sp/(2-p))$,
\begin{equation}\label{Holder2}
|u(t_1,x_1)-u(t_2,x_2)|\le\frac{C_1L}{R^\beta}\left(L^
{(2-p)/(sp)}|x_1-x_2|+|t_1-t_2|^{1/({sp})}\right)^\beta
\end{equation}
for all $R>0$ and every pair of points $(t_1,x_1),(t_2,x_2)\in
(t_0-(C_0R)^{sp}/2,t_0)\times B_{C_0R/2}(x_0)$ so that $(t_0-(C_0R)^{sp},t_0)\times B_{C_0R}(x_0)\subset I\times D$ with some $t_0\in \R$ and $x_0\in \R^d$, where
\begin{align*}
L:=&\|u\|_{L^\infty ((t_0-2(C_0R)^{sp},t_0)\times B_{2C_0R}(x_0))}+
\left(C_0^{\beta(2-p)-sp}\right)^{1/(2-p)}\\
&
+\left(\oint_{t_0-2(C_0R)^{sp}}^{t_0}\left((C_0R)^{sp}\int_{\R^d\setminus B_{C_0R/2}(x_0)}\frac{
u_-(t,x)^{p-1}}{|x-x_0|^{d+sp}}\,dx\right)^{1+\varepsilon}\,dt\right)^{1/(1+\varepsilon)}.
\end{align*}
\end{itemize}

\end{theorem}

\begin{remark}\label{R:rem}
Below are some
comments on Theorem \ref{main1} and its proof.
 \begin{itemize}
 \item[{\rm(i)}] When $p=2$, both assertions \eqref{Holder} and \eqref{Holder2} in Theorem
    \ref{main1} are reduced
      to
     (thanks to the fact that $C_0>1$)
\begin{equation*}
|u(t_1,x_1)-u(t_2,x_2)|\le\frac{ C_1L}{R^\beta}\left(|x_1-x_2|+|t_1-t_2|^{{1}/({2s})}\right)^\beta,
\end{equation*} where
\begin{align*}
L:=&\|u\|_{L^\infty ((t_0-2(C_0R)^{2s},t_0)\times B_{2C_0R}(x_0))}\\
&+\left(\oint_{t_0-2(C_0R)^{2s}}^{t_0}\left((C_0R)^{2s}\int_{\R^d\setminus B_{C_0R/2}(x_0)}\frac{u_-(t,x)}{|x-x_0|^{d+2s}}\,dx\right)^{1+\varepsilon}\,dt\right)^{1/(1+\varepsilon)}.\end{align*}
This is consistent with \cite[Theorem 1.5]{KW} for the linear setting.
\item[{\rm(ii)}]
To the best of our knowledge,
an explicit H\"{o}lder regularity theory for weak super-solutions of the parabolic equation associated with fractional $p$-Laplacian type operators, as stated in Theorem \ref{main1}, does not appear in the literature.
In particular, no previous work treats this problem via the expansion of positivity method combined with a tail measured in the time variable by an $L^{1+\varepsilon}$-norm.
Unlike the existing results \cite[Theorem 3]{DZZ} and \cite[Theorems 2.6 and 2.7]{APT}, our
approach controls the decay of oscillation through a nonlocal parabolic tail using an $L^{1+\varepsilon}$-time norm instead of the $L^\infty$-time norm used in those papers.
Moreover, Theorem \ref{main1} covers H\"{o}lder continuity of weak super-solutions for all $p>1$, whereas \cite[Theorem 3]{DZZ} treats only the case $p>2$.

\item[{\rm(iii)}] Recently, a general modulus of continuity
was quantified for locally bounded weak solutions of
the parabolic equation associated with fractional $p$-Laplacian type operators for all $p>1$; see \cite[Theorem 1.2]{LN3}.
Like Theorem \ref{main1}, \cite[Theorem 1.2]{LN3} is formulated using a nonlocal parabolic tail measured by an $L^{1+\varepsilon}$-norm in the time variable.
However, there is an essential difference between \cite[Theorem 1.2]{LN3} and Theorem \ref{main1}.
The oscillation inequality in \cite[Theorem 1.2]{LN3} involves a quantity $\omega$ that depends on the weak solution $u$, which distinguishes it from Propositions \ref{pro2} and \ref{pro4} in our paper. Consequently, the modulus of continuity in \cite[Theorem 1.2]{LN3} cannot be applied directly to weak solutions when the solution (or its a priori estimates) is not already known, because the relevant domain is not explicit. (See also \cite[Theorem 1.1]{LN2} and \cite[Theorem 1.1]{LN3} for related statements.) By contrast, Theorem \ref{main1} is directly applicable to weak solutions, and this direct applicability is a primary novelty of our result.

\item[{\rm(iv)}]  Our approach
is
strongly inspired by the methods of
\cite[Theorems 2.6 and 2.7]{APT}.
In particular, the proof of Theorem \ref{main1} relies on the expansion of positivity, the main tool for reducing oscillations. To implement the expansion of positivity,
we perform an exponential change of time variable and employ auxiliary results such as the De Giorgi-type lemma and the Shrinking lemma (see the proofs of Propositions \ref{pro1} and \ref{pro3} below or \cite[Propositions 5.1 and 6.1]{APT}).
Much of our analysis uses energy estimates and the De Giorgi-Nash-Moser iteration; however, the fact that the nonlocal parabolic tail in Theorem \ref{main1} is measured by an $ L^{1+\varepsilon}$-norm in time requires substantial, non-trivial modifications. For instance, the induction arguments in the proofs of Propositions \ref{pro2} and \ref{pro4} are completely different from those in \cite[Sections 7 and 8]{APT}. Moreover, we establish a new Caccioppoli-type inequality adapted to
 fractional $p$-Laplacian type operators for all $p>1$; see Lemma \ref{2} in the Appendix.
    \end{itemize}
\end{remark}

\subsection{Overview of the related
literature and the novelties of our result}
Local regularity of weak solutions to nonlocal elliptic equations of fractional $p$-Laplacian type operators
(i.e., $-\sL_tu=0$) has been well investigated for $p>1$,
cf. \cite{CM, DKP1, DKP2, KM}. In detail, \cite[Theorem 1.1]{KM}
proved that harmonic functions satisfy a priori estimates in H\"{o}lder spaces. \cite[Theorems 1.1 and 1.2]{DKP1} established the local boundedness and the
H\"{o}lder continuity of weak solutions to the elliptic counterpart of \eqref{EQ}, while a very interesting nonlocal version of the elliptic Harnack inequality was later developed in \cite[Theorems 1.1 and 1.2]{DKP2}. As indicated by \cite{DKP1, DKP2, KM}, a nonlocal logarithmic estimate (e.g., see \cite[Lemma 3.3]{KM} or \cite[Lemma 1.3]{DKP1}) played a key role in the proof of H\"{o}lder regularity. Cozzi \cite[Theorems 2.3, 2.4 and 2.5]{CM} proved the local boundedness, the H\"{o}lder continuity as well as the elliptic Harnack inequality of weak solutions by introducing the fractional De Giorgi classes for $-\sL_tu=f$ with a general source term $f(x,u)$.

The regularity theory of weak solutions to the nonlocal parabolic equations of type \eqref{EQ} has drawn considerable
attention in the last decade. Previously, local boundedness estimates for fractional parabolic $p$-Laplacian type equations were proved in \cite{DZZ, KWZ, PZ,St1}.
Regarding the H\"{o}lder regularity, recently there have been a few developments. Ding et al.\,\cite[Theorem 3]{DZZ} imposed some structural conditions on $f$ to ensure the H\"{o}lder continuity for $\partial_tu-\sL_tu=f$ when $p>2$. \cite[Theorems 2.6 and 2.7]{APT} extended the fundamental work of Caffarelli et al.\,\cite{CCV} on linear nonlocal parabolic equations to equations modeled by the fractional $p$-Laplacian type. Liao \cite[Theorem 1.1]{LN2} achieved the local H\"{o}lder  regularity for \eqref{EQ} by employing De Giorgi's iteration and the refined DiBenedetto's intrinsic scaling method. However, previous studies in those quoted papers
required a nonlocal supremum parabolic tail uniformly
with respect to the time variable. Subsequently, \cite[Theorem 1.2]{BK} also established the H\"{o}lder regularity of weak solutions to the nonlinear and nonlocal equations of type \eqref{EQ},
which extended \cite[Theorem 1.1]{LN2} by weakening regularity assumptions for the parabolic tail.
Furthermore, Kassmann and Weidner in \cite{KW} revealed a fundamental limitation: the parabolic H\"{o}lder estimate fails to hold in general, if the nonlocal parabolic tail is only of $L^1$-norm in the
time variable. Meanwhile, they demonstrated the H\"{o}lder continuity for linear nonlocal parabolic equations under a $L^{1+\varepsilon}$-norm tail with a counterexample establishing its sharpness. Liao \cite[Theorems 1.1 and 1.2]{LN3} later extended this result to the fractional $p$-Laplacian type operator by using a parabolic tail with the $L^{1+\varepsilon}$-norm for the time variable, and the approach in \cite{LN3} bypassed logarithmic estimates or exponential variable transformations. However, as mentioned in Remark \ref{R:rem}(iii),
the assertions of the H\"{o}lder regularity in \cite{BK,LN2,LN3} are expressed by the domain depending on the weak super-solution to \eqref{EQ} itself, which
does not seem to be
very practical in further applications.
Note that the exponential change of variable in time leads to an expansion of positivity result, which is suitable for proving H\"{o}lder regularity results. We refer to the monograph \cite{DGV} for this interesting topic. With the aid of the expansion of positivity, Liao \cite[Theorems 1.1 and 1.2]{LN1} has effected a direct proof of the H\"{o}lder regularity for parabolic $p$-Laplacian type equations in terms of a nonlocal supremum parabolic tail. As far as we know, there is no theory yet for the H\"{o}lder regularity of weak super-solutions to the nonlinear and nonlocal equation of type \eqref{EQ} that relies on expansion of positivity estimates, with a parabolic tail in terms of the $L^{1+\varepsilon}$-norm in
the time variable. Thus, the purpose of this work is to
answer
this problem affirmably for weak super-solutions to \eqref{EQ} with all $p>1$.

\medskip

We summarize the novelties of our results and their approaches
as follows:
\begin{itemize}
\item We establish H\"{o}lder regularity for locally bounded weak super-solutions
of parabolic equations driven by
fractional $p$-Laplacian type operators for all $p>1$.
Our approach to H\"{o}lder estimates is new:
we measure the nonlocal parabolic tail using the $L^{1+\varepsilon}$-norm in
the time variable. Moreover, the results
are directly applicable when the domain of the weak super-solution is known: the radii of the balls appearing
in the H\"{o}lder estimates can be chosen independently of any a priori information about the solution.

\item
We propose an efficient iterative scheme to establish
oscillation inequalities for locally bounded weak super-solutions
of parabolic equations driven by
fractional $p$-Laplacian type operators for all $p>1$.
In particular,
the modulus of oscillation now incorporates a new factor depending on the parabolic tail described above.
We also derive several
De Giorgi-type iteration lemmas for fractional $p$-Laplacian type operators
adapted to these new parabolic tails.
\end{itemize}

 \ \

The rest of the paper is arranged as follows. In Section \ref{section2}, we provide some preliminary tools for the proof of Theorem \ref{main1}, including the De Giorgi-type iteration lemma, the De Giorgi-type lemma with a forward in time version, the qualitative expansion of the positivity in time, and the Shrinking lemma. Section \ref{section3} is devoted to proving
Theorem \ref{main1} with the aid of the expansion of positivity. The proof here is split into two cases according to
$p\in (2,\infty)$ and $p\in (1,2)$, where the constants $L$ involved in the proofs are completely different. For readers' convenience, some known facts used in this paper are collected in Section \ref{section4}. In particular, we present the Caccioppoli-type inequality for nonlocal parabolic operators of the fractional $p$-Laplacian type for all $p>1$, which plays a key role in the proofs in Section \ref{section2}.

\section{Preliminary tools}\label{section2}

From this section, we fix $I\subset \R$, $D\subset \R^d$, $x_0\in \R^d$, $R>0$ and $T_2\ge T_1$ so that $\mathcal{Q}:=(T_1,T_2]\times B_R(x_0)\subset I\times D$. Set
$$\mathcal{Q}^-_\rho(\theta; (t_0,x_0))
:=(t_0-\theta\rho^{sp},t_0)\times B_\rho(x_0), \quad
\mathcal{Q}^+_\rho(\theta; (t_0,x_0))
:=(t_0,t_0+\theta\rho^{sp})\times B_\rho(x_0)$$ for $x_0\in\R^d$, $t_0\in (T_1,T_2]$, $\rho\in (0,R/2]$ and $\theta>0$, where $ B_\rho(x_0)=\{x\in \R^d:|x-x_0|<\rho\}$.

Let $p>1$ and $q\in (0,\infty]$.
For any $x_0\in\R^d$, $t_0\in \R$, $r>0$, $T_0>0$ and any function
$0\le u\in L^q (t_0-T_0,t_0;W^{s,p}(\R^d))$
(see the definition in the Appendix), the parabolic tail of $u$  is defined by
$$
{\rm Tail}_q
(u; x_0, r, t_0-T_0,t_0)
:=\left(\frac{1}{T_0}\int_{t_0-T_0}^{t_0}\left(r^{sp}\int_{\R^d\setminus B_{r}(x_0)}\frac{
u(t,x)^{p-1}}{|x-x_0|^{d+sp}}\,dx\right)^{q/(p-1)}\,dt\right)^{1/q}.
$$
In particular, we write
$$
{\rm \tilde{Tail}}^{p-1}_{p-1}(u; x_0, r, t_0-T_0,t_0)
:=\int_{t_0-T_0}^{t_0} \int_{\R^d\setminus B_{r}(x_0)}\frac{
u(t,x)^{p-1}}{|x-x_0|^{d+sp}}\,dx \,dt.$$
For simplicity, when $I_0=( t_0-T_0,t_0)$, we write ${\rm \tilde{Tail}}^{p-1}_{p-1}(u; x_0, r, t_0-T_0,t_0)$ as ${\rm \tilde{Tail}}^{p-1}_{p-1}(u; x_0, r, I_0)$.

We begin with the following De Giorgi-type iteration lemma.
\begin{lemma}\label{Lem1}
Let $u$ be a locally bounded weak super-solution to \eqref{EQ}
in $I\times D$.
Suppose that $\mathcal{Q}^-_{2\rho}(\theta; (t_0,x_0))\subset\mathcal{Q}$ with $\rho\in(0,R/2]$ and $\theta>0$. Then, there is a constant $C_1:=C_1(s,p,d)>0$ so that, for any $M
>0$ and $a\in(0,1)$ satisfying
\begin{equation}\label{mu}
\mu^-\le\mathop{\ess\inf}_{\mathcal{Q}^-_{2\rho}(\theta; (t_0,x_0))} u
\end{equation} and
$$|\{u\le\mu^-+  M\}\cap \mathcal{Q}^-_{2\rho}(\theta;(t_0,x_0))|\le\nu|\mathcal{Q}^-_{2\rho}(\theta;(t_0,x_0))|$$ with $\nu:=C_1(1-a)^{\frac{2s+d}{s}}\theta^{-1}M^{2-p}\big(1+{\theta}^{-1} M^{2-p}\big)^{-\frac{d+sp}{sp}}$,  it holds
that for all $\lambda\in (0,1)$ either
$$\tilde{\gamma}{\rm \tilde{Tail}}^{p-1}_{p-1}
((u-\mu^-)_-; x_0,2\rho, t_0-\theta(2\rho)^{sp}, t_0)>(1-\lambda)a  M,$$
or
$$u\ge\mu^-+a\lambda M
\quad a.e.\,\,in\,\, \mathcal{Q}^-_{\rho}(\theta;(t_0,x_0)),$$
where $\tilde{\gamma}:=\tilde{\gamma}(s,p,d)\ge1$.
\end{lemma}

\begin{proof}
Without loss of generality, we may and do assume $(t_0,x_0)=(0,0)$. The proof is split into three steps.

(1) For $a\in (0,1)$ and $\rho\in (0,R/2]$, define for $n=0,1,\cdots$,
$$k_n=\mu^-+a   M +\frac{(1-a)  M}{2^{n}},\quad \rho_n={\rho}+\frac{\rho}{2^{n}}.$$
Both $k_n$ and $\rho_n$ are decreasing with respect to $n$. Define $$\mathcal{Q}^-_n=(-\theta\rho^{sp}_n,0)\times B_n,\quad\hat{\mathcal{Q}}^-_n=(-\theta\hat{\rho}^{sp}_n,0)\times\hat{B}_n,\quad \tilde{\mathcal{Q}}^-_n=(-\theta\tilde{\rho}^{sp}_n,0)\times\tilde{B}_n,\quad \bar{\mathcal{Q}}^-_n=(-\theta\bar{\rho}^{sp}_n,0)\times\bar{B}_n,$$ where
 $$ \hat{\rho}_n=\frac{3\rho_n+\rho_{n+1}}{4},\quad \tilde{\rho}_n=\frac{\rho_n+\rho_{n+1}}{2},\quad \bar{\rho}_n=\frac{\rho_n+3\rho_{n+1}}{4}$$ and
$$B_n=B_{\rho_n},\quad \hat{B}_n=B_{\hat{\rho}_n},\quad \tilde{B}_n=B_{\tilde{\rho}_n},\quad \bar{B}_n=B_{\bar{\rho}_n}.$$
Since $\rho_n\ge \hat\rho_n\ge \tilde \rho_n\ge \bar \rho _n\ge \rho_{n+1}$, it holds that
$\mathcal{Q}^-_n\supset\hat{\mathcal{Q}}^-_n\supset\tilde{\mathcal{Q}}^-_n\supset\bar{\mathcal{Q}}^-_n\supset\mathcal{Q}^-_{n+1}.$
Choose $\phi_n(t,x)=\phi_{n,1}(t)\phi_{n,2}(x)$ such that $\phi_{n,1}:(-\theta\rho^{sp}_n,0)\to [0,1]$ is differentiable and satisfies $\phi_{n,1}(t)=1$ for all $t\in(-\theta\tilde{\rho}^{sp}_n,0)$ and $\phi_{n,1}(-\theta\rho^{sp}_n)=0$, and $\phi_{n,2}
:B_n\to [0,1]$ is a smooth cut-off function of $\tilde{B}_n\subset\hat{B}_n$, that is, $\phi_{n,2}(x)=1$ for all $x\in \tilde{B}_n$ and $\phi_{n,2}(x)=0$ for all $x\in B_n\setminus
\hat{B}_n$.
In particular,
$$
\|\nabla_x\phi_n\|_\infty\le \frac{c_02^n}{\rho} \quad \text{and}
\quad
\|\partial_t\phi_n\|_\infty\le \frac{c_02^{nsp}}{\theta\rho^{sp}}.$$
Let
$$\ell(t):=\tilde{\gamma}\int^t_{-\theta (2\rho)^{sp}
}\int_{{\R}^d\backslash B_{2\rho}}\frac{(u-\mu^-)^{p-1}_-(\tau,y)}{|y|^{d+sp}}\,dy\,d\tau,\quad k_n(t)
:=k_n-\ell(t)$$ and $$w_{n}(t,x)
:=u(t,x)-k_n(t)=u(t,x)+\ell(t)-k_n,$$
where $\tilde{\gamma}\ge1$ is a constant that will be chosen suitably. Taking $
t_2=\bar{t}$ for some $\bar{t}\in(-\theta\tilde{\rho}^{sp}_n,0)$, $\tau_1
=-\theta\tilde{\rho}^{sp}_n$, $h=\theta\rho^{sp}_n-\theta\tilde{\rho}^{sp}_n$, $B_r=B_n$, and $\xi=1$ as well as the functions $\phi(t,x)=\phi_n(t,x)$, $k(t)=k_n(t)$, and $w(t,x)
=w_{n}(t,x)$ as above in Lemma \ref{2} below, we find
\begin{align*}
&\int_{-\theta\tilde{\rho}^{sp}_n}^{\bar{t}}\int_{\tilde{B}_n\times \tilde{B}_n}|w_-(t,x)-w_-(t,y)|^p J(t;x,y)\,dx\,dy\,dt+\int_{\tilde{B}_n}w_-(\bar{t},x)^2\,dx\\
&\le\int_{-\theta\rho^{sp}_n}^{\bar{t}}\int_{B_n\times B_n}|w_-(t,x)-w_-(t,y)|^p\min\{\phi(t,x)^p,\phi(t,y)^p\} J(t;x,y)\,dx\,dy\,dt\\
&\quad+\int_{B_{n}} w_-(t,x)^2\phi(t,x)^p\,dx\bigg|_{-\theta\rho^{sp}_n}^{\bar{t}}\\
&\le C_0\int_{-\theta\rho^{sp}_n}^{\bar{t}}\int_{B_n}w_-(t,x) \phi(t,x)^p\left(\int_{B_n^c}w_-(t,y)^{p-1}J(t;x,y)\,dy\right)\,dx\,dt\\
&\quad+C_0\int_{-\theta\rho^{sp}_n}^{\bar{t}}\int_{B_n\times B_n}\max\{w_-(t,x)^p,w_-(t,y )^p\}|\phi(t,x)-\phi(t,y)|^pJ(t;x,y)\,dx\,dy\,dt\\
&\quad+\int_{-\theta\rho^{sp}_n}^{\bar{t}}\int_{B_n}
w_-(t,x)\phi(t,x)^p(-\ell'(t))\,dx\,dt+C_0\int_{-\theta\rho^{sp}_n}^{\bar{t}}\int_{B_n}
|\partial_t\phi(t,x)^{p}|w_-(t,x)^2\,dx\,dt,
\end{align*} where in the first inequality we used the fact that $\displaystyle\int_{B_{n}} w_-(-\theta\rho^{sp}_n,x)^2\phi(-\theta\rho^{sp}_n,x)^p\,dx=0.$
By taking the supremum
with respect to $\bar{t}\in(-\theta\tilde{\rho}^{sp}_n,0)$ in both sides of the above inequality and recalling $\ell'(t)\ge0$ for all
$t\in(-\theta{\rho}^{sp}_n,0)$, we obtain
\begin{align*}
&\int_{-\theta\tilde{\rho}^{sp}_n}^{0}\int_{\tilde{B}_n\times \tilde{B}_n}|w_-(t,x)-w_-(t,y)|^p J(t;x,y)\,dx\,dy\,dt+\sup_{-\theta\tilde{\rho}^{sp}_n<t<0}\int_{\tilde{B}_n}w_-(t,x)^2\,dx\\
&\le2C_0\int_{-\theta\rho^{sp}_n}^{0}\int_{B_n}w_-(t,x) \phi(t,x)^p\left(\int_{B_n^c}w_-(t,y)^{p-1}J(t;x,y)\,dy\right)\,dx\,dt\\
&\quad+2C_0\int_{-\theta\rho^{sp}_n}^{0}\int_{B_n\times B_n}\max\{w_-(t,x)^p,w_-(t,y )^p\}|\phi(t,x)-\phi(t,y)|^pJ(t;x,y)\,dx\,dy\,dt\\
&\quad+\int_{-\theta\rho^{sp}_n}^{0}\int_{B_n}
w_-(t,x)\phi(t,x)^p(-\ell'(t))\,dx\,dt+2C_0\int_{-\theta\rho^{sp}_n}^{0}\int_{B_n}
|\partial_t\phi(t,x)^{p}|w_-(t,x)^2\,dx\,dt\\
&=:I_1+I_2+I_3+I_4.
\end{align*}

Below, we will estimate each term in the right hand side above. Define $A_n:=\{u(t,x)+\ell(t)<k_n\}\cap\mathcal{Q}^-_n$. For $I_4$,
$$
I_4\le
\frac{c_12^{nsp}}{\theta\rho^{sp}}M^2\int_{-\theta\rho^{sp}_n}^{0}\int_{B_n}\I_{\{u(t,x)+\ell(t)<k_n\}}\,dx\,dt=
\frac{c_12^{nsp}}{\theta\rho^{sp}}M^2|A_n|,
$$
where in the inequality we used the facts that $\|\phi\|_\infty\le 1$, $\|\partial_t\phi\|_\infty\le c_0\frac{2^{nsp}}{\theta\rho^{sp}}$ and for all $(t,x)\in (-\theta\rho^{sp}_n, 0)\times  B_n$,  $$w_-(t,x)=\Big(u(t,x)-\mu^-+\ell(t)-a M-\frac{(1-a) M}{2^{n}}\Big)_-\le  M$$ due to \eqref{mu}.

For $I_3$, recalling the definition of $\ell(t)$, we have
$$
I_3=-\tilde{\gamma}\int_{-\theta\rho^{sp}_n}^{0}\int_{B_n}
w_-(t,x)\phi(t,x)^p\left(\int_{{\R}^d\backslash B_{2\rho}}\frac{(u-\mu^-)^{p-1}_-(t,y)}{|y|^{d+sp}}\,dy\right)\,dx\,dt.
$$

For $I_2$,
\begin{align*}
I_2&\le c_2\int_{-\theta\rho^{sp}_n}^{0}\int_{B_n\times B_n}\max\{w_-(t,x)^p,w_-(t,y )^p\}\frac{|\phi(t,x)-\phi(t,y)|^p}{|x-y|^{d+sp}}\,dx\,dy\,dt\\
&\le c_3 2^{np}\frac{M^p}{\rho^p}\int_{-\theta\rho^{sp}_n}^{0}\int_{B_n\times B_n}\frac{\I_{\{u(t,x)+\ell(t)<k_n\}}}{|x-y|^{d+(s-1)p}}\,dx\,dy\,dt\\
&\quad+c_3 2^{np}\frac{M^p}{\rho^p}\int_{-\theta\rho^{sp}_n}^{0}\int_{B_n\times B_n}\frac{\I_{\{u(t,y)+\ell(t)<k_n\}}}{|x-y|^{d+(s-1)p}}\,dx\,dy\,dt\\
&=2c_3 2^{np}\frac{M^p}{\rho^p}\int_{-\theta\rho^{sp}_n}^{0}\int_{B_n}\I_{\{u(t,x)+\ell(t)<k_n\}}\left(\int_{B_n}\frac{dy}{|x-y|^{d+(s-1)p}}\right)\,dx\,dt\\ &\le c_4 2^{np}\frac{M^p}{\rho^{sp}}|A_n|.
\end{align*}
Here
in the second inequality we used the facts that  for all $(t,x), (t,y)\in (-\theta\rho^{sp}_n, 0)\times  B_n$,
$$\max\{w_-(t,x)^p,w_-(t,y )^p\}\le w_-(t,x)^p+w_-(t,y )^p \le 2M^p$$
and for all $t\in (-\theta\rho^{sp}_n, 0)$ and $x,y\in B_n$,
$$ |\phi(t,x)-\phi(t,y)|^p\le\|\nabla_x\phi\|_\infty^p|x-y|^p\le c_5\left(\frac{2^n}{\rho}\right)^p|x-y|^p,$$
and the equality in the fourth line
follows from the symmetry property.

For $I_1$,
\begin{align*}
I_1&\le2 C_0\int_{-\theta\rho^{sp}_n}^{0}\int_{B_n}w_-(t,x) \phi(t,x)^p\left(\mathop{\ess\sup}_{x\in\hat{B}_n}\int_{{\R}^d\backslash B_n}\frac{w_-(t,y)^{p-1}}{|x-y|^{d+sp}}\,dy\right)\,dx\,dt\\
&=2C_0\int_{-\theta\rho^{sp}_n}^{0}\int_{B_n}w_-(t,x) \phi(t,x)^p\,\mathop{\ess\sup}_{x\in\hat{B}_n}\left(\int_{B_{2\rho}\backslash B_n}\frac{w_-(t,y)^{p-1}}{|x-y|^{d+sp}}\,dy+\int_{\R^d\backslash B_{2\rho}}\frac{w_-(t,y)^{p-1}}{|x-y|^{d+sp}}\,dy\right)\,dx\,dt\\
&\le c_6 2^{n(d+sp)}\int_{-\theta\rho^{sp}_n}^{0}\int_{B_n}w_-(t,x) \phi(t,x)^p\left(\int_{B_{2\rho}\backslash B_n}\frac{w_-(t,y)^{p-1}}{|y|^{d+sp}}\,dy\right)\,dx\,dt\\
&\quad+ c_6\int_{-\theta\rho^{sp}_n}^{0}\int_{B_n}w_-(t,x) \phi(t,x)^p\left(\int_{{\R}^d\backslash B_{2\rho}}\frac{w_-(t,y)^{p-1}}{|y|^{d+sp}}\,dy\right)\,dx\,dt\\
&\le c_7 2^{n(d+sp)}\frac{M^{p-1}}{\rho^{sp}}\int_{-\theta\rho^{sp}_n}^{0}\int_{B_n}w_-(t,x) \phi(t,x)^p\,dx\,dt\\
&\quad+c_7\int_{-\theta\rho^{sp}_n}^{0}\int_{B_n}w_-(t,x) \phi(t,x)^p\left(\frac{M^{p-1}}{\rho^{sp}}+\int_{{\R}^d\backslash B_{2\rho}}\frac{(u-\mu^-)^{p-1}_-(t,y)}{|y|^{d+sp}}\,dy\right)\,dx\,dt\\
&\le c_8 2^{n(d+sp)}\frac{M^p}{\rho^{sp}}|A_n|+c_7\int_{-\theta\rho^{sp}_n}^{0}\int_{B_n}w_-(t,x) \phi(t,x)^p\left(\int_{{\R}^d\backslash B_{2\rho}}\frac{(u-\mu^-)^{p-1}_-(t,y)}{|y|^{d+sp}}\,dy\right)\,dx\,dt.
\end{align*}
Here
in the second inequality we used the facts that, for all $x\in\hat{B}_n$ and $y\in B_{2\rho}\backslash B_n$,
$$\frac{|y|}{|x-y|}\le1+\frac{|x|}{|x-y|}\le1+\frac{\hat{\rho}_n}{\rho_n-\hat{\rho}_n}=\frac{4\rho_n}{\rho_n-\rho_{n+1}}\le2^{4+n},$$
and, for all $x\in\hat{B}_n$ and $y\in\R^d\backslash B_
{2\rho}$,
due to $\rho\in(0,R/2]$,
$$\frac{|y|}{|x-y|}\le1+\frac{|x|}{|x-y|}\le1+\frac{\hat{\rho}_n}{2\rho-\hat{\rho}_n}\le1+\frac{\rho}{2\rho-\rho}=
2;$$
while for the first term on the right-hand side of the
third inequality we used the fact
that
for all $(t,y)\in (-\theta\rho_n^{sp},0)\times
B_{2\rho}\backslash B_n$, $w_-(t,y)^{p-1}\le
M^{p-1}$, and  for the second term we used the fact that for all
$(t,y)\in (-\theta\rho_n^{sp},0)\times {\R}^d\backslash B_{2\rho}$,
\begin{align*}
w_-(t,y)^{p-1}&=\Big[\Big(u(t,y)-\mu^-+\ell(t)-a  M -\frac{(1-a) M}{2^{n}}\Big)_-\Big]^{p-1}\\
&\le[(u-\mu^-)_-(t,y)+ M]^{p-1}\le 2^p\left((u-\mu^-)^{p-1}_-(t,y)+M^{p-1}\right).
\end{align*}
The last term in the estimate above will be cancelled with $I_3$, if we choose $ \tilde{\gamma}=c_7$.

Putting all the estimates above together, we arrive at
\begin{equation}\label{EN1}\begin{split}
&\int_{-\theta\tilde{\rho}^{sp}_n}^{0}\int_{\tilde{B}_n\times \tilde{B}_n}|w_-(t,x)-w_-(t,y)|^p J(t;x,y)\,dx\,dy\,dt+\sup_{-\theta\tilde{\rho}^{sp}_n<t<0}\int_{\tilde{B}_n}w_-(t,x)^2\,dx\\
&\le c_9 2^{n(d+sp)}\frac{M^p}{\rho^{sp}}\Big(1+{\theta}^{-1} M^{2-p}\Big)|A_n|.
\end{split}\end{equation}

(2) Now, take
$0\le\varphi_n\le1$ to be a cutoff function in $\tilde{\mathcal{Q}}^-_{n+1}\subset\bar{\mathcal{Q}}^-_n$ such that it satisfies $
\|\nabla_x\varphi_n\|_\infty\le c_{10}2^n/\rho$. An application of the fractional Sobolev inequality (see Lemma \ref{Sob}) with $\kappa:=1+2s/d$ gives us that
\begin{align*}
&\frac{(1-a)M}{2^{n+1}}|A_{n+1}|=\frac{(1-a) M}{2^{n+1}}\iint_{{\mathcal{Q}}^-_{n+1}}\I_{\{u(t,x)+\ell(t)<k_{n+1}\}}\,dx\,dt\\
&\le\iint_{\tilde{\mathcal{Q}}^-_n}(w_-
\varphi_n)(t,x)\,dx\,dt\le\left[\iint_{\tilde{\mathcal{Q}}^-_n}(w_-
\varphi_n)^{\kappa p}(t,x)\,dx\,dt\right]^{1/(\kappa p)}|A_n|^{1-1/(\kappa p)}\\
&\le c_{11}\bigg[\tilde{\rho}_n^{sp}\int_{-\theta\tilde{\rho}^{sp}_n}^{0}\int_{\tilde{B}_n\times \tilde{B}_n}|(w_-
\varphi_n)(t,x)-(w_-
\varphi_n)(t,y)|^p J(t;x,y)\,dx\,dy\,dt+\iint_{\tilde{\mathcal{Q}}^-_n}(w_-
\varphi_n)^p(t,x)\,dx\,dt\bigg]^{1/(\kappa p)}\\
&\quad \times \left[\sup_{-\theta\tilde{\rho}^{sp}_n<t<0}\oint_{\tilde{B}_n}(w_-
\varphi_n)^2(t,x)\,dx\right]^{s/(d\kappa)}|A_n|^{1-1/(\kappa p)}\\
&=:c_{11}(J_1+J_2)^{1/(\kappa p)}\times J_3^{s/(d\kappa)}|A_n|^{1-1/(\kappa p)},
\end{align*}
where in the first inequality we used the facts that $
\varphi_n=1$ in ${\mathcal{Q}}^-_{n+1}$, ${\mathcal{Q}}^-_{n+1}\subset \tilde{\mathcal{Q}}^-_n$ and $w_-(t,x)=(u(t,x)+\ell(t)-k_n)_-\ge(k_{n+1}-k_n)_-=\frac{(1-a) M}{2^{n+1}}$ on the set $\{u(t,x)+\ell(t)<k_{n+1}\}$, and in the second inequality we used the H\"{o}lder inequality and the fact that $\tilde{\mathcal{Q}}^-_n\subset{\mathcal{Q}}^-_n$.

From the Young inequality, we have
$$|(w_-
\varphi_n)(t,x)-(w_-
\varphi_n)(t,y)|^p\le 2^p|w_-(t,x)-w_-(t,y)|^p
\varphi_n(t,x)^p+2^pw_-(t,y)^p|
\varphi_n(t,x)-
\varphi_n(t,y)|^p.$$
Thus,
\begin{align*}
J_1&\le c_{12}\rho^{sp}\int_{-\theta\tilde{\rho}^{sp}_n}^{0}\int_{\tilde{B}_n\times \tilde{B}_n}\frac{|w_-(t,x)-w_-(t,y)|^p
\varphi_n(t,x)^p}{|x-y|^{d+sp}}\,dx\,dy\,dt\\
&\quad+c_{12}\rho^{sp}\int_{-\theta\tilde{\rho}^{sp}_n}^{0}\int_{\tilde{B}_n\times \tilde{B}_n}\frac{w_-(t,y)^p|
\varphi_n(t,x)-
\varphi_n(t,y)|^p}{|x-y|^{d+sp}}\,dx\,dy\,dt\\
&\le c_{13}\rho^{sp}\int_{-\theta\tilde{\rho}^{sp}_n}^{0}\int_{\tilde{B}_n\times \tilde{B}_n}\frac{|w_-(t,x)-w_-(t,y)|^p}{|x-y|^{d+sp}}\,dx\,dy\,dt\\
&\quad+c_{13}\rho^{sp}2^{np}\frac{M^p}{\rho^p}\int_{-\theta\rho^{sp}_n}^{0}\int_{B_n\times B_n}\frac{\I_{\{u(t,y)+\ell(t)<k_n\}}}{|x-y|^{d+(s-1)p}}\,dx\,dy\,dt\\
&\le c_{14}2^{n(d+sp)}M^p\Big(1+{\theta}^{-1}M^{2-p}\Big)|A_n|+
c_{14}2^{np}M^p|A_n|\\
&\le 2c_{14}2^{n(d+sp)}M^p\Big(1+{\theta}^{-1}M^{2-p}\Big)|A_n|,
\end{align*}
where in the second inequality we used the facts that for all $(t,y)\in (-\theta
\tilde{\rho}_n^{sp},0)\times \tilde B_n$, $w_-(t,y)^{p}\le M^{p}$ and for all $t\in(-\theta
\tilde{\rho}_n^{sp},0)$ and $x,y\in
\tilde{B}_n$,  $$|
\varphi_n(t,x)-
\varphi_n(t,y)|^p\le\|\nabla_x
\varphi_n\|_\infty^p|x-y|^p\le c_{15}\left(\frac{2^n}{\rho}\right)^p|x-y|^p,$$ and the third inequality follows from \eqref{EN1}.

It is obvious that
$$J_2\le\iint_{{\mathcal{Q}}^-_n}(w_-
\varphi_n)^p(t,x)\,dx\,dt\le M^p|A_n|.$$
Furthermore, by \eqref{EN1} again, we get
$$J_3\le|\tilde{B}_n|^{-1}\sup_{-\theta\tilde{\rho}^{sp}_n<t<0}\int_{\tilde{B}_n}w_-(t,x)^2\,dx\le c_{16}2^{n(d+sp)}\frac{M^p}{\rho^{d+sp}}\Big(1+{\theta}^{-1}M^{2-p}\Big)|A_n|.$$

Putting all the inequalities together, we find that
\begin{equation}\label{A_n+1-}\begin{split}
\frac{(1-a) M}{2^{n+1}}|A_{n+1}|&\le c_{11}(J_1+J_2)^{1/(\kappa p)}  J_3^{s/(d\kappa)}|A_n|^{1-1/(\kappa p)}\\
&\le c_{11}\bigg[2c_{14}2^{n(d+sp)} M^p\Big(1+{\theta}^{-1}  M ^{2-p}\Big)|A_n|+   M ^p|A_n|\bigg]^{1/(\kappa p)}\\
&\qquad\times\left[c_{16}2^{n(d+sp)}\frac{ M^p}{\rho^{d+sp}}\Big(1+{\theta}^{-1} M^{2-p}\Big)|A_n|\right]^{s/(d\kappa)}|A_n|^{1-1/(\kappa p)}\\
&\le c_{17}\bigg[2^{n(d+sp)} M ^p\Big(1+{\theta}^{-1} M^{2-p}\Big)|A_n|\bigg]^{1/(\kappa p)}\\
&\qquad\times \left[2^{n(d+sp)}\frac{ M ^p}{\rho^{d+sp}}\Big(1+{\theta}^{-1}  M ^{2-p}\Big)|A_n|\right]^{s/(d\kappa)}|A_n|^{1-1/(\kappa p)}\\
&=c_{17}b^n\Big(1+{\theta}^{-1}  M ^{2-p}\Big)^{\frac{1}{\kappa p}+\frac{s}{d\kappa}}\theta^{\frac{s}{d\kappa}}\theta^{-\frac{s}{d\kappa}}\rho^{-(d+sp)\frac{s}{d\kappa}}M^{\frac{d+sp}{d\kappa}}|A_n|^{1+\frac{s}{d\kappa}},
\end{split}\end{equation}
where $b
:=2^{(d+sp)(\frac{1}{\kappa p}+\frac{s}{d\kappa})}$.

(3) Denote $Y_n:=\frac{|A_n|}{|{\mathcal{Q}}^-_n|}$ for all $n\ge1$.  Dividing \eqref{A_n+1-} by $|{\mathcal{Q}}^-_{n+1}|$ and noting $|{\mathcal{Q}}^-_n|
\le c_{18}
|{\mathcal{Q}}^-_{n+1}|$, we get
$$Y_{n+1}\le c_{19}(1-a)^{-1}\Big(1+{\theta}^{-1}M^{2-p}\Big)^{\frac{1}{\kappa p}+\frac{s}{d\kappa}}\theta^{\frac{s}{d\kappa}}M^{\frac{d+sp}{d\kappa}-1}(2b)^nY_n^{1+\frac{s}{d\kappa}}.$$
By Lemma \ref{iteration}, we see that $Y_\infty=0$ provided
\begin{align*}
&Y_0\le\left(c_{19}(1-a)^{-1}\Big(1+{\theta}^{-1}M^{2-p}\Big)^{\frac{1}{\kappa p}+\frac{s}{d\kappa}}\theta^{\frac{s}{d\kappa}}M^{\frac{d+sp}{d\kappa}-1}\right)^{-d\kappa/s}(2b)^{-(d\kappa/s)^2}\\
&\quad=c_{19}^{-(d+2s)/s}(1-a)^{
(d+2s)/s}\theta^{-1}M^{2-p}\Big(1+{\theta}^{-1}M^{2-p}\Big)^{-\frac{d+sp}{sp}}(2b)^{-(d+2s)^2/s^2}=:\nu.
\end{align*}
This is, if
$$Y_0=|A_0|/|{\mathcal{Q}}^-_0|=|\{u(t,x)-\mu^-+\ell(t)\le M\}\cap\mathcal{Q}^-_{2\rho}(\theta)|/|\mathcal{Q}^-_{2\rho}(\theta)|\le\nu,$$
then
$Y_\infty=0$, i.e.,
$$u(t,x)-\mu^-+\ell(t)\ge a M
\quad a.e.\,\,\hbox{in}\,\,\mathcal{Q}^-_{\rho}(\theta).$$
In particular, this implies that, if
$$\ell(0)\le (1-\lambda)a M \quad \left(\text{which is equivalent to }\,\, \tilde{\gamma}{\rm \tilde{Tail}}^{p-1}_{p-1}
((u-\mu^-)_-; 0,2\rho, -\theta(2\rho)^{sp},0)\le (1-\lambda)aM\right)$$ for some $\lambda\in (0,1)$,
and
$$|\{u-\mu^-\le M\}\cap\mathcal{Q}^-_{2\rho}(\theta)|\le\nu|\mathcal{Q}^-_{2\rho}(\theta)|,$$
then
$$u-\mu^-\ge a \lambda M
\quad a.e.\,\,\hbox{in}\,\,\mathcal{Q}^-_{ \rho}(\theta).$$
The proof is concluded.
\end{proof}

We next prove the De Giorgi-type lemma with a forward in time version.

\begin{lemma}\label{Lem2}
Let $u$ be a locally bounded weak super-solution to \eqref{EQ}
in $I\times D$. Suppose that $\mathcal{Q}^+_{2\rho}(\theta;(t_0,x_0))\subset\mathcal{Q}$ with $\rho\in (0,R/2]$ and $\theta>0$. Then, there is a constant $C_2:=C_2(s,p,d)>0$ so that, for any $M>0$ and $a\in(0,1)$ satisfying
$$
\mu^-\le\mathop{\ess\inf}_{\mathcal{Q}^+_{2\rho}(\theta;(t_0,x_0))} u,
$$
$$u(t_0,\cdot)\ge\mu^-+M \quad a.e.\,\,in\,\,B_{2\rho}(x_0)$$
and
$$|\{u\le\mu^-+  M\}\cap \mathcal{Q}^+_{2\rho}(\theta;(t_0,x_0))|\le\tilde{\nu}|\mathcal{Q}^+_{2\rho}(\theta;(t_0,x_0))|$$ with $\tilde{\nu}:=C_2(1-a)^{\frac{2s+d}{s}}\theta^{-1} M^{2-p}$, it holds that for all $\lambda\in (0,1)$ either
$$\tilde{\gamma}{\rm \tilde{Tail}}^{p-1}_{p-1}
((u-\mu^-)_-; x_0,2\rho, t_0, t_0+\theta(2\rho)^{sp})>(1-\lambda)a M,$$
or
$$u\ge\mu^-+a\lambda M \quad a.e.\,\,in\,\,(t_0, t_0+\theta(2\rho)^{sp}]\times B_{\rho}(x_0),$$
where $\tilde{\gamma}:=\tilde{\gamma}(s,p,d)\ge1$.
\end{lemma}

\begin{proof}
Without loss of generality, we may and do assume that $(t_0,x_0)=(0,0)$ and $\mu^-=0$. The proof is split into three steps.

(1) We introduce
$k_n$, $\rho_n$, $\hat{\rho}_n$, $\tilde{\rho}_n$, $\bar{\rho}_n$, $B_n$, $\hat{B}_n$, $\tilde{B}_n$ and $\bar{B}_n$ as these in Lemma \ref{Lem1} with $\mu^-=0$. The only difference is that the cylinders $\mathcal{Q}^+_n$, $\hat{\mathcal{Q}}^+_n$, $\tilde{\mathcal{Q}}^+_n$ and $\bar{\mathcal{Q}}^+_n$ are now of forward type and the height of the cylinders is fixed, i.e.,
$$\mathcal{Q}^+_n=(0,\theta(2\rho)^{sp})\times B_n,\quad \hat{\mathcal{Q}}^+_n=(0,\theta(2\rho)^{sp})\times\hat{B}_n,\quad \tilde{\mathcal{Q}}^+_n=(0,\theta(2\rho)^{sp})\times\tilde{B}_n,\quad \bar{\mathcal{Q}}^+_n=(0,\theta(2\rho)^{sp})\times\bar{B}_n.$$

Let
$\phi_n$ be a time-independent smooth cut-off function in $\tilde{B}_n\subset\hat{B}_n$, which satisfies $\|\nabla
\phi_n\|_\infty\le c_02^n/\rho$.
We now apply
Lemma \ref{2} on $\mathcal{Q}^+_n$ with $
t_2=\bar{t}$ for some $\bar{t}\in(0,\theta(2\rho)^{sp})$, $\tau_1-h=0$, $B_r=B_n$, $\xi=1$, $\phi(t,x)=\phi_n(x)
$, and with the function $w$
as the same form as in Lemma \ref{Lem1} with
$$\ell(t)
:=\tilde \gamma\int_0^t\int_{\R^d\setminus B_{2\rho}}\frac{(u-\mu^-)_-^{p-1}(\tau,y)}{|y|^{d+sp}}\,dy\,d\tau.$$ Since $\tilde{B}_n\subset B_n$ and $\phi=1$ in $\tilde{B}_n$, we obtain
\begin{align*}
&\int_{0}^{\bar{t}}\int_{\tilde{B}_n\times \tilde{B}_n}|w_-(t,x)-w_-(t,y)|^p J(t;x,y)\,dx\,dy\,dt+\int_{\tilde{B}_n}w_-(\bar{t},x)^2\,dx\\
&\le\int_{0}^{\bar{t}}\int_{B_n\times B_n}|w_-(t,x)-w_-(t,y)|^p\min\{\phi(x)^p,\phi(y)^p\} J(t;x,y)\,dx\,dy\,dt\\
&\quad +\int_{B_{n}} w_-(t,x)^2\phi(x)^p\,dx\bigg|^{\bar{t}}_0\\
&\le C_0\int_{0}^{\bar{t}}\int_{B_n}w_-(t,x) \phi(x)^p\left(\int_{B_n^c}w_-(t,y)^{p-1}J(t;x,y)\,dy\right)\,dx\,dt\\
&\quad+C_0\int_{0}^{\bar{t}}\int_{B_n\times B_n}\max\{w_-(t,x)^p,w_-(t,y )^p\}|\phi(x)-\phi(y)|^pJ(t;x,y)\,dx\,dy\,dt\\
&\quad+\int_{0}^{\bar{t}}\int_{B_n}
w_-(t,x)\phi(x)^p(-\ell'(t))\,dx\,dt,
\end{align*}
where in the first inequality we used the facts that $u(0,\cdot)\ge
M\,\,\text{a.e.}\,\,\text{in}\,\,B_{2\rho}(0)$ and $k(0)=-\ell(0)+k_n= k_n\le
M$ which imply
$$\int_{B_n}w_-(0,x)^2\phi(x)^p\,dx\le\int_{B_{2\rho}}(u(0,x)-k(0))_-^2\,dx=0.$$
By taking the supremum
with respect to $\bar{t}\in(0,\theta(2\rho)^{sp})$ in both sides of the above inequality and recalling $\ell'(t)\ge0$ for all $t\in(0,\theta(2\rho)^{sp})$, we obtain
\begin{equation}\label{EE}\begin{split}
&\int_{0}^{\theta(2\rho)^{sp}}\int_{\tilde{B}_n\times \tilde{B}_n}|w_-(t,x)-w_-(t,y)|^p J(t;x,y)\,dx\,dy\,dt+\sup_{0<t<\theta(2\rho)^{sp}}\int_{\tilde{B}_n}w_-(t,x)^2\,dx\\
&\le2C_0\int_{0}^{\theta(2\rho)^{sp}}\int_{B_n}w_-(t,x) \phi(x)^p\left(\int_{B_n^c}w_-(t,y)^{p-1}J(t;x,y)\,dy\right)\,dx\,dt\\
&\quad+2C_0\int_{0}^{\theta(2\rho)^{sp}}\int_{B_n\times B_n}\max\{w_-(t,x)^p,w_-(t,y )^p\}|\phi(x)-\phi(y)|^pJ(t;x,y)\,dx\,dy\,dt\\
&\quad+\int_{0}^{\theta(2\rho)^{sp}}\int_{B_n}
w_-(t,x)\phi(x)^p(-\ell'(t))\,dx\,dt.
\end{split}\end{equation}
Furthermore, the right hand of \eqref{EE} can be treated similarly as that in Lemma \ref{Lem1}. Finally, following the arguments for \eqref{EN1}, we obtain that
\begin{equation}\label{EN2}\begin{split}
&\int_{0}^{\theta(2\rho)^{sp}}\int_{\tilde{B}_n\times \tilde{B}_n}|w_-(t,x)-w_-(t,y)|^p J(t;x,y)\,dx\,dy\,dt+\sup_{0<t<\theta(2\rho)^{sp}}\int_{\tilde{B}_n}w_-(t,x)^2\,dx\\
&\le c_1 2^{n(d+sp)}\frac{M^p}{\rho^{sp}}|A_n|,
\end{split}\end{equation}
where $A_n:=\{u(t,x)+\ell(t)<  k_n\}\cap\mathcal{Q}^+_n$.

(2) Now, take $0\le\varphi_n\le1$ to be a cutoff function in $B_{n+1}\subset \bar{B}_n$, which satisfies $
\|\nabla\varphi_n\|_\infty\le c_02^n/\rho$. An application of the fractional Sobolev inequality (see Lemma \ref{Sob}) with $\kappa:=1+2s/d$ gives us that
\begin{align*}
&\frac{(1-a)M}{2^{n+1}}|A_{n+1}| =\frac{(1-a) M}{2^{n+1}}\iint_{\mathcal{Q}^+_{n+1}}\I_{\{u(t,x)+\ell(t)<k_{n+1}\}}\,dx\,dt\\
&\le\iint_{\tilde{\mathcal{Q}}^+_n}w_-(t,x)\varphi_n(x)\,dx\,dt\le\left[\iint_{\tilde{\mathcal{Q}}^+_n}(w_-(t,x)\varphi_n(x))^{\kappa p}\,dx\,dt\right]^{1/(\kappa p)}|A_n|^{1-1/(\kappa p)}\\
&\le c_2\bigg[\tilde{\rho}_n^{sp}\int_{0}^{\theta(2\rho)^{sp}}\int_{\tilde{B}_n\times \tilde{B}_n}|w_-(t,x)\varphi_n(x)-w_-(t,y)\varphi_n(y)|^p J(t;x,y)\,dx\,dy\,dt+\iint_{\tilde{\mathcal{Q}}^+_n}(w_-(t,x)\varphi_n(x))^p\,dx\,dt\bigg]^{1/(\kappa p)}\\
&\quad \times \left[\sup_{0<t<\theta(2\rho)^{sp}}\oint_{\tilde{B}_n}(w_-(t,x)\varphi_n(x))^2\,dx\right]^{s/(d\kappa)}|A_n|^{1-1/(\kappa p)}\\
&=:c_2(J_1+J_2)^{1/(\kappa p)} J_3^{s/(d\kappa)}|A_n|^{1-1/(\kappa p)},
\end{align*}
where in the first inequality we used the facts that $\varphi_n=1$ in $B_{n+1}$, and $w_-(t,x)=(u(t,x)+\ell(t)-k_n)_-\ge(k_{n+1}-k_n)_-=\frac{(1-a)M}{2^{n+1}}$ on the set $\{u(t,x)+\ell(t)<k_{n+1}\}$, and in the second inequality we used the H\"{o}lder inequality and the fact that $\tilde{\mathcal{Q}}^+_n\subset{\mathcal{Q}}^+_n$.

According to the Young inequality,
$$|w_-(t,x)\varphi_n(x)-w_-(t,y)\varphi_n(y)|^p\le2^p|w_-(t,x)-w_-(t,y)|^p
\varphi_n(x)^p+2^pw_-(t,y)^p|\varphi_n(x)-\varphi_n(y)|^p.$$
Thus,
\begin{align*}
J_1&\le c_3\rho^{sp}\int_{0}^{\theta(2\rho)^{sp}}\int_{\tilde{B}_n\times \tilde{B}_n}\frac{|w_-(t,x)-w_-(t,y)|^p\varphi_n(x)^p}{|x-y|^{d+sp}}\,dx\,dy\,dt\\
&\quad+c_3\rho^{sp}\int_{0}^{\theta(2\rho)^{sp}}\int_{\tilde{B}_n\times \tilde{B}_n}\frac{w_-(t,y)^p|\varphi_n(x)-\varphi_n(y)|^p}{|x-y|^{d+sp}}\,dx\,dy\,dt\\
&\le c_4\rho^{sp}\int_{0}^{\theta(2\rho)^{sp}}\int_{\tilde{B}_n\times \tilde{B}_n}\frac{|w_-(t,x)-w_-(t,y)|^p}{|x-y|^{d+sp}}\,dx\,dy\,dt\\
&\quad+c_4\rho^{sp}2^{np}\frac{M^p}{\rho^p}\int_{0}^{\theta(2\rho)^{sp}}\int_{B_n\times B_n}\frac{\I_{\{u(t,y)+\ell(t)<k_n\}}}{|x-y|^{d+(s-1)p}}\,dx\,dy\,dt\\
&\le c_52^{n(d+sp)}M^p|A_n|+
c_52^{np}M^p|A_n|\\
&\le2c_52^{n(d+sp)}M^p|A_n|,
\end{align*}
where in the second inequality we used the facts that for all $(t,y)\in (0,\theta (2\rho)^{sp})\times \tilde B_n$, $w_-(t,y)\le M$, and for all $t\in (0,\theta (2\rho)^{sp})$ and $x,y\in \tilde B_n$,
$$|\varphi_n(x)-\varphi_n(y)|^p\le\|\nabla\varphi_n\|_\infty^p|x-y|^p\le c_6\left(\frac{2^n}{\rho}\right)^p|x-y|^p,$$
and the third inequality follows from \eqref{EN2}.

It is clear that
$$J_2\le\iint_{{\mathcal{Q}}^+_n}(w_-(t,x)\varphi_n(x))^p\,dx\,dt\le M^p|A_n|.$$
Furthermore, by \eqref{EN2} again, we get
$$J_3\le |\tilde{B}_n|^{-1}\sup_{0<t<\theta(2\rho)^{sp}}\int_{\tilde{B}_n}w_-(t,x)^2\,dx\le c_72^{n(d+sp)}\frac{M^p}{\rho^{d+sp}}|A_n|.$$

Putting all the inequalities together, we find that
\begin{equation}\label{A_n+1}\begin{split}
&\frac{(1-a) M}{2^{n+1}}|A_{n+1}|\le c_2(J_1+J_2)^{1/(\kappa p)} J_3^{s/(d\kappa)}|A_n|^{1-1/(\kappa p)}\\
&\le c_2\bigg[2c_52^{n(d+sp)}M^p|A_n|+ M^p|A_n|\bigg]^{1/(\kappa p)} \left[c_72^{n(d+sp)}\frac{M^p}{\rho^{d+sp}}|A_n|\right]^{s/(d\kappa)}|A_n|^{1-1/(\kappa p)}\\
&\le c_8\bigg[2^{n(d+sp)}M^p|A_n|\bigg]^{1/(\kappa p)}\times \left[2^{n(d+sp)}\frac{M^p}{\rho^{d+sp}}|A_n|\right]^{s/(d\kappa)}|A_n|^{1-1/(\kappa p)}\\
&=c_8b^n\rho^{-(d+sp)\frac{s}{d\kappa}}M^{\frac{d+sp}{d\kappa}}|A_n|^{1+\frac{s}{d\kappa}},
\end{split}\end{equation}
where $b
:=2^{(d+sp)(\frac{1}{\kappa p}+\frac{s}{d\kappa})}$.

(3) Dividing \eqref{A_n+1} by $|{\mathcal{Q}}^+_{n+1}|$ and noting $|{\mathcal{Q}}^+_n|
\le c_9
|{\mathcal{Q}}^+_{n+1}|$, we get $$Y_{n+1}\le c_{10}(1-a)^{-1}M^{\frac{d+sp}{d\kappa}-1}\theta^{\frac{s}{d\kappa}} (2b)^nY_n^{1+\frac{s}{d\kappa}},$$ where  $Y_n:=\frac{|A_n|}{|{\mathcal{Q}}^+_n|}$.
By Lemma \ref{iteration}, we see that $Y_\infty=0$ provided
$$
Y_0\le\left(c_{10}
(1-a)^{-1}M^{\frac{d+sp}{d\kappa}-1}\right)^{-d\kappa/s}\theta ^{-1} (2b)^{-(d\kappa/s)^2}
=c_{10}^{-(d+2s)/s}(1-a)^{
(d+2s)/s}\theta^{-1} M^{2-p}(2b)^{-(d+2s)^2/s^2}=:\tilde{\nu}.
$$
Hence, if
$$Y_0=|A_0|/|{\mathcal{Q}}^+_0|=|\{u(t,x)+\ell(t)\le M\}\cap\mathcal{Q}^+_{2\rho}
(\theta)|/|\mathcal{Q}^+_{2\rho}(\theta)|\le\tilde{\nu},$$
then $Y_\infty=0$, that is,
$$u(t,x)+\ell(t)\ge a  M\quad a.e.\,\,
\hbox{in}\,\, (0,\theta(2\rho)^{sp}]\times B_{\rho}.$$
Note that, for some $\lambda\in (0,1)$,
$\ell(\theta(2\rho)^{sp})\le (1-\lambda)a  M$ holds, if and only if $$ \tilde{\gamma}{\rm \tilde{Tail}}^{p-1}_{p-1}
(u_-; 0,2\rho,0,\theta(2\rho)^{sp})\le (1-\lambda)a  M.$$
The proof is completed
by using the conclusion above.
\end{proof}

The following lemma propagates measure information forward in time, which is a general expansion of positivity estimate.

\begin{lemma}\label{Lem3}
Let $u$ be a locally bounded weak super-solution to \eqref{EQ}
in $I\times D$. Let $\alpha\in (0,1)$. Then, there exist $\delta:=\delta(s,p,d,\alpha)$, $\varepsilon:=\varepsilon(s,p,d,\alpha)\in(0,1)$, such that for any $M, \rho>0$ satisfying
$$
\mathop{\ess\inf}_{(t,x)\in (t_0,t_0+\delta M^{2-p}\rho^{sp}]\times B_\rho(x_0)}u(t,x)\ge \mu^-$$ and
$$|\{u(t_0,\cdot)\ge\mu^-+M\}\cap B_\rho(x_0)|\ge\alpha|B_\rho(x_0)|,$$
 it holds  that either
$$
{\rm \tilde{Tail}}^{p-1}_{p-1}
((u-\mu^-)_-; x_0,\rho,t_0,t_0+\delta M^{2-p}\rho^{sp})>\delta M,$$
or
$$|\{u(t,\cdot)\ge\mu^-+\varepsilon M\}\cap B_\rho(x_0)|\ge\frac{\alpha}{2}|B_\rho(x_0)|
\quad \text{for all}\,\,
t\in(t_0,t_0+\delta M^{2-p}\rho^{sp}],$$
provided that the cylinder $(t_0,t_0+\delta M^{2-p}\rho^{sp}]\times B_\rho(x_0)$ is included in $\mathcal{Q}$.
\end{lemma}

\begin{proof}
It suffices to show the case that $(t_0,x_0)=(0,0)$ and $\mu^-=0$. Throughout the proof, we assume that
$$
{\rm \tilde{Tail}}^{p-1}_{p-1}(u_-; 0,\rho,0,\delta M^{2-p}\rho^{sp})\le\delta M.$$
Choose a time-independent cutoff function $\phi(x)$ in $B_{(1-\sigma)\rho}\subset B_{(1-\sigma/2)\rho}$ with $\sigma\in(0,1)$, which satisfies $\|\nabla\phi\|_\infty\le4(\sigma\rho)^{-1}$. Use Lemma \ref{2} in the cylinder $(0,\delta M^{2-p}\rho^{sp}]\times B_\rho$, with $
t_2=\bar{t}$ for some $\bar{t}\in(0,\delta M^{2-p}\rho^{sp})$, $\tau_1-h=0$, $B_r=B_\rho$, $\xi=1$, $\phi(t,x)=\phi(x)$, and with the truncation
$$w_-:=(u-M)_-.$$
Then, for all $0<\bar{t}\le\delta M^{2-p}\rho^{sp}$,
\begin{align*}
&\int_{B_{(1-\sigma)\rho}}w_-(\bar{t},x)^2\,dx-\int_{B_\rho}w_-(0,x)^2\,dx\le\int_{B_{\rho}}w_-(t,x)^2\phi(x)^p\,dx\bigg|_{0}^{\bar{t}}\\
&\le C_0\int_{0}^{\bar{t}}\int_{B_\rho}w_-(t,x) \phi(x)^p\left(\int_{\R^d\backslash B_\rho}w_-(t,y)^{p-1}J(t;x,y)\,dy\right)\,dx\,dt\\
&\quad+C_0\int_{0}^{\bar{t}}\int_{B_\rho}\int_{ B_\rho}\max\{w_-(t,x)^p,w_-(t,y )^p\}|\phi(x)-\phi(y)|^pJ(t;x,y)\,dx\,dy\,dt\\
&\le C_0\int_{0}^{\delta M^{2-p}\rho^{sp}}\int_{B_\rho}w_-(t,x) \phi(x)^p\left(\int_{\R^d\backslash B_\rho}w_-(t,y)^{p-1}J(t;x,y)\,dy\right)\,dx\,dt\\
&\quad+C_0\int_{0}^{\delta M^{2-p}\rho^{sp}}\int_{B_\rho}\int_{ B_\rho}\max\{w_-(t,x)^p,w_-(t,y )^p\}|\phi(x)-\phi(y)|^pJ(t;x,y)\,dx\,dy\,dt\\
&=:I_1+I_2.
\end{align*}

For $I_1$, note that for $x\in\supp[\phi]\subseteq B_{(1-\sigma/2)\rho}$ and $y\in\R^d\backslash B_\rho$,
$$\frac{|y|}{|x-y|}\le1+\frac{|x|}{|x-y|}\le1+\frac{(1-\sigma/2)\rho}{\sigma\rho/2}=
\frac{2}{\sigma}.$$
Then,
\begin{align*}
I_1&\le c_1\int_{0}^{\delta M^{2-p}\rho^{sp}}\int_{B_\rho}w_-(t,x) \phi(x)^p\left(\mathop{\ess\sup}_{x\in
\supp[\phi]}\int_{\R^d\backslash B_\rho}\frac{w_-(t,y)^{p-1}}{|x-y|^{d+sp}}\,dy\right)\,dx\,dt\\
&\le \frac{c_2M|B_\rho|}{\sigma^{d+sp}}\int_{0}^{\delta M^{2-p}\rho^{sp}}\int_{\R^d\backslash B_\rho}\frac{w_-(t,y)^{p-1}}{|y|^{d+sp}}\,dy\,dt\\
&\le \frac{c_3M|B_\rho|}{\sigma^{d+sp}}\left(\delta M+\int_{0}^{\delta M^{2-p}\rho^{sp}}\int_{\R^d\backslash B_\rho}\frac{u_-(t,y)^{p-1}}{|y|^{d+sp}}\,dy\,dt\right)\\
&\le \frac{c_4\delta M^2}{\sigma^{d+sp}}|B_\rho|,
\end{align*}
where the second inequality follows from $w_-(t,x)\le M$ for all $t\in (0, \delta M^{2-p}\rho^{sp})$ and $x\in B_\rho$, in the third inequality we used the fact that $w_-(t,y)\le u_-(t,y)+M$ for all $t\in (0, \delta M^{2-p}\rho^{sp})$ and $y\in \R^d\backslash B_\rho$, and
in the last inequality we used the fact that
${\rm \tilde{Tail}}^{p-1}_{p-1}
(u_-; 0,\rho,0,\delta M^{2-p}\rho^{sp})\le\delta M.$
For $I_2$,
\begin{align*}
I_2&\le c_5(\delta M^{2-p}\rho^{sp})M^p\int_{B_\rho}\int_{ B_\rho}\frac{|\phi(x)-\phi(y)|^p}{|x-y|^{d+sp}}\,dx\,dy\\
&\le c_6(\delta M^{2-p}\rho^{sp})\frac{M^p}{(\sigma\rho)^p}\int_{B_\rho}\int_{ B_\rho}\frac{1}{|x-y|^{d+(s-1)p}}\,dx\,dy\\
&\le \frac{ c_7\delta M^2}{\sigma^p}|B_\rho|,
\end{align*} where in the first inequality we used $w_-(t,x)\le M$ for all $t\in (0, \delta M^{2-p}\rho^{sp})$ and $x\in B_\rho$.
Putting all the inequalities together, we know that for all $0<\bar{t}\le\delta M^{2-p}\rho^{sp}$,
\begin{align*}
\int_{B_{(1-\sigma)\rho}}(u-M)_-^2(\bar{t},x)\,dx&\le\int_{B_\rho}(u- M)_-^2(0,x)\,dx+
\frac{c_4\delta M^2}{\sigma^{d+sp}}|B_\rho|+
\frac{c_7\delta M^2}{\sigma^{p}}|B_\rho|\\
&\le M^2\Big[(1-\alpha)+\frac{c_
4\delta}{\sigma^{d+sp}}+\frac{c_
7\delta }{\sigma^{p}}\Big]|B_\rho|,
\end{align*}
where in the last inequality we used the condition that
$$|\{u(0,\cdot)\ge M\}\cap B_\rho|\ge\alpha|B_\rho|.$$

On the other hand,
\begin{align*}
\int_{B_{(1-\sigma)\rho}}(u- M)_-^2(\bar{t},x)\,dx&\ge\int_{B_{(1-\sigma)\rho}}(u- M)_-^2(\bar{t},x)\I_{\{u<\varepsilon M\}}\,dx\\
&\ge(1-\varepsilon)^2M^2|A_{\varepsilon,(1-\sigma)\rho}(\bar{t})|,
\end{align*}
where $$A_{\varepsilon,(1-\sigma)\rho}(\bar{t}):=\{u(\bar{t},\cdot)<\varepsilon M\}\cap B_{(1-\sigma)\rho}.$$

We also note that
\begin{align*}
|A_{\varepsilon,\rho}(\bar{t})|&=|A_{\varepsilon,(1-\sigma)\rho}(\bar{t})\cup(A_{\varepsilon,\rho}(\bar{t})\backslash A_{\varepsilon,(1-\sigma)\rho}(\bar{t}))|\\
&\le|A_{\varepsilon,(1-\sigma)\rho}(\bar{t})|+|B_\rho\backslash B_{(1-\sigma)\rho}|\\
&\le|A_{\varepsilon,(1-\sigma)\rho}(\bar{t})|+d\sigma|B_\rho|.
\end{align*}
Thus, for all $0<\bar{t}\le\delta M^{2-p}\rho^{sp}$, we have
\begin{align*}
|A_{\varepsilon,\rho}(\bar{t})|\le\frac{1}{(1-\varepsilon)^2}\Big[(1-\alpha)+\frac{c_
4\delta}{\sigma^{d+sp}}+\frac{c_
7\delta }{\sigma^{p}}\Big]|B_\rho|+d\sigma|B_\rho|.
\end{align*}
By choosing  $\sigma$, $\varepsilon$ and $\delta$ so that
$$d\sigma=\frac{1}{8}\alpha,\quad \frac{1-\alpha}{(1-\varepsilon)^2}\le1-\frac{3}{4}\alpha,\quad \frac{\delta}{(1-\varepsilon)^2}\left[\frac{c_4}{\sigma^{d+sp}}+\frac{c_7 }{\sigma^{p}}\right]=\frac{1}{8}\alpha,$$
we obtain that, for all $0<\bar{t}\le\delta M^{2-p}\rho^{sp}$,
$$|A_{\varepsilon,\rho}(\bar{t})|\le\Big(1-\frac{1}{2}\alpha\Big)|B_\rho|.$$
That is,
$$|\{u(\bar{t},\cdot)\ge\varepsilon M\}\cap B_\rho|\ge\frac{\alpha}{2}|B_\rho|\quad \text{for all}\,\, \bar{t}\in(0,\delta M^{2-p}\rho^{sp}].$$
The proof is complete. \end{proof}

The next
lemma is the so-called shrinking statement for measure information.

\begin{lemma}\label{Lem4}
Let $u$ be a locally bounded weak super-solution to \eqref{EQ}
in $I\times D$. Let $\alpha\in (0,1)$. For any $M>0$ satisfying  $\mathcal{Q}^-_{2\rho}(\theta;(t_0,x_0))\subset\mathcal{Q}$ with $\rho\in (0,R/2]$, $\theta
\ge(\sigma M)^{2-p}$ and $\sigma\in(0,1/2)$, if $$
\mathop{\ess\inf}_{(t,x)\in (t_0-\theta \rho^{sp}, t_0]\times B_{2\rho}(x_0)}u(t,x)\ge \mu^- $$
and
$$|\{u(t,\cdot)\ge\mu^-+M\}\cap B_{2\rho}(x_0)|\ge\alpha|B_{\rho}(x_0)|
\quad \text{for all}\,\,t\in(t_0-\theta\rho^{sp},t_0],$$
then it holds either
$$
{\rm \tilde{Tail}}^{p-1}_{p-1}
((u-\mu^-)_-; x_0,2\rho,t_0-\theta\rho^{sp},t_0)>\sigma M,$$
or
$$\left|\{u\le\mu^-+\sigma M/2\}\cap \mathcal{Q}^-_\rho(\theta;(t_0,x_0))\right|\le C_3\Big(\frac{\sigma}{1-\sigma}\Big)^{p-1}|\mathcal{Q}^-_\rho(\theta;(t_0,x_0))|$$
with some constant $C_3:=C_3(s,p,d,\alpha)>0$.
\end{lemma}

\begin{proof}
It suffices to show the case that $(t_0,x_0)=(0,0)$ and $\mu^-=0$. We assume that
$$
{\rm \tilde{Tail}}^{p-1}_{p-1}(u_-; 0,2\rho,-\theta\rho^{sp},0)\le\sigma M.$$
 Introduce a time-independent cutoff function $\phi(x)$ in $B_\rho\subset B_{3\rho/2}$, which satisfies $\|\nabla\phi\|_\infty\le4/\rho$. Let us employ Lemma \ref{2} in the cylinder $(-\theta\rho^{sp},0]\times B_{2\rho}$, with $
t_2=0$, $\tau_1-h=-\theta\rho^{sp}$, $\xi=1$, $\phi(t,x)=\phi(x)$, and with the truncation
$$w_-:=(u-\sigma M)_-.$$
Then,
\begin{align*}
&\iint_{\mathcal{Q}^-_\rho(\theta)}w_-(t,x)\left(\int_{B_{2\rho}}w_+(t,y)^{p-1} J(t;x,y)\,dy\right)\,dx\,dt\\
&\le\int_{-\theta\rho^{sp}}^{0}\int_{B_{2\rho}}w_-(t,x)\phi(x)^p\left(\int_{B_{2\rho}}w_+(t,y)^{p-1} J(t;x,y)\,dy\right)\,dx\,dt\\
&\le
C_0\int_{-\theta\rho^{sp}}^{0}\int_{B_{2\rho}}w_-(t,x) \phi(x)^p\left(\int_{\R^d\backslash B_{2\rho}}w_-(t,y)^{p-1}J(t;x,y)\,dy\right)\,dx\,dt\\
&\quad+
C_0\int_{-\theta\rho^{sp}}^{0}\int_{B_{2\rho}}\int_{ B_{2\rho}}\max\{w_-(t,x)^p,w_-(t,y )^p\}|\phi(x)-\phi(y)|^pJ(t;x,y)\,dx\,dy\,dt\\
&\quad+
C_0
\int_{B_{2\rho}}w_-(-\theta\rho^{sp},x)^2\,dx=:I_1+I_2+I_3.
\end{align*}

For $I_1$, we observe that for $x\in
\supp[\phi]\subseteq B_{3\rho/2}$ and $y\in\R^d\backslash B_{2\rho}$,
$$\frac{|y|}{|x-y|}\le1+\frac{|x|}{|x-y|}\le1+\frac{{3\rho}/{2}}{{\rho}/{2}}=4,$$
which, along with $\theta
\ge(\sigma M)^{2-p}$, yields that
\begin{align*}
I_1&\le c_1\int_{-\theta\rho^{sp}}^{0}\int_{B_{2\rho}}w_-(t,x) \phi(x)^p\left(\mathop{\ess\sup}_{x\in
\supp[\phi]}\int_{\R^d\backslash B_{2\rho}}\frac{w_-(t,y)^{p-1}}{|x-y|^{d+sp}}\,dy\right)\,dx\,dt\\
&\le c_2\sigma M|B_{2\rho}|\int_{-\theta\rho^{sp}}^{0}\int_{\R^d\backslash B_{2\rho}}\frac{w_-(t,y)^{p-1}}{|y|^{d+sp}}\,dy\,dt\\
&\le c_3\sigma M|B_{2\rho}|\left(\theta(\sigma M)^{p-1}+\int_{-\theta\rho^{sp}}^{0}\int_{\R^d\backslash B_{2\rho}}\frac{u_-(t,y)^{p-1}}{|y|^{d+sp}}\,dy\,dt\right)\\
&\le c_3\theta(\sigma M)^p|B_{2\rho}|+c_3(\sigma M)^2|B_{2\rho}|\\
&
\le2c_3\theta(\sigma M)^p|B_{2\rho}|
=\frac{c_4}{\rho^{sp}}(\sigma M)^p|\mathcal{Q}^-_\rho(\theta)|,
\end{align*}
where the second inequality follows from the fact that $w_-(t,x)\le \sigma M$ for all $
(t,x)\in (-\theta\rho^{sp},0]\times B_{2\rho}$, and  in the fourth inequality we used
${\rm \tilde{Tail}}^{p-1}_{p-1}
(u_-; 0,2\rho,-\theta\rho^{sp},0)\le\sigma M.$
For $I_2$,
\begin{align*}
I_2&\le c_5\theta\rho^{sp}(\sigma M)^p\int_{B_{2\rho}}\int_{ B_{2\rho}}\frac{|\phi(x)-\phi(y)|^p}{|x-y|^{d+sp}}\,dx\,dy\\
&\le c_6\theta\rho^{sp}\frac{1}{\rho^p}(\sigma M)^p\int_{B_{2\rho}}\int_{ B_{2\rho}}\frac{1}{|x-y|^{d+(s-1)p}}\,dx\,dy\\
&\le \frac{c_7}{\rho^{sp}}(\sigma M)^p|\mathcal{Q}^-_\rho(\theta)|,
\end{align*}
where
the first inequality follows from
$w_-(t,x)\le \sigma M$ for all $
(t,x)\in (-\theta\rho^{sp},0]\times B_{2\rho}$.
For $I_3$, due to $\theta
\ge(\sigma M)^{2-p}$,
$$I_3\le
C_
0(\sigma M)^2|B_{2\rho}|
\le\frac{c_8}{\rho^{sp}}(\sigma M)^p|\mathcal{Q}^-_\rho(\theta)|.$$
Hence, putting all the inequalities together, we
obtain
\begin{equation}\label{Right}\begin{split}
\iint_{\mathcal{Q}^-_\rho(\theta)}w_-(t,x)\left(\int_{B_{2\rho}}w_+(t,y)^{p-1} J(t;x,y)\,dy\right)\,dx\,dt\le \frac{c_9}{\rho^{sp}}(\sigma M)^p|\mathcal{Q}^-_\rho(\theta)|,
\end{split}\end{equation}
provided the tail estimate above is enforced.

On the other hand,
\begin{equation}\label{Left}\begin{split}
&\iint_{\mathcal{Q}^-_\rho(\theta)}w_-(t,x)\left(\int_{B_{2\rho}}w_+(t,y)^{p-1} J(t;x,y)\,dy\right)\,dx\,dt\\
&\ge
c_{10}
\iint_{\mathcal{Q}^-_\rho(\theta)}w_-(t,x)\I_{\{u(t,x)\le \sigma M/2\}}\left(\int_{B_{2\rho}}\frac{w_+(t,y)^{p-1}\I_{\{u(t,y)\ge M\}}}{|x-y|^{d+sp}}\,dy\right)\,dx\,dt\\
&\ge
\frac{c_{10}}{(4\rho)^{d+sp}}(\sigma M-\sigma M/2)(M-\sigma M)^{p-1}\Big|\Big\{
u \le \sigma M/2\Big\}\cap\mathcal{Q}^-_\rho(\theta)\Big|\Big|\{u(t,\cdot)\ge M\}\cap B_{2\rho}\Big|\\
&\ge
\frac{c_{10}\alpha|B_\rho|}{(4\rho)^{d+sp}}\frac{1}{2}\sigma(1-\sigma)^{p-1}M^p\Big|\Big\{
u \le \sigma M/2\Big\}\cap\mathcal{Q}^-_\rho(\theta)\Big|\\
\end{split}\end{equation}
where in the last inequality we used the fact that for all $t\in(-\theta\rho^{sp},0]$,
$$|\{u(t,\cdot)\ge M\}\cap B_{2\rho}|\ge\alpha|B_\rho|.$$

Combining \eqref{Right} with \eqref{Left}, we conclude the proof.
\end{proof}

\section{Proof of Theorem \ref{main1}}\label{section3}
This section is devoted to the proof of Theorem \ref{main1}. Though the proofs for the cases that $p\in(2,\infty)$ and $p\in (1,2)$
are similar, there are a lot of different points in the arguments. In particular, the constants $L$ involved in the proofs are completely different. We then divide this section into two parts. In each part, we will start from the key ingredient of the proof for the reduction of oscillations, which is called the expansion of positivity in the literature.

\subsection{Proof of Theorem \ref{main1}(i): $p>2$}

\begin{proposition}\label{pro1} Suppose that $p>2$. Let $u$ be a locally bounded weak super-solution to \eqref{EQ} in $I\times D$. Let $\alpha\in (0,1)$. Then, there exist
constants $\eta:=\eta(s,p,d,\alpha),\,\delta:=\delta(s,p,d,\alpha),\,\,\sigma':=\sigma'(s,p,d,\alpha)\in(0,1)$ and $b:=b(s,p,d,\alpha)\in(0,1)$, such that for any $M, \rho>0$ satisfying
\begin{equation}\label{inf}
\mathop{\ess\inf}_{(t,x)\in (t_0,t_0+\delta M^{2-p}\rho^{sp}]\times B_\rho(x_0)}u(t,x)\ge \mu^-
\end{equation}
and
\begin{equation}\label{A}
|\{u(t_0,\cdot)\ge\mu^-+M\}\cap B_\rho(x_0)|\ge\alpha|B_\rho(x_0)|,
\end{equation}
it holds  that either
$$
{\rm \tilde{Tail}}^{p-1}_{p-1}\left((u-\mu^-)_-; x_0,\rho, t_0, t_0+\frac{b^{p-2}}{(\eta M)^{p-2}}\delta\rho^{sp} \right)>\eta M,$$
or
$$u(t,\cdot)\ge\mu^-+\eta M \quad \text{a.e.\,\,in}\,\,B_{2\rho}(x_0)$$
for every
$t_0+\frac{b^{p-2}}{(\eta M)^{p-2}}\sigma'\delta\rho^{sp}\le t\le t_0+\frac{b^{p-2}}{(\eta M)^{p-2}}\delta\rho^{sp}$, provided that all the cylinders above are  included in $\mathcal{Q}$.
\end{proposition}

\begin{proof}
Without loss of generality, we will assume that $(t_0, x_0)=(0,0)$
and $\mu^-=0$.
The proof is split into five steps.

(i) It is clear from
\eqref{inf} and
\eqref{A} that for all $\sigma\le1$,
$$\mathop{\ess\inf}_{(t,x)\in (0,\delta(\sigma M)^{2-p}\rho^{sp}]\times B_\rho}u(t,x)\ge\mu^-$$
and
$$|\{u(0,\cdot)
\ge\sigma M\}\cap B_\rho|\ge\alpha|B_\rho|.$$
For $\tau\ge0$, set
$$\sigma_\tau:=e^{-\frac{\tau}{p-2}}\le1.$$
According to Lemma \ref{Lem3}
with $\mu^-=0$, there are constants $\delta:=\delta(s,p,d,\alpha),\,\,\varepsilon:=\varepsilon(s,p,d,\alpha)\in(0,1)$ such that
\begin{equation}\label{B}
\left|\left\{u\left(\frac{\delta\rho^{sp}}{(\sigma_\tau M)^{p-2}},\cdot\right)\ge\varepsilon\sigma_\tau M\right\}\cap B_\rho\right|\ge\frac{\alpha}{2}|B_\rho|
\end{equation}
provided that
\begin{equation}\label{G}
{\rm \tilde{Tail}}^{p-1}_{p-1}
(u_-;0,\rho,J_1)\le\delta\sigma_\tau M,
\end{equation}
where $J_1:=(0,\delta(\sigma_\tau M)^{2-p}\rho^{sp})$.
Let us define
$$\omega(\tau,x):=\frac{1}{\sigma_\tau M}(\delta\rho^{sp})^{\frac{1}{p-2}}u\left(
\frac{\delta\rho^{sp}}{(\sigma_\tau M)^{p-2}},x\right).$$
Then, \eqref{B} written in terms of $\omega$ translates to
$$
|\{\omega(\tau,\cdot)\ge b_0\}\cap B_\rho|\ge\frac{\alpha}{2}|B_\rho| \quad \text{for all}\,\,\tau>0,
$$
where
\begin{equation}\label{E}
b_0:=\varepsilon(\delta\rho^{sp})^{\frac{1}{p-2}}.
\end{equation}
That is,
\begin{equation}\label{C}
|\{\omega(\tau,\cdot)\ge b_0\}\cap B_{8\rho}|\ge\frac{\alpha}{2}
4^{-d}|B_{4\rho}|\quad \text{for all}\,\,\tau>0.
\end{equation}

On the other hand, recall that
$$\sigma_\tau=e^{-\frac{\tau}{p-2}}\le1,$$
then we can rewrite $\omega(\tau,x)$ as follows:
$$\omega(\tau,x)=
\left(\frac{e^\tau\delta\rho^{sp}}{ M^{p-2}}\right)^{\frac{1}{p-2}}u\left(\frac{e^\tau\delta\rho^{sp}}{ M^{p-2}},x\right).$$ One can easily see that
$$\frac{\partial\omega(\tau,x)}{\partial \tau}-\mathcal L_\tau
\omega(\tau,\cdot)(x)\ge0,\quad (\tau,x)\in (0,\infty)\times D.$$
Therefore, $\omega(\tau,x)$ is a locally bounded weak
super-solution
in $(0,\infty)\times D$.

(ii) Let
$$b_{j_*}=\frac{1}{2^{j_*}}b_0
,$$
where $b_0$ is given by \eqref{E} and $j_*$ is to be chosen later.
By \eqref{C}, we now
apply Lemma \ref{Lem4}
to $\omega$ over the cylinder
$$\mathcal{Q}^*_{4\rho}(\theta):=\mathcal{Q}^-_{4\rho}(\theta;(\tau_*,0)) \quad \text{for}\,\,\tau_*=\theta(8\rho)^{sp} $$
with
$\mu^-=0$, $\theta=b^{2-p}_{j_*}$,
 $M=b_0$ and $\sigma=\frac{1}{2^{j_*}}$, to see that for some constant $C_1:=C_1(s,p,d,\alpha)>0$
$$
|\{\omega\le b_{{j_*}+1}\}\cap \mathcal{Q}^*_{4\rho}(\theta)|\le C_1\Big(\frac{1}{2^{j_*}-1}\Big)^{p-1}|\mathcal{Q}^*_{4\rho}(\theta)|
$$
provided we enforce that
\begin{equation}\label{Tail}
{\rm \tilde{Tail}}^{p-1}_{p-1}(\omega_-;0,8\rho,\tau_*-\theta(4\rho)^{sp},\tau_*)\le b_{j_*}.
\end{equation}
For a given $\nu\in(0,1)$ that will be fixed later, we choose $j_*$ such that
$$C_1\Big(\frac{1}{2^{j_*}-1}\Big)^{p-1}\le \nu $$ and set $$ \varepsilon_\nu:=\frac{1}{2^{j_*+1}}.$$
Note that $b_{j_*+1}=\varepsilon_\nu b_0$, then we obtain that
\begin{equation}\label{D}
|\{\omega\le \varepsilon_\nu b_0\}\cap \mathcal{Q}^*_{4\rho}(\theta)|\le \nu|\mathcal{Q}^*_{4\rho}(\theta)|
\end{equation}
provided that
\eqref{Tail} is satisfied.
Furthermore, we rewrite the tail alternative of \eqref{Tail} in terms of $u$ as follows:
\begin{equation} \label{eq:2-15}
\begin{split}
&{\rm \tilde{Tail}}^{p-1}_{p-1}
(\omega_-;0,8\rho,
 \tau_*-\theta(4\rho)^{sp},\tau_* )\\
 &=\int_{\theta(8\rho)^{sp}-\theta(4\rho)^{sp}}^{\theta(8\rho)^{sp}} \int_{\R^d\setminus B_{8\rho}}\frac{\omega^{p-1}_-(\tau,y)}{|y|^{d+sp}}\,dy\,d\tau\\
&=
\int_{\theta(8\rho)^{sp}-\theta(4\rho)^{sp}}^{\theta(8\rho)^{sp}} \int_{\R^d\setminus B_{8\rho}}\frac{u^{p-1}_-(e^\tau M^{2-p}\delta\rho^{sp},y)}{|y|^{d+sp}}
 \left(\frac{e^\tau\delta\rho^{sp}}{ M^{p-2}}\right)^{\frac{p-1}{p-2}}
\,dy\,d\tau\\
&=\int_{J_2}\int_{\R^d\setminus B_{8\rho}}\frac{u^{p-1}_-(t,y)}{|y|^{d+sp}}t^{\frac{1}{p-2}}\,dy\,dt\\
&\le (\delta\rho^{sp})^{\frac{1}{p-2}}\frac{\exp\Big(\theta(8\rho)^{sp}/(p-2)\Big)}{M}\int_{J_2}\int_{\R^d\setminus B_{8\rho}}\frac{u^{p-1}_-(t,y)}{|y|^{d+sp}}\,dy\,dt,
\end{split}\end{equation}
where
the third equality follows from $d\tau=t^{-1}\,dt$, and
\begin{equation}\label{J2}
J_2:=(\exp(\theta(8\rho)^{sp}-\theta(4\rho)^{sp})M^{2-p}\delta\rho^{sp},\exp(\theta(8\rho)^{sp})M^{2-p}\delta\rho^{sp}).
\end{equation}
Thus, \eqref{Tail} is
verified if the following holds
$$\int_{J_2}\int_{\R^d\setminus B_{8\rho}}\frac{u^{p-1}_-(t,y)}{|y|^{d+sp}}\,dy\,dt\le\frac{b_{j_*}}{(\delta\rho^{sp})^{\frac{1}{p-2}}}\frac{M}{\exp\Big(\theta(8\rho)^{sp}/(p-2)\Big)};$$ i.e.,
$${\rm \tilde{Tail}}^{p-1}_{p-1}(u_-;0,8\rho,J_2)\le\frac{b_{j_*}}{(\delta\rho^{sp})^{\frac{1}{p-2}}}\frac{M}{\exp\Big(\theta(8\rho)^{sp}/(p-2)\Big)}.$$
Noting that
$$b_{j_*}=\frac{b_0}{2^{j_*}}=\frac{\varepsilon(\delta\rho^{sp})^{\frac{1}{p-2}}}{2^{j_*}}=2\varepsilon_\nu\varepsilon(\delta\rho^{sp})^{\frac{1}{p-2}} \quad \text{and} \quad \theta=b^{2-p}_{j_*}=\frac{(2\varepsilon_\nu\varepsilon)^{2-p}}{\delta\rho^{sp}},$$
we see
that \eqref{Tail} is satisfied  if
\begin{equation}\label{H}
{\rm \tilde{Tail}}^{p-1}_{p-1}
(u_-;0,8\rho,J_2)\le2\varepsilon_\nu\varepsilon M \exp\Big(-\frac{8^{sp}}{(p-2)(2\varepsilon_\nu\varepsilon)^{p-2}\delta}\Big).
\end{equation}

(iii) We apply Lemma \ref{Lem1} to $\omega$ over the cylinder
$\mathcal{Q}^*_{4\rho}(\theta)$ with
$$\mu^-=0, \quad  M=\varepsilon_\nu b_0, \quad a=\frac{1}{2},\quad \theta
=\frac{(2\varepsilon_\nu\varepsilon)^{2-p}}{\delta\rho^{sp}}\quad \text{and} \quad  \lambda=\frac{1}{2},$$ to
obtain either
\begin{equation}\label{F}
\tilde{\gamma}{\rm \tilde{Tail}}^{p-1}_{p-1}
(\omega_-;0,4\rho, \tau_*-\theta(4\rho)^{sp},\tau_* )
>
\frac{1}{4}\varepsilon_\nu b_0,
\end{equation}
or
\begin{equation}\label{om}
\omega(\tau,x)\ge\frac{1}{4}\varepsilon_\nu b_0 \quad \text{for\,\,a.e.} \,\, (\tau,x)\in
\mathcal{Q}^-_{2\rho}(\theta;(\tau^*,0))
\end{equation}
holds, provided
$$\frac{|\{\omega\le\varepsilon_\nu b_0\}\cap \mathcal{Q}^*_{4\rho}(\theta)|}{|\mathcal{Q}^*_{4\rho}(\theta)|}\le
C_22^{-(2s+d)/s}\theta^{-1}(\varepsilon_\nu b_0)^{2-p}\Big(1+\theta^{-1}(\varepsilon_\nu b_0)^{2-p}\Big)^{-\frac{d+sp}{sp}}=C_22^{p-4-d/s}\Big(1+2^{p-2}\Big)^{-\frac{d+sp}{sp}}$$
for some constant $C_2:=C_2(s,p,d)>0$. Here, $\tilde \gamma\ge1$ and $C_2$ is the constant $C_1$ given in Lemma \ref{Lem1}.
Then the condition above is ensured by \eqref{D} since one can take $C_2$ small enough so that $$\nu=
C_22^{p-4-d/s}\Big(1+2^{p-2}\Big)^{-\frac{d+sp}{sp}}\in (0,1).$$
Furthermore,
by \eqref{eq:2-15} (replacing $\R^d\setminus B_{8\rho}$
with $\R^d\setminus B_{4\rho}$),
the tail alternative \eqref{F} can be reformulated in terms of $u$ as follows:
\begin{align*}
{\rm \tilde{Tail}}^{p-1}_{p-1}
(\omega_-;0,4\rho, \tau_*-\theta(4\rho)^{sp},\tau_* )
&\le(\delta\rho^{sp})^{\frac{1}{p-2}}\frac{\exp\Big(\theta(
8\rho)^{sp}/({p-2})\Big)}{M}
{\rm \tilde{Tail}}^{p-1}_{p-1}(u_-;0,4\rho, J_2),
\end{align*}
where $J_2$ is defined in \eqref{J2}. In particular, if
\begin{equation}\label{TI}\begin{split}
{\rm \tilde{Tail}}^{p-1}_{p-1}
(u_-;0,4\rho, J_2)&\le{
(4\tilde{\gamma})}^{-1}\frac{\varepsilon_\nu b_0}{(\delta\rho^{sp})^{\frac{1}{p-2}}}\frac{M}{\exp\Big(\theta(8\rho)^{sp}/(p-2)\Big)}\\
&={
(4\tilde{\gamma})}^{-1}\varepsilon_\nu\varepsilon M \exp\Big(-\frac{8^{sp}}{(p-2)(2\varepsilon_\nu\varepsilon)^{p-2}\delta}\Big),\end{split}
\end{equation} then $$\tilde{\gamma}{\rm \tilde{Tail}}^{p-1}_{p-1}
(\omega_-;0,4\rho,\tau_*-\theta(4\rho)^{sp},\tau_*)\le
\frac{1}{4}\varepsilon_\nu b_0,$$ and so \eqref{om} is satisfied.

(iv) Now, we assume that all the tail assumptions from \eqref{G}, \eqref{H} and \eqref{TI} hold. For any $\tau\in(\theta(8\rho)^{sp}-\theta(2\rho)^{sp},\theta(8\rho)^{sp})$ in \eqref{om}, we have
$$e^{\frac{\tau}{p-2}}\in\left[\exp\left(\frac{8^{sp}-2^{sp}}{(p-2)(2\varepsilon_\nu\varepsilon)^{p-2}\delta}\right),
\exp\left(\frac{8^{sp}}{(p-2)(2\varepsilon_\nu\varepsilon)^{p-2}\delta}\right)\right]=:[k_1,k_2].$$
Recalling that $\sigma_\tau:=e^{-\frac{\tau}{p-2}}$ and $b_0:=\varepsilon(\delta\rho^{sp})^{\frac{1}{p-2}}$, we can rewrite the conclusion in \eqref{om} in terms of $u$ as follows:
\begin{equation}\label{utx}
u(t,x)=\sigma_\tau M(\delta\rho^{sp})^{\frac{1}{2-p}}\omega(\tau,x)\ge\frac{1}{k_2}M(\delta\rho^{sp})^{\frac{1}{2-p}}\frac{1}{4}\varepsilon_\nu b_0=\frac{\varepsilon_\nu\varepsilon}{4k_2}M\ge \frac{\varepsilon_\nu\varepsilon}{4\tilde \gamma k_2}M\ge \eta M
\end{equation}
for a.e. $x\in B_{2\rho}$ and all
$$t:=e^\tau M^{2-p}\delta\rho^{sp}\in\left[k_1^{p-2} M^{2-p}\delta\rho^{sp},k_2^{p-2}M^{2-p}\delta\rho^{sp}\right]=\left[\frac{b^{p-2}}{(\eta M)^{p-2}}\sigma'\delta\rho^{sp},\frac{b^{p-2}}{(\eta M)^{p-2}}\delta\rho^{sp}\right],$$
where $\eta\le \frac{\varepsilon_\nu\varepsilon}{4\tilde \gamma k_2}$ which will be fixed later,  $b=\eta k_2\in (0,1)$ and $\sigma'=\left(\frac{k_1}{k_2}\right)^{p-2}\in(0,1)$.

(v) Recall that, to ensure that
\eqref{utx} holds,
we need to assume
that \eqref{G}, \eqref{H} and \eqref{TI} hold, i.e.,
\begin{align*}
&{\rm \tilde{Tail}}^{p-1}_{p-1}(u_-;0,\rho,J_1)\le\delta\sigma_\tau M\,\,\hbox{with all}\,\, \tau \le \theta(8\rho)^{sp},\\
&{\rm \tilde{Tail}}^{p-1}_{p-1}
(u_-;0,8\rho,J_2)\le2\varepsilon_\nu\varepsilon M \exp\Big(-\frac{8^{sp}}{(p-2)(2\varepsilon_\nu\varepsilon)^{p-2}\delta}\Big),\\
&{\rm \tilde{Tail}}^{p-1}_{p-1}
(u_-;0,4\rho,J_2)\le
(4\tilde{\gamma})^{-1}\varepsilon_\nu\varepsilon M \exp\Big(-\frac{8^{sp}}{(p-2)(2\varepsilon_\nu\varepsilon)^{p-2}\delta}\Big),
\end{align*}
where $\tilde{\gamma}\ge1$, and
\begin{align*}
&J_1=
(0,\delta(\sigma_\tau M)^{2-p}\rho^{sp}),\\
&J_2=\Big(\exp\big(\theta(8\rho)^{sp}-\theta(4\rho)^{sp}\big)M^{2-p}\delta\rho^{sp},\exp(\theta(8\rho)^{sp})M^{2-p}\delta\rho^{sp}\Big).
\end{align*}
Since $\tau\le\theta(8\rho)^{sp}$,
\begin{equation}\label{JJ}
J_1=(0,e^\tau M^{2-p}\delta\rho^{sp})\subset(0,\exp(\theta(8\rho)^{sp}) M^{2-p}\delta\rho^{sp})=:J_0,
\end{equation}
and $$\inf_{\tau\le \theta(8\rho)^{sp}} \sigma_\tau= \exp\left({-\frac{\theta(8\rho)^{sp}}{
p-2}}\right)=\exp\Big(-\frac{8^{sp}}{(p-2)(2\varepsilon_\nu\varepsilon)^{p-2}\delta}\Big)=\frac{1}{k_2}.$$
It obvious that $J_2\subset J_0$. Set
$$\eta:=\min\Big\{\frac{\delta}{k_2},\frac{\varepsilon_\nu\varepsilon}{4\tilde{\gamma}k_2}\Big\} \quad \text{for} \,\, \tilde{\gamma}\ge1,$$
which clearly satisfies that $\eta\le \frac{\varepsilon_\nu\varepsilon}{4\tilde \gamma k_2}$.
We see that all the tail assumptions  \eqref{G}, \eqref{H} and \eqref{TI} are satisfied, if we require
$${\rm \tilde{Tail}}^{p-1}_{p-1}(u_-;0,\rho,J)\le\eta M,$$
where
$$J=\left(0,\frac{b^{p-2}}{(\eta M)^{p-2}}\delta\rho^{sp}\right).$$
Hence
the proof is complete.
\end{proof}

\begin{proposition}\label{pro2}
Let $b,\eta$ and $\delta$ be
as
in Proposition $\ref{pro1}$ for $\alpha={1}/{2}$.  Then, there exist a non-decreasing sequence $\{m_i\}_{i\ge0}$, a non-increasing sequence $\{M_i\}_{i\ge0}$, $C_0>1$, $\beta\in(0,1)$ with
$\beta<\min\{\frac{sp}{p-1},\log_{C_0}\big(\frac{2}{2-\eta}\big)\}$, $\varepsilon>0$  and $j_0\in \mathbb{N}$, such that, for any $i\in\mathbb{N}\cup\{0\}$,
\begin{equation}\label{J}
m_i\le u\le M_i \quad a.e.\,\,\hbox{in}\,\,\mathcal{Q}_{R_i}(d_i) \quad \hbox{and} \quad M_i-m_i=C_0^{-\beta i}L
\end{equation}
with
\begin{align*}L\ge&2C_0^{\frac{sp}{p-1}j_0}\|u\|_{L^\infty(\mathcal{Q}_{R_0}(d_0))}+\left(\oint^{t_0}_{t_0-(C_0R)^{sp}}\left((C_0R)^{sp}\int_{\R^d\setminus B_{C_0R}(x_0)}\frac{
u_-(t,x)^{p-1}}{|x-x_0|^{d+sp}}\,dx\right)^{1+\varepsilon}\,dt\right)^{1/
(1+(p-1)\varepsilon)}\\
&+\frac{2b}{\eta}\left(\delta C_0^{\beta(p-2)-sp}\right)^{1/(p-2)},\end{align*}
where $\mathcal{Q}_{R_0}(d_0):=(t_0-(C_0R)^{sp}
,t_0)\times B_{C_0R}(x_0)$,  and for $i\in\mathbb{N}$, $\mathcal{Q}_{R_i}(d_i):=(t_0-d_iR_i^{sp},t_0)\times B_{R_i}(x_0)$ with
$$ R_i:=C_0^{1-i}R, \quad
d_i:=\frac{b^{p-2}}{\Big(\eta\frac{C_0^{-\beta i}L}{2}\Big)^{p-2}}\delta.
$$
\end{proposition}

\begin{proof}
Without loss of generality, we may and do assume that $(t_0,x_0)=(0,0)$. Choose $\beta\in(0,1)$ so that $\beta<\frac{sp}{p-2}$
(which is due to $\beta<\frac{sp}{p-1}$).
Then, thanks to the definition of $L$, $C_0^{sp-\beta(p-2)}\ge \delta(2b/(\eta L))^{p-2}$, and so $\mathcal{Q}_{R_1}(d_1)\subset\mathcal{Q}_{R_0}(d_0)$. Furthermore, for all $i\in \mathbb{N}$, by $\beta<\frac{sp}{p-2}$ again,
$$d_{i+1}R_{i+1}^{sp}
=\frac{b^{p-2}}{\Big(\eta\frac{C_0^{-\beta i}L}{2}\Big)^{p-2}}\delta(C_0^{1-i}R)^{sp}C_0^{\beta(p-2)-sp}\le \frac{b^{p-2}}{\Big(\eta\frac{C_0^{-\beta i}L}{2}\Big)^{p-2}}\delta(C_0^{1-i}R)^{sp}
=d_iR_i^{sp},$$
and so $\mathcal{Q}_{R_{i+1}}(d_{i+1})\subseteq\mathcal{Q}_{R_i}(d_i)$ for $i\in\mathbb{N}$.
In the following, for simplicity we write
$(-d_iR_i^{sp},0)\times B_{R_i}(0)=I_i\times B_i$ for all $i\ge1$ and $(-
(C_0R)^{sp},0)\times B_{C_0R}(0)=I_0\times B_0$.

The proof
is split into
two steps.
First we show \eqref{J} holds for $i=0,1,\ldots,j_0$ (with $j_0$ to be fixed later), and then we use the induction approach to obtain that \eqref{J} holds for all $i>j_0$.

(i) For $i=0,1,\ldots,j_0$, set $m_i:=\frac{-C_0^{-\beta i}L}{2}$ and $M_i:=\frac{C_0^{-\beta i}L}{2}$. Since $\mathcal{Q}_{R_i}(d_i)\subset\mathcal{Q}_{R_0}(d_0)$, we have
$$\|u\|_{L^\infty(\mathcal{Q}_{R_i}(d_i))}\le\|u\|_{L^\infty(\mathcal{Q}_{R_0}(d_0))}=\frac{C_0^{-\beta i}}{2}\left(2C_0^{\frac{sp}{p-1}j_0}\|u\|_{L^\infty(\mathcal{Q}_{R_0}(d_0))}\right){C_0^{\beta i-\frac{sp}{p-1}j_0}}\le\frac{C_0^{-\beta i}}{2}L=M_i,$$
where in the last inequality we have chosen $\beta\in(0,1)$ satisfying $\beta<\frac{sp}{p-1}$.

(ii) Now we fix an integer $j> j_0$ and suppose that the sequences $\{m_i\}_{i\ge0}$ and $\{M_i\}_{i\ge0}$ have been constructed up to $i=j$.
We claim that \eqref{J} holds
once we construct $m_{j+1}$ and $M_{j+1}$ appropriately. To do so, we consider the function
$$v:=u-\frac{M_j+m_j}{2}.$$
Then, one of the following two alternatives must hold:
\begin{equation}\label{Al1}
\Big|\{v(s_j,\cdot)\ge0\}\cap B_
j\Big|\ge\frac{1}{2}\Big|B_
j\Big|
\end{equation} and
\begin{equation}\label{Al2}
\Big|\{v(s_j,\cdot)\ge0\}\cap B_
j\Big|<\frac{1}{2}\Big|B_
j\Big|,
\end{equation}
where
$$s_j:=-d_jR_j^{sp}=-\frac{b^{p-2}}{\Big(\eta\frac{C_0^{-\beta j}L}{2}\Big)^{p-2}}\delta R_j^{sp}.$$
Without loss of generality, we assume that \eqref{Al1} holds; otherwise, if \eqref{Al2} holds, analogous conclusion follows with $v$ replaced by $-v$. Now, denote $\omega=\frac{C_0^{-\beta j}L}{2}+v$. Then, \eqref{Al1}
can be written as
$$\Big|\Big\{\omega(s_j,\cdot)\ge\frac{C_0^{-\beta j}L}{2}\Big\}\cap B_
j\Big|\ge\frac{1}{2}\Big|B_
j\Big|.$$
Moreover, from the induction hypothesis, we see that $\omega\ge0$  in $\mathcal{Q}_{R_j}(d_j)$ because
$$\omega=v+\frac{C_0^{-\beta j}L}{2}=u-\frac{M_j+m_j}{2}+\frac{C_0^{-\beta j}L}{2}=u-\frac{M_j+m_j}{2}+\frac{M_j-m_j}{2}=u-m_j\ge0.$$
Our goal is to apply Proposition \ref{pro1} with $\mu^-=0$,
$\alpha=\frac{1}{2}$, $M=\frac{C_0^{-\beta j}L}{2}$ and $\rho=R_j$ to deduce that
\begin{equation}\label{omega}
\omega(t,x)\ge\eta\frac{C_0^{-\beta j}L}{2} \quad \text{a.e.}\,\,(t,x)\in\mathcal{Q}_{R_{j+1}}(d_{j+1}),
\end{equation}
provided $C_0>(\frac{1}{1-\sigma'})^{\frac{1}{sp-\beta(p-2)}}$ where $\sigma'\in(0,1)$ is as obtained in Proposition \ref{pro1}. Furthermore, \eqref{omega} implies that
\begin{align*}
u(t,x)&=\frac{M_j+m_j}{2}+
v(t,x)=\frac{M_j+m_j}{2}+
\omega(t,x)-\frac{C_0^{-\beta j}L}{2}\\
&\ge\frac{M_j+m_j}{2}-\frac{C_0^{-\beta j}L}{2}(1-\eta)=M_j-\frac{M_j-m_j}{2}-\frac{C_0^{-\beta j}L}{2}(1-\eta)\\
&=M_j-\frac{C_0^{-\beta j}L}{2}-\frac{C_0^{-\beta j}L}{2}(1-\eta)=M_j-\frac{C_0^{-\beta j}L}{2}(2-\eta)\\
&\ge M_j-C_0^{-\beta(j+1)}L,
\end{align*}
where the fourth equality follows from the induction hypothesis, and the last inequality follows by taking $\beta\in(0,1)$ small enough so that $\beta<\log_{C_0}\big(\frac{2}{2-\eta}\big)$. By defining $m_{j+1}:=M_j-C_0^{-\beta(j+1)}L$ and $M_{j+1}:=M_j$, we verify \eqref{J} for $m_{j+1}$ and $M_{j+1}$, as desired.

It remains to verify the tail alternative of Proposition \ref{pro1}.
We have, for all $j>j_0$,
\begin{align*}
&{\rm \tilde{Tail}}^{p-1}_{p-1}
\big(\omega_-;0,R_j,-d_jR_j^{sp},0
\big)\\
&={\rm \tilde{Tail}}^{p-1}_{p-1}
\left(\Big(\frac{C_0^{-\beta j}L}{2}+ v\Big)_-;0,R_j,-d_jR_j^{sp},0\right)
=\int^0_{-d_jR_j^{sp}}\int_{\R^d\setminus B_j}\frac{\Big(\frac{C_0^{-\beta j}L}{2}+ v(t,x)\Big)_-^{p-1}}{|x|^{d+sp}}\,dx\,dt\\
&= \int^0_{-d_jR_j^{sp}} \left(\int_{B_0\setminus B_j}\frac{\Big(\frac{C_0^{-\beta j}L}{2}+ v(t,x)\Big)_-^{p-1}}{|x|^{d+sp}}\,dx+\int_{\R^d\setminus B_0}\frac{\Big(\frac{C_0^{-\beta j}L}{2}+ v(t,x)\Big)_-^{p-1}}{|x|^{d+sp}}\,dx \right) \,dt \\
&=:J_1+J_2.
\end{align*}
Below, we
estimate $J_1$ and $J_2$. Take
$x\in B_0\setminus B_j$ and let $l\in\{0,1,\ldots,j-1\}$ be the unique integer for which $x\in B_{l}\setminus B_{l+1}$. In particular, $|x|\ge R_{l+1}=C_0^{-l}R$ implies that $\frac{|x|}{R}\ge C^{-l}_0$. By the monotonicity of $\{M_j\}_{j\ge0}$ and $\{m_j\}_{j\ge0}$ and by the induction hypothesis, for a.e. $t\in I_j$, we have
\begin{align*}
v(t,x)&=u(t,x)-\frac{M_j+m_j}{2}\ge m_l-\frac{M_j+m_j}{2}\ge m_l-M_l+M_j-\frac{M_j+m_j}{2}\\
&=-(M_l-m_l)+\frac{M_j-m_j}{2}=-\left[C_0^{-\beta l}-\frac{C_0^{-\beta j}}{2}\right]L\ge-\left[\left(\frac{|x|}{R}\right)^\beta-\frac{C_0^{-\beta j}}{2}\right]L.
\end{align*}
Then,
$$\Big(\frac{C_0^{-\beta j}L}{2}+ v(t,x)\Big)_-^{p-1}\le\left[\left(\frac{|x|}{R}\right)^\beta-C_0^{-\beta j}\right]^{p-1}L^{p-1},$$
and so
\begin{align*}
\int_{B_0\setminus B_j}\frac{\Big(\frac{C_0^{-\beta j}L}{2}+v(t,x)\Big)_-^{p-1}}{|x|^{d+sp}}\,dx
 &\le\int_{\R^d\setminus B_j}\frac{\left[\left(\frac{|x|}{R}\right)^\beta-C_0^{-\beta j}\right]^{p-1}L^{p-1}}{|x|^{d+sp}}\,dx\\
&=C_0^{ \beta (1-j)(p-1)}L^{p-1}\int_{\R^d\setminus B_j}\frac{\left(\Big(\frac{|x|}{C_0^{1-j}R}\Big)^\beta-\frac{1}{C^\beta_0}\right)^{p-1}}{|x|^{d+sp}}\,dx\\
&=\frac{C_0^{\beta (1-j)(p-1)}L^{p-1}}{R_j^{sp}}\int_{\{|y|\ge 1\}}\frac{(|y|^\beta-C_0^{-\beta})^{p-1}}{|y|^{d+sp}}\,dy\\
&=\frac{C_1C_0^{\beta (1-j)(p-1)}L^{p-1}}{R_j^{sp}}\int_1^\infty\frac{(\rho^\beta-C_0^{-\beta})^{p-1}}{\rho^{1+sp}}\,d\rho
\end{align*}
for some $C_1>0$, where
in the second equality we used the change of variable $y:=\frac{
x}{C_0^{1-j}R}=\frac{
x}{R_j}$. Thus,
\begin{align*}
J_1\le& C_1d_jC_0^{\beta(1- j)(p-1)}L^{p-1}\int_1^\infty\frac{(\rho^\beta-C_0^{-\beta})^{p-1}}{\rho^{1+sp}}\,d\rho\\
=&\left(\frac{C_1C_0^{\beta(p-1)}b^{p-2}\delta }{(\eta/{2})^{p-2}}\int_1^\infty\frac{(\rho^\beta-C_0^{-\beta})^{p-1}}{\rho^{1+sp}}\,d\rho\right) C_0^{-\beta j} L.
\end{align*}
In particular, by taking $\beta\in (0,1)$ small enough, we can get,
for $j> j_0$,
$$J_1\le \frac{\eta}{3}\frac{C_0^{-\beta j} L}{2}.$$

Using the monotonicity of $\{m_{j}\}_{j\ge0}$ and $\{M_{j}\}_{j\ge0}$, for all $j>j_0$, we have
$$(M_j-m_j)+2m_{j_0}\le M_j+m_j\le2M_{j_0}-(M_j-m_j).$$
Recalling $m_{j_0}=-\frac{C_0^{-\beta j_0}L}{2}$ and $M_{j_0}=\frac{C_0^{-\beta j_0}L}{2}$,
it holds that
$$-(C_0^{-\beta j_0}-C_0^{-\beta j})L\le M_j+m_j\le(C_0^{-\beta j_0}-C_0^{-\beta j})L,$$
and so
\begin{align*}\Big(\frac{C_0^{-\beta j}L}{2}+ v(t,x)\Big)_-^{p-1}&=\Big(\frac{C_0^{-\beta j}L}{2}+u(t,x)-\frac{M_j+m_j}{2}\Big)_-^{p-1}\\
&\le\Big(u(t,x)-\frac{C_0^{-\beta j_0}L}{2}\Big)_-^{p-1}\le2^{p-1}u_-(t,x)^{p-1}+\left({C_0^{-\beta j_0}L}\right)^{p-1}.
\end{align*}
Thus,
\begin{align*}
J_2&\le\int^0_{-d_jR_j^{sp}}\left(\int_{\R^d\setminus B_0}\frac{ ({C_0^{-\beta j_0}L} )^{p-1}}{|x|^{d+sp}}\,dx +\int_{\R^d\setminus B_0}\frac{2^{p-1}u_-(t,x)^{p-1}}{|x|^{d+sp}}\,dx \right)\,dt  \\
&\le C_2\frac{b^{p-2}\delta}{(\eta/2)^{p-2}} C_0^{-\beta j_0(p-1)} C_0^{(\beta(p-2)-sp)j}L\\
&\quad+2^{p-1}(d_jR_j^{sp})^{\varepsilon/(1+\varepsilon)}\left(\int^0_{-d_jR_j^{sp}}\left(\int_{\R^d\setminus B_0}\frac{u_-(t,x)^{p-1}}{|x|^{d+sp}}\,dx \right)^{1+\varepsilon}\,dt\right
)^{1/(1+\varepsilon)}\\
&\le C_2\frac{b^{p-2}\delta}{(\eta/2)^{p-2}} C_0^{-\beta j_0(p-1)} C_0^{(\beta(p-1)-sp)j}C_0^{-\beta j}L\\
&\quad+C_3(b^{p-2}\delta)^{\varepsilon/(1+\varepsilon)}(2/\eta)^{\varepsilon(p-2)/(1+\varepsilon)}C_0^{-j\left(\varepsilon  (sp-\beta(p-2))/({1+\varepsilon})-\beta\right)}C_0^{-\beta j}L^{(2-p)\varepsilon/(1+\varepsilon)}\\
&\qquad\times(C_0R)^{sp\varepsilon/(1+\varepsilon)}\left(\int^0_{-(C_0R)^{sp}}\left(\int_{\R^d\setminus B_0}\frac{u_-(t,x)^{p-1}}{|x|^{d+sp}}\,dx \right)^{1+\varepsilon}\,dt\right)^{1/(1+\varepsilon)}\\
&=: J_{21}+J_{22},
\end{align*}
where $\varepsilon>0$, and in the second inequality
we used the H\"{o}lder inequality. Hence, by the fact that $\beta(p-1)<sp$, there is $j_1^*\ge 1$ such
that
for all $j> \max\{j_0,j^*_1\}$,
$$J_{21}\le \frac{\eta}{3}\frac{C_0^{-\beta j} L}{2}.$$
On the other hand, we choose $\varepsilon>0$ and $\beta\in (0,1)$ with $\beta<sp/(p-2)$ such that
$$ \frac{\varepsilon}{1+\varepsilon} (sp-\beta(p-2))>\beta,$$ and take $j_2^*\ge j_1^*$ so that for all $j>\max\{j_0,j_2^*\}$,
$$C_
3(b^{p-2}\delta)^{\varepsilon/(1+\varepsilon)}(2/\eta)^{\varepsilon(p-2)/(1+\varepsilon)}C_0^{-j\left(\varepsilon  (sp-\beta(p-2))/({1+\varepsilon})-\beta\right)}\le \eta/6.$$
This along with the definition of $L$ with $\varepsilon>0$ above yields
that for all $j>\max\{j_0,j^*_2\}$,
$$J_{22} \le \frac{\eta}{3}\frac{C_0^{-\beta j} L}{2}.$$ Here we used the fact that
$$L^{1+(p-2)\varepsilon/(1+\varepsilon)}\ge \left(\oint^{
0}_{
-(C_0R)^{sp}}\left((C_0R)^{sp}\int_{\R^d\setminus B_{C_0R}(0)}\frac{
u_-(t,x)^{p-1}}{|x
|^{d+sp}}\,dx\right)^{1+\varepsilon}\,dt\right)^{1/(1+\varepsilon)}.$$

Putting all the inequalities together and taking $j_0=j_2^*$,
we have
for all $j> j_0$,
\begin{align*}
&{\rm \tilde{Tail}}^{p-1}_{p-1}
\big(\omega_- ;0,R_j, -d_jR_j^{sp},0\big)
\le\eta\frac{C_0^{-\beta j}L}{2}=:\eta M.
\end{align*}
The proof is complete. \end{proof}

\begin{proof}[Proof of Theorem $\ref{main1}$
 {\rm(i):} $p>2$]
We use the
notation in Proposition \ref{pro2}.  As a direct consequence of Proposition \ref{pro2}, we have the following
decay of oscillations:
\begin{equation}\label{osc}
\mathop{\ess\, \text{osc}}_{\mathcal{Q}_{R_i}(t_0,x_0,d_i)}u
:=\mathop{\ess\sup}_{\mathcal{Q}_{R_i}(t_0,x_0,d_i)}u-\mathop{\ess\inf}_{\mathcal{Q}_{R_i}(t_0,x_0,d_i)}u\le C_0^{-\beta i}L,
\end{equation}
where $d_i$ with $i\ge1$ depends on $L$ that satisfies $L\ge L(t_0,x_0)$
with
\begin{align*}L(t_0,x_0):=&2C_0^{\frac{sp}{p-1}j_0}\|u\|_{L^\infty(\mathcal{Q}_{R_0}(d_0))}+\left(\oint^{t_0}_{t_0-(C_0R)^{sp}}\left((C_0R)^{sp}\int_{\R^d\setminus B_{C_0R}(x_0)}\frac{u_-(t,x)^{p-1}}{|x-x_0|^{d+sp}}\,dx\right)^{1+\varepsilon}\,dt\right)^{1/
(1+(p-1)\varepsilon)}\\
&+\frac{2b}{\eta}\left(\delta C_0^{\beta(p-2)-sp}\right)^{1/(p-2)}.\end{align*}

 Now, we turn to the proof of the theorem. In the proof, without loss of generality we will assume that $(t_0, x_0)=(0,0)$. Consider any two points $(t_1,x_1),(t_2,x_2)\in\mathcal{Q}_{R_0}(d_0)/2:=(-(C_0R)^{sp}/2,0)\times B_{C_0R/2}$
with $t_1\neq t_2$ and $x_1\neq x_2$. Then there exist non-negative integers $n$ and $m$ such that
\begin{equation}\label{Rn}
R_{n+1}<|x_1-x_2|\le R_n \qquad \text{and} \qquad d_{m+1}R_{m+1}^{sp}<|t_1-t_2|\le d_mR_m^{sp},
\end{equation}
where we understand that $d_0=1$ and $R_0=C_0R$ without confusion, and we again emphasize that $d_i$ (for all $i\ge1$) depends on the choice of $L$ that will be fixed later.
From the first inequality in \eqref{Rn}, we deduce
\begin{equation}\label{fir}
\frac{|x_1-x_2|}{R}>C_0^{-n}.
\end{equation}
On the other hand, from the second inequality in \eqref{Rn}, we get
$$|t_1-t_2|>\frac{b^{p-2}}{\Big(\eta\frac{C_0^{-\beta (m+1)}L}{2}\Big)^{p-2}}\delta(C_0^{-m}R)^{sp},$$ which implies that
\begin{equation}\label{sec}
C_*L^{\frac{p-2}{sp}}\frac{|t_1-t_2|^{\frac{1}{sp}}}{R}>C_0^{\beta(m+1)({p-2})/{sp}}C_0^{-m}\ge C_0^{-m},
\end{equation}
where \begin{equation}\label{C41}
C_*=\left(\frac{\eta^{p-2}}{(2b)^{p-2}\delta}\right)^{\frac{1}{sp}},
\end{equation}
and the last inequality follows from $C_0^{\beta(m+1)({p-2})/{sp}}\ge1$ since $p>2$
and $C_0>1$.

For any $(t_1,x_1)\in\mathcal{Q}_{R_0}(d_0)/2$ and  $x\in {B_{C_0R}(x_1)}^c$,
$$\frac{|x|}{|x-x_1|}\le 1+\frac{|x_1|}{|x-x_1|}\le 1+\frac{C_0R/2}{C_0R}\le 2 \quad\text{and}\quad |x|\ge |x-x_1|-|x_1|\ge C_0R-C_0R/2=C_0R/2,$$
which implies that
\begin{equation}\label{e:add}
\begin{split}
L(t_1,x_1)&\le 2C_0^{\frac{sp}{p-1}j_0}\|u\|_{L^\infty ((-2(C_0R)^{sp},0)\times B_{2C_0R})}+\frac{2b}{\eta}\left(\delta C_0^{\beta(p-2)-sp}\right)^{1/(p-2)}\\
&\quad+\left(\oint_{-2(C_0R)^{sp}}^{0}\left((C_0R)^{sp}\int_{\R^d\setminus B_{C_0R/2}}\frac{u_-(t,x)^{p-1}}{|x|^{d+sp}}\,dx\right)^{1+\varepsilon}\,dt\right)^{1/(1+(p-1)\varepsilon)}=:L.
\end{split}
\end{equation}
In particular, $L$ is independent of $(t_1,x_1)$. Below we will take this constant $L$ in \eqref{Rn} for all $d_i\ge1$ and $i\ge1$.

Let $k:=\min\{n,m\}$. Putting \eqref{Rn}, \eqref{fir}, \eqref{sec} and \eqref{e:add} together, we obtain that
\begin{align*}
|u(t_1,x_1)-u(t_2,x_2)|&\le\mathop{\ess\, \text{osc}}_{\mathcal{Q}_{R_k}(t_1,x_1,d_k)}u\le\max\{C_0^{-\beta n},C_0^{-\beta m}\}L\le(C_0^{-n}+C_0^{-m})^\beta L\\
&\le \frac{L}{R^\beta}\left(|x_1-x_2|+C_*L^{\frac{p-2}{sp}}|t_1-t_2|^{\frac{1}{sp}}\right)^\beta,
\end{align*}
where in the second inequality we used  \eqref{osc}. The proof is complete.
\end{proof}

\subsection{Proof of Theorem \ref{main1} (ii): $1<p<2$}
\begin{proposition}\label{pro3} Suppose that $1<p<2$. Let $u$ be a locally bounded weak super-solution to \eqref{EQ} in $I\times D$. Let $\alpha\in (0,1)$. Then, there exist
constants $\eta:=\eta(s,p,d,\alpha),\,\delta:=\delta(s,p,d,\alpha)$ and $\epsilon:=\epsilon(s,p,d,\alpha)\in(0,1)$ such that for any $M, \rho>0$ satisfying
\begin{equation}\label{inf2}
\mathop{\ess\inf}_{(t,x)\in (t_0,t_0+\delta M^{2-p}\rho^{sp}]\times B_\rho(x_0)}u(t,x)\ge \mu^-
\end{equation}
and
\begin{equation}\label{Brho}
|\{u(t_0,\cdot)\ge\mu^-+M\}\cap B_\rho(x_0)|\ge\alpha|B_\rho(x_0)|,
\end{equation}
it holds  that either
$${\rm \tilde{Tail}}^{p-1}_{p-1}\left((u-\mu^-)_-; x_0,\rho, t_0, t_0+\delta M^{2-p}\rho^{sp} \right)>\eta M,$$
or
$$u(t,\cdot)\ge\mu^-+\eta M \quad a.e.\,\,in\,\,B_{2\rho}(x_0)$$
for every $t_0+(1-\epsilon)\delta M^{2-p}\rho^{sp}\le t\le t_0+\delta M^{2-p}\rho^{sp}$, provided that all the cylinders above are included in $\mathcal Q$.
\end{proposition}

\begin{proof}
Without loss of generality, we will assume that $(t_0, x_0)=(0,0)$ and $\mu^-=0$.  The proof is split into five steps.

(i) Applying \eqref{inf2} and \eqref{Brho} to Lemma \ref{Lem3} with $\mu^-=0$,
there exist constants $\delta$ and $\varepsilon\in(0,1)$ such that
\begin{equation}\label{ut}
|\{u(t,\cdot)\ge\varepsilon M\}\cap B_\rho|\ge\frac{\alpha}{2}|B_\rho|\quad  \text{for all}\,\,t\in(0,\delta M^{2-p}\rho^{sp}],
\end{equation}
provided that
\begin{equation}\label{Jtt1}
{\rm \tilde{Tail}}^{p-1}_{p-1}
(u_-;0,\rho,J'_1)\le\delta M,
\end{equation}
where $J'_1:=(0,\delta M^{2-p}\rho^{sp})$.
 Let us now define the change of variables
\begin{equation}\label{ztau}
z=\frac{x}{\rho}, \quad -e^{-\tau}=\frac{t-\delta M^{2-p}\rho^{sp}}{\delta M^{2-p}\rho^{sp}} \quad \text{and} \quad v(\tau,z)=\frac{1}{M}\delta^{1/(p-2)}e^{\frac{\tau}{2-p}}u(t,x).
\end{equation}
Then, for
$\delta\in (0,1)$ and $p\in (1,2)$,
\eqref{ut} implies
\begin{equation}\label{vt}
|\{v(\tau,\cdot)\ge\varepsilon e^{\frac{\tau}{2-p}}\}\cap B_1|\ge\frac{\alpha}{2}|B_1|\quad  \text{for all}\,\,\tau\in(0,\infty).
\end{equation}
For $\tau_0>0$ which will be fixed later, set
\begin{equation}\label{2j}
b_0
:=\varepsilon e^{\frac{\tau_0}{2-p}}.
\end{equation}
Then, by \eqref{vt},
\begin{equation}\label{b0}
|\{v(\tau,\cdot)\ge b_0\}\cap B_{16}|\ge\frac{\alpha}{2}8^{-d}|B_8| \quad \text{for all}\,\,\tau\in(\tau_0,\infty).
\end{equation}

On the other hand, we write $v(\tau,z)=(e^{-\tau} \delta M^{2-p})^{1/(p-2)}u((1-e^{-\tau})\delta M^{2-p}\rho^{sp},\rho z).$ We can check that
$$\frac{\partial v(\tau,z)}{\partial \tau}-\mathcal L_\tau v(\tau,\cdot)(z)\ge0,\quad (\tau,z)\in (0,\infty)\times (D/\rho).$$ In particular, $v(\tau,z)$ is a locally bounded weak
super-solution
in $(0,\infty)\times (D/\rho)$.

(ii) Let $$b_{j}=\frac{1}{2^{j}}b_0,$$
where $b_0$ is given by \eqref{2j} and $j\ge1$.
By \eqref{b0},
we can
apply Lemma \ref{Lem4}
to $v$ over the cylinder
$$\mathcal{Q}_{\tau_0}=(\tau_0+b_0^{2-p},\tau_0+2b_0^{2-p})\times B_8$$
with
$\mu^-=0$,
$M=2^{(1-3sp)/(2-p)} b_0$, $\theta= M^{2-p}$ and $\sigma=\frac{1}{2^{j}}$, to see that for some constant $C_1:=C_1(s,p,d,\alpha)>0$,
$$
|\{v\le b_{j+1}\}\cap\mathcal{Q}_{\tau_0}|\le C_1\Big(\frac{1}{2^{j}-1}\Big)^{p-1}|\mathcal{Q}_{\tau_0}|
$$
provided we enforce that
\begin{equation}\label{Tail0}
{\rm \tilde{Tail}}^{p-1}_{p-1}
\big(v_-;0,16, \tau_0,\tau_0+2b_0^{2-p} \big)\le b_{j}.
\end{equation}

For a given $\nu\in(0,1)$ that will be fixed later, we choose $j_*\ge1$ such that
\begin{equation}\label{C1}
C_1\Big(\frac{1}{2^{j_*-1}-1}\Big)^{p-1}\le\nu.
\end{equation}
Then, we obtain that
\begin{equation}\label{bj}
|\{v\le b_{j_*}\}\cap\mathcal{Q}_{\tau_0}|\le\nu|\mathcal{Q}_{\tau_0}|
\end{equation}
provided that \eqref{Tail0} is satisfied with $j=j_*-1$. Furthermore, we will rewrite the tail alternative of \eqref{Tail0}
with $j=j_*-1$ in terms of $u$. Observe that
\begin{align*}
&\d\tau=\frac{e^{\tau}}{\delta M^{2-p}\rho^{sp}}\d t=\frac{1}{\delta M^{2-p}\rho^{sp}}\left(\frac{\delta M^{2-p}\rho^{sp}}{\delta M^{2-p}\rho^{sp}-t}\right)\d t,\\
&v^{p-1}_-(\tau,z)=
\delta ^{(p-1)/(p-2)}
\frac{1}{M^{p-1}}u^{p-1}_-(t,x)e^{\tau\frac{p-1}{2-p}}=
\delta ^{(p-1)/(p-2)}
\frac{1}{M^{p-1}}u^{p-1}_-(t,x)\left(\frac{\delta M^{2-p}\rho^{sp}}{\delta M^{2-p}\rho^{sp}-t}\right)^{\frac{p-1}{2-p}}.
\end{align*}
Then,
\begin{align*}
&{\rm \tilde{Tail}}^{p-1}_{p-1}
(v_-;0,16,\tau_0,\tau_0+2b_0^{2-p})=\int_{\tau_0}^{\tau_0+2b_0^{2-p}} \int_{\R^d\setminus B_{16}}\frac{v^{p-1}_-(\tau,z)}{|z|^{d+sp}}\,dz\,d\tau\\
&=\frac{\delta^{(p-1)/(p-2)}}{M^{p-1}}\int_{J'_2}\int_{\R^d\setminus B_{16\rho}}\frac{u^{p-1}_-(t,x)}{|x|^{d+sp}}\left(\frac{\delta M^{2-p}\rho^{sp}}{\delta M^{2-p}\rho^{sp}-t}\right)^{\frac{p-1}{2-p}}\rho^{d+sp}\Big(\frac{1}{\rho}\Big)^d\frac{1}{\delta M^{2-p}\rho^{sp}}\left(\frac{\delta M^{2-p}\rho^{sp}}{\delta M^{2-p}\rho^{sp}-t}\right)\,dx\,dt\\
&=\frac{1}{\delta^{1+(p-1)/(2-p)} M}\int_{J'_2}\int_{\R^d\setminus B_{16\rho}}\frac{u^{p-1}_-(t,x)}{|x|^{d+sp}}\left(\frac{\delta M^{2-p}\rho^{sp}}{\delta M^{2-p}\rho^{sp}-t}\right)^{\frac{1}{2-p}}\,dx\,dt\\
&\le\frac{\exp\Big((\tau_0+2b_0^{2-p})/({2-p})\Big)}{\delta^{1+(p-1)/(2-p)} M}\int_{J'_2}\int_{\R^d\setminus B_{16\rho}}\frac{u^{p-1}_-(t,x)}{|x|^{d+sp}}\,dx\,dt,
\end{align*}
where $$J'_2:=\left((1-e^{-\tau_0})\delta M^{2-p}\rho^{sp},(1-e^{-\tau_0-2b_0^{2-p}})\delta M^{2-p}\rho^{sp}\right).$$
Thus, \eqref{Tail0} holds with
$j=j_*-1$,
if
$$\int_{J'_2}\int_{\R^d\setminus B_{16\rho}}\frac{u^{p-1}_-(t,x)}{|x|^{d+sp}}\,dx\,dt\le\frac{\delta^{1+(p-1)/(2-p)} M}{\exp\Big((\tau_0+2b_0^{2-p})/({2-p})\Big)}b_{j_*-1};$$
namely
\begin{equation}\label{TJ2}
{\rm \tilde{Tail}}^{p-1}_{p-1}
(u_-;0,16\rho,J'_2)\le 2\exp\left(\frac{-2b_0^{2-p}}{2-p}\right)\frac{\varepsilon\delta^{1+(p-1)/(2-p)} M}{2^{j_*}}.
\end{equation}

(iii) Assume momentarily that $j_*$ is fixed, and
so $\nu$ has been determined according to (ii).
By choosing $j_*>0$ suitably (allowing the possibility of $j_*\notin {\N}$),
without loss of generality we may assume that $2^{j_*(2-p)}$ is an integer. Then,
let us subdivide $\mathcal{Q}_{\tau_0}$ into $2^{j_*(2-p)}$ cylinders so that each of them has length
$b_{j_*}^{2-p}$, by setting
$$\mathcal{Q}_n=(\tau_0+b_0^{2-p}+nb_{j_*}^{2-p},\tau_0+b_0^{2-p}+(n+1)b_{j_*}^{2-p})\times B_8 \qquad \text{for}\,\, n=0,1,\ldots,2^{j_*(2-p)}-1.$$
From (ii),
either \eqref{TJ2} holds or for at least one of the sub-cylinders, there must hold
\begin{equation}\label{mathQ}
|\{v\le b_{j_*}\}\cap\mathcal{Q}_n|\le\nu|\mathcal{Q}_n|.
\end{equation}
We now apply Lemma \ref{Lem1} to $v$ over $\mathcal{Q}_n$ with
$$\mu^-=0,\quad
M=b_{j_*}, \quad a=\frac{1}{2},\quad \lambda=\frac{1}{2}\quad \text{and}\quad \theta=8^{-sp}b_{j_*}^{2-p},$$
to obtain
$$
v\big(\tau_0+b_0^{2-p}+(n+1)b_{j_*}^{2-p},z\big)\ge\frac{1}{4}b_{j_*} \quad \text{a.e.}\,\,z\in B_4,
$$
provided the following two estimates hold
\begin{equation}\label{vb}
\tilde{\gamma}{\rm \tilde{Tail}}^{p-1}_{p-1}
(v_-;0,8, \tau_0+b_0^{2-p}+nb_{j_*}^{2-p},\tau_0+b_0^{2-p}+(n+1)b_{j_*}^{2-p})\le\frac{1}{4}b_{j_*}\end{equation}
and  \begin{equation}\label{vb-}
\frac{|\{v\le b_{j_*}\}\cap \mathcal{Q}_n|}{|\mathcal{Q}_n|}\le C_22^{-(2+d/s)}
\theta^{-1}b_{j_*}^{2-p}\Big(1+\theta^{-1}b_{j_*}^{2-p}\Big)^{-\frac{d+sp}{sp}}
=C_22^{3sp-(2+d/s)}\Big(1+8^{sp}\Big)^{-\frac{d+sp}{sp}},
 \end{equation}
where $\tilde\gamma\ge1$ and
$C_2:=C_2(s,p,d)>0$ is the constant $C_1$ given in Lemma \ref{Lem1}. We can take $C_2$ small enough so that the condition \eqref{vb-} is ensured by \eqref{mathQ} with $\nu=C_22^{3sp-(2+d/s)}(1+8^{sp})^{-\frac{d+sp}{sp}}$.
Summarizing, for the choice of $j_*$ and hence $\nu$,
by the above
and \eqref{C1}, there exist a time level $\tau_1:=\tau_0+b_0^{2-p}+(n+1)b_{j_*}^{2-p}$ in the range
\begin{equation}\label{tau1}
\tau_0+b_0^{2-p}<\tau_1<\tau_0+2b_0^{2-p}\end{equation}
such that
\begin{equation}\label{vz}v(\tau_1,z)\ge\sigma_0e^{\frac{\tau_0}{2-p}}
\quad \text{a.e.}\,\, z\in B_4,\end{equation} where $\sigma_0=\varepsilon2^{-(j_*+2)}\in(0,1).$
Furthermore, by the arguments in step (ii), the tail alternative \eqref{vb} can be reformulated in terms of $u$ as follows:
\begin{align*}
&{\rm \tilde{Tail}}^{p-1}_{p-1}
(v_-;0,8, \tau_0+b_0^{2-p}+nb_{j_*}^{2-p},\tau_0+b_0^{2-p}+(n+1)b_{j_*}^{2-p})\\
&\le\frac{\exp\Big((\tau_0+2b_0^{2-p}) /(2-p)\Big)}{\delta^{1+(p-1)/(2-p)} M}\int_{J'_3}\int_{\R^d\setminus B_{8\rho}}\frac{u^{p-1}_-(t,x)}{|x|^{d+sp}}\,dx\,dt,
\end{align*}
where $$J'_3:=\left((1-e^{-(\tau_0+b_0^{2-p}+nb_{j_*}^{2-p})})\delta M^{2-p}\rho^{sp},(1-e^{-(\tau_0+b_0^{2-p}+(n+1)b_{j_*}^{2-p})})\delta M^{2-p}\rho^{sp}\right).$$
In particular, if
\begin{equation}\label{I}\begin{split}
{\rm \tilde{Tail}}^{p-1}_{p-1}
(u_-;0,8\rho, J'_3)\le&{(4\tilde{\gamma})}^{-1}\frac{\delta^{1+(p-1)/(2-p)} M}{\exp\Big((\tau_0+2b_0^{2-p})
/(2-p)\Big)}b_{j_*}\\
=&{(4\tilde{\gamma})}^{-1}\exp\left(\frac{-2b_0^{2-p}}{2-p}\right)\frac{\varepsilon\delta^{1+(p-1)/(2-p)} M}{2^{j_*}},\end{split}
\end{equation}
then
$$\tilde{\gamma}{\rm \tilde{Tail}}^{p-1}_{p-1}
(v_-;0,8, \tau_0+b_0^{2-p}+nb_{j_*}^{2-p},\tau_0+b_0^{2-p}+(n+1)b_{j_*}^{2-p})\\
\le\frac{1}{4}b_{j_*},$$
and so \eqref{vz} is satisfied.

(iv) Recalling the change of variables from \eqref{ztau}, we see that if the tail alternatives \eqref{TJ2} and \eqref{I} are satisfied, then the following conclusion holds
$$u(t_1,\cdot)\ge\sigma_0Me^{-\frac{\tau_1-\tau_0}{2-p}}=:M_0 \qquad \text{in}\,\, B_{4\rho},$$
where $t_1=(1-e^{-\tau_1})\delta M^{2-p}\rho^{sp}$, thanks to \eqref{ztau}.
We now apply Lemma \ref{Lem2} with $M$ replaced by $M_0$, $\mu^-=0$ and
$\theta=C_3M_0^{2-p}$
with $C_3>0$
over the cylinder $(t_1,t_1+\theta(4\rho)^{sp}]\times B_{4\rho}$ to see that the following holds:
\begin{equation}\label{add-1}u(t,x)\ge\frac{1}{4}M_0=\frac{1}{4}\sigma_0Me^{-\frac{\tau_1-\tau_0}{2-p}}\ge\frac{1}{4}\sigma_0\exp\Big(-\frac{2 \varepsilon^{2-p} e^{\tau_0}}{2-p}\Big)M \ge
 \frac{1}{4}\sigma_0\exp\Big(-\frac{2   e^{\tau_0}}{2-p}\Big)M\quad \text{a.e. in}\,\,B_{2\rho}\end{equation}
for all
$$t_1\le t\le t_1+C_3M_0^{2-p}(4\rho)^{sp}$$
provided
\begin{equation}\label{4M}
{\rm \tilde{Tail}}^{p-1}_{p-1}
(u_-;0,4\rho,t_1,t_1+C_3M_0^{2-p}(4\rho)^{sp})\le{(4\tilde{\gamma})}^{-1}M_0.
\end{equation} Here, in \eqref{add-1} we used \eqref{tau1}.
In the following, we take \begin{equation}\label{tau0}
\tau_0=\ln\Big(\frac{\delta}{C_34^{sp}\sigma_0^{2-p}}\Big).
\end{equation} Then, taking into account the change of variables \eqref{ztau}, we can see that
\begin{equation*}
t_1+e^{-\tau_1}\delta M^{2-p}\rho^{sp}=\delta M^{2-p}\rho^{sp}=t_1+C_3M_0^{2-p}(4\rho)^{sp}=t_1+C_3\Big(\sigma_0Me^{-\frac{\tau_1-\tau_0}{2-p}}\Big)^{2-p}(4\rho)^{sp}.
\end{equation*}
Furthermore, according to \eqref{tau1} again,
\begin{equation*}
t_1=\delta M^{2-p}\rho^{sp}-C_3\Big(\sigma_0Me^{-\frac{\tau_1-\tau_0}{2-p}}\Big)^{2-p}(4\rho)^{sp}\le(1-e^{-\tau_0-2e^{\tau_0}})\delta M^{2-p}\rho^{sp}.
\end{equation*}
Hence,  we get that
\begin{equation}\label{uM}
u(t,\cdot)\ge\frac{1}{4}\sigma_0\exp\Big(-\frac{2e^{\tau_0}}{2-p}\Big)M\ge\eta M \quad a.e.\,\,\text{in}\,\,B_{2\rho},
\end{equation}
holds for all
$$(1-\epsilon)\delta M^{2-p}\rho^{sp}\le t\le\delta M^{2-p}\rho^{sp},$$ where $\epsilon=\exp(-\tau_0-2e^{\tau_0})$
and  $\eta\le\frac{\sigma_0}{4}\exp\Big(-\frac{2 e^{\tau_0}}{2-p}\Big)$ that is fixed later.

(v) Recall that, to ensure that \eqref{uM} holds, we
need to assume
that \eqref{Jtt1}, \eqref{TJ2}, \eqref{I} and \eqref{4M} hold, i.e.,
\begin{align*}
&{\rm \tilde{Tail}}^{p-1}_{p-1}
(u_-;0,\rho,J'_1)\le\delta M,\\
&{\rm \tilde{Tail}}^{p-1}_{p-1}
(u_-;0,16\rho,J'_2)\le 2\delta^{1+(p-1)/(2-p)}\exp\left(\frac{-2b_0^{2-p}}{2-p}\right)\frac{\varepsilon M}{2^{j_*}},\\
&{\rm \tilde{Tail}}^{p-1}_{p-1}
(u_-;0,8\rho,J'_3)\le{(4\tilde{\gamma})}^{-1}\delta^{1+(p-1)/(2-p)}\exp\left(\frac{-2b_0^{2-p}}{2-p}\right)\frac{\varepsilon M}{2^{j_*}},\\
&{\rm \tilde{Tail}}^{p-1}_{p-1}
(u_-;0,4\rho, t_1,t_1+C_3M_0^{2-p}(4\rho)^{sp})\le{(4\tilde{\gamma})}^{-1}M_0,
\end{align*}
where $\tilde{\gamma}\ge1$,
\begin{align*}
&J'_1=
(0,\delta M^{2-p}\rho^{sp}),\\
&J'_2=\left((1-e^{-\tau_0})\delta M^{2-p}\rho^{sp},(1-e^{-\tau_0-2b_0^{2-p}})\delta M^{2-p}\rho^{sp}\right),\\
&J'_3=\left((1-e^{-(\tau_0+b_0^{2-p}+nb_{j_*}^{2-p})})\delta M^{2-p}\rho^{sp},(1-e^{-(\tau_0+b_0^{2-p}+(n+1)b_{j_*}^{2-p})})\delta M^{2-p}\rho^{sp}\right),
\end{align*}
$t_1=(1-e^{-\tau_1})\delta M^{2-p}\rho^{sp}$ and $t_1+C_3M_0^{2-p}(4\rho)^{sp}=\delta M^{2-p}\rho^{sp}.$
We note that in all the tail alternatives, $\rho$ is the smallest radius and $J'_1$ is the largest time interval. On the other hand,
we have
$$\exp\left(\frac{-2b_0^{2-p}}{2-p}\right)\frac{\varepsilon M}{2^{j_*}}=\exp\left(-\frac{2\varepsilon^{2-p}e^{\tau_0}}{2-p}\right)\frac{\varepsilon M}{2^{j_*}}\ge4\sigma_0\exp\left(-\frac{2e^{\tau_0}}{2-p}\right)M$$ and
$$M_0=\sigma_0Me^{-\frac{\tau_1-\tau_0}{2-p}}\ge\sigma_0\exp\Big(-\frac{2e^{\tau_0}}{2-p}\Big)M,$$
where we used $\sigma_0=\varepsilon2^{-(j_*+2)}$. Set
\begin{align*}
\eta&:=\min\Big\{\delta,8\sigma_0\delta^{1+(p-1)/(2-p)}\exp\left(-\frac{2e^{\tau_0}}{2-p}\right),\frac{\sigma_0\delta^{1+(p-1)/(2-p)}}{\tilde{\gamma}}\exp\left(-\frac{2e^{\tau_0}}{2-p}\right),\frac{\sigma_0}{4\tilde{\gamma}}\exp\left(-\frac{2e^{\tau_0}}{2-p}\right)\Big\}\\
&\,\,=\min\Big\{\delta,\frac{\sigma_0\delta^{1+(p-1)/(2-p)}}{\tilde{\gamma}}\exp\left(-\frac{2e^{\tau_0}}{2-p}\right),\frac{\sigma_0}{4\tilde{\gamma}}\exp\left(-\frac{2e^{\tau_0}}{2-p}\right)\Big\},
\end{align*}
which clearly satisfies that $\eta\le\frac{\sigma_0}{4}\exp\Big(-\frac{2e^{\tau_0}}{2-p}\Big)$ by $\tilde \gamma\ge1$. We see that all the tail assumptions \eqref{Jtt1}, \eqref{TJ2}, \eqref{I} and \eqref{4M} are satisfied, if
$$
{\rm \tilde{Tail}}^{p-1}_{p-1}(u_-;0,\rho,0,\delta M^{2-p}\rho^{sp})\le\eta M.$$
Therefore, the proof is complete.
\end{proof}

\begin{proposition}\label{pro4}
Let $b,\eta$ and $\delta$ be
as
in Proposition $\ref{pro3}$ for $\alpha={1}/{2}$. Then, there exist a non-decreasing sequence $\{m_i\}_{i\ge0}$, a non-increasing sequence $\{M_i\}_{i\ge0}$, $C_0>1$, $\beta\in(0,1)$ with
$\beta<\min\{sp,
\log_{C_0}\big(\frac{2}{2-\eta}\big)\}$, $\varepsilon>0$ and $j_0\in \mathbb{N}$, such that, for any $i\in\mathbb{N}\cup\{0\}$,
\begin{equation}\label{Jp}
m_i\le u\le M_i \quad a.e.\,\,\hbox{in}\,\,\mathcal{Q}_{R_i}(d_i) \quad \hbox{and} \quad M_i-m_i=C_0^{-\beta i}L
\end{equation}
with
\begin{align*}L\ge &2C_0^{\beta j_0}\|u\|_{L^\infty(\mathcal{Q}_{R_0}(d_0))}+\left(\oint^{t_0}_{t_0-(C_0R)^{sp}}\left((C_0R)^{sp}\int_{\R^d\setminus B_{ C_0R }(x_0)}\frac{u_-(t,x)^{p-1}}{|x-x_0|^{d+sp}}\,dx\right)^{1+\varepsilon}\,dt\right)^{1/(1+\varepsilon)}\\
&+\eta^{-1}(C_0^{\beta(2-p)-sp}\delta^{sp})^{1/(2-p)},\end{align*}
where $\mathcal{Q}_{R_0}(d_0):=(t_0-(C_0R)^{sp},t_0)\times B_{C_0R}(x_0)$, and, for $i\in\mathbb{N}$, $\mathcal{Q}_{R_i}(d_i):=(t_0-R_i^{sp},t_0)\times B_{d_iR_i}(x_0)$ with
$$R_i:=C_0^{1-i}R, \quad d_i:=(\eta C_0^{-\beta i}L)^{\frac{p-2}{sp}}\delta.$$
\end{proposition}

\begin{proof}
The strategy of the following proof is similar to that of Proposition \ref{pro2}; however, modifications such as adjusting constants are non-trivial, so we include details for the reader's convenience.
Without loss of generality, we may and do assume that $(t_0,x_0)=(0,0)$. Choose $\beta\in(0,1)$ such that $\beta<\frac{sp}{2-p}$. Then, for all $i\ge1$,
$$d_{i+1}R_{i+1}=(\eta C_0^{-\beta i}L)^{\frac{p-2}{sp}}\delta C_0^{1-i}R\times C_0^{-\beta\frac{p-2}{sp}}C_0^{-1}\le(\eta C_0^{-\beta i}L)^{\frac{p-2}{sp}}\delta C_0^{1-i}R=d_iR_i,$$
and so $\mathcal{Q}_{R_{i+1}}(d_{i+1})\subseteq\mathcal{Q}_{R_i}(d_i)$ for $i\in\mathbb{N}$. Furthermore, by the definition of $L$,
$L\ge\eta^{-1}(C_0^{\beta(2-p)-sp}\delta^{sp})^{1/(2-p)}$, we have $\mathcal{Q}_{R_1}(d_1)\subset\mathcal{Q}_{R_0}(d_0)$. In the following, for simplicity we write
$
(-R_i^{sp},0)\times B_{d_iR_i}(0)=I_i\times B_i$ for all $i\ge1$ and $(-(C_0R)^{sp},0)\times B_{C_0R}(0)=I_0\times B_0$.

The proof will proceed in two steps. First we show \eqref{Jp} holds for $i=0,1,\ldots,j_0$ with $j_0$ to be fixed later, and then use the induction approach to obtain that \eqref{Jp} holds for all $i>j_0$.

(i) For $i=0,1,\ldots,j_0$, set $m_i:=\frac{-C_0^{-\beta i}L}{2}$ and $M_i:=\frac{C_0^{-\beta i}L}{2}$. Since $\mathcal{Q}_{R_i}(d_i)\subset\mathcal{Q}_{R_0}(d_0)$, we have
$$\|u\|_{L^\infty(\mathcal{Q}_{R_i}(d_i))}\le\|u\|_{L^\infty(\mathcal{Q}_{R_0}(d_0))}=\frac{C_0^{-\beta i}}{2}\left(2C_0^{\beta j_0}\|u\|_{L^\infty(\mathcal{Q}_{R_0}(d_0))}\right)
{C_0^{\beta(i-j_0)}}
\le\frac{C_0^{-\beta i}}{2}L=M_i.$$

(ii) Now we fix an integer $j> j_0$, and suppose that the sequences $\{m_i\}_{i\ge0}$ and $\{M_i\}_{i\ge0}$ have been constructed up to $i=j$.
We claim that \eqref{Jp} holds
once we construct $m_{j+1}$ and $M_{j+1}$ appropriately. To do so, we consider the function
$$v:=u-\frac{M_j+m_j}{2}.$$
Then, one of the following two alternatives must hold:
\begin{equation}\label{Al1p}
\Big|\{v(s_j,\cdot)\ge0\}\cap B_j\Big|\ge\frac{1}{2}\Big|B_j\Big|
\end{equation}and
\begin{equation}\label{Al2p}
\Big|\{v(s_j,\cdot)\ge0\}\cap B_j\Big|<\frac{1}{2}\Big|B_j\Big|,
\end{equation}
where
$$s_j:=-\Big(\frac{C_0^{-\beta j}L}{2}\Big)^{2-p}\delta (d_jR_j)^{sp}
=-\delta^{1+sp} (2 \eta)^{p-2} R_j^{sp}
.$$
Without loss of generality, we assume that \eqref{Al1p} holds; otherwise, if \eqref{Al2p} holds, then analogous conclusion follows with $v$ replaced by $-v$. Now, denote $\omega=\frac{C_0^{-\beta j}L}{2}+v$, then \eqref{Al1p} becomes
$$\Big|\Big\{\omega(s_j,\cdot)\ge\frac{C_0^{-\beta j}L}{2}\Big\}\cap B_j\Big|\ge\frac{1}{2}\Big|B_j\Big|.$$
Moreover from the induction hypothesis, we see that $\omega\ge0$  in $\mathcal{Q}_{R_j}(d_j)$ because
$$\omega=v+\frac{C_0^{-\beta j}L}{2}=u-\frac{M_j+m_j}{2}+\frac{C_0^{-\beta j}L}{2}=u-\frac{M_j+m_j}{2}+\frac{M_j-m_j}{2}=u-m_j\ge0.$$
Our goal is to apply Proposition \ref{pro3} with $\mu^-=0$,
$\alpha=\frac{1}{2}$, $M=\frac{C_0^{-\beta j}L}{2}$ and $\rho=d_jR_j$ to deduce that
\begin{equation}\label{omegap}
\omega(t,x)\ge\eta\frac{C_0^{-\beta j}L}{2} \quad \text{a.e.}\,\,(t,x)\in\mathcal{Q}_{R_{j+1}}(d_{j+1}),
\end{equation} where we will take $C_0>1$ large enough so that
$C_0^{-sp}\le\epsilon\delta^{1+sp} (2 \eta)^{p-2}$ with
$\epsilon$ being given in Proposition \ref{pro3}.
Furthermore, \eqref{omegap} implies that
\begin{align*}
u(t,x)&=\frac{M_j+m_j}{2}+v=\frac{M_j+m_j}{2}+\omega-\frac{C_0^{-\beta j}L}{2}\\
&\ge\frac{M_j+m_j}{2}-\frac{C_0^{-\beta j}L}{2}(1-\eta)=M_j-\frac{M_j-m_j}{2}-\frac{C_0^{-\beta j}L}{2}(1-\eta)\\
&=M_j-\frac{C_0^{-\beta j}L}{2}-\frac{C_0^{-\beta j}L}{2}(1-\eta)=M_j-\frac{C_0^{-\beta j}L}{2}(2-\eta)\\
&\ge M_j-C_0^{-\beta(j+1)}L,
\end{align*}
where the fourth equality follows from the induction hypothesis, and in the last inequality we choose $\beta\in(0,1)$ small enough so that $\beta<\log_{C_0}\big(\frac{2}{2-\eta}\big)$. By defining $m_{j+1}:=M_j-C_0^{-\beta(j+1)}L$ and $M_{j+1}:=M_j$, we verify \eqref{Jp} for $m_{j+1}$ and $M_{j+1}$, as desired.

It remains to verify the tail alternative of Proposition \ref{pro3}.
We have, for all $j>j_0$,
\begin{align*}
&{\rm \tilde{Tail}}^{p-1}_{p-1}\big(\omega_-;0,d_jR_j, -\delta^{1+sp} (2 \eta)^{p-2} R_j^{sp},0\big)\\
&={\rm \tilde{Tail}}^{p-1}_{p-1}
\left(\Big(\frac{C_0^{-\beta j}L}{2}+ v\Big)_- ;0,d_jR_j, -\delta^{1+sp}(2\eta)^{p-2} R_j^{sp},0 \right)\\
&=\int^0_{-\delta^{1+sp}(2\eta)^{p-2}  R_j^{sp}}\int_{\R^d\setminus B_j}\frac{\Big(\frac{C_0^{-\beta j}L}{2}+
v(t,x)\Big)_-^{p-1}}{|x|^{d+sp}}\,dx\,dt\\
&=\int^0_{-\delta^{1+sp}(2\eta)^{p-2} R_j^{sp}}\left(\int_{B_0\setminus B_j}+\int_{\R^d\setminus B_0}\right)\frac{\Big(\frac{C_0^{-\beta j}L}{2}+
v(t,x)\Big)_-^{p-1}}{|x|^{d+sp}}\,dx\,dt\\
&=: J_1+J_2.
\end{align*}
Below, we estimate $J_1$ and $J_2$. Take
$x\in B_0\setminus B_j$, and let $l\in\{0,1,\ldots,j-1
\}$ be the unique integer for which $x\in B_{l}\setminus B_{l+1}$. In particular, $|x|\ge d_{l+1}R_{l+1}$ implies that
$$\frac{|x|}{C_0R}\frac{(\eta L)^{\frac{2-p}{sp}}}{\delta}\ge C^{-(l+1)\Gamma}_0,$$
where $\Gamma:=\frac{\beta(p-2)}{sp}+1>0$ (by the fact $\beta<\frac{sp}{2-p}$). By the monotonicity of $\{M_j\}_{j\ge0}$ and $\{m_j\}_{j\ge0}$ and by the induction hypothesis, for a.e. $t\in I_j$, we have
\begin{align*}
v(t,x)&=u(t,x)-\frac{M_j+m_j}{2}\ge m_l-\frac{M_j+m_j}{2}\ge m_l-M_l+M_j-\frac{M_j+m_j}{2}\\
&=-(M_l-m_l)+\frac{M_j-m_j}{2}=-\left[C_0^{-\beta l}-\frac{C_0^{-\beta j}}{2}\right]L\ge-\left[\left(\frac{C_0^{\Gamma-1}|x|}{R}\frac{(\eta L)^{\frac{2-p}{sp}}}{\delta}\right)^{\frac{\beta}{\Gamma}}-\frac{C_0^{-\beta j}}{2}\right]L.
\end{align*}
Hence,
$$\Big(\frac{C_0^{-\beta j}L}{2}+v(t,x)\Big)_-^{p-1}\le\left[\left(\frac{C_0^{\Gamma-1}|x|}{R}\frac{(\eta L)^{\frac{2-p}{sp}}}{\delta}\right)^{\frac{\beta}{\Gamma}}-C_0^{-\beta j}\right]^{p-1}L^{p-1},$$
and moreover
\begin{align*}
\int_{B_0\setminus B_j}\frac{\Big(\frac{C_0^{-\beta j}L}{2}+
v(t,x)\Big)_-^{p-1}}{|x|^{d+sp}}\,dx
&\le\int_{\R^d\setminus B_j}\frac{\Big[\Big(\frac{C_0^{\Gamma-1}|x|}{R}\frac{(\eta L)^{\frac{2-p}{sp}}}{\delta}\Big)^{\frac{\beta}{\Gamma}}-C_0^{-\beta j}\Big]^{p-1}L^{p-1}}{|x|^{d+sp}}\,dx\\
&=C_0^{\beta(1- j)(p-1)}L^{p-1}\int_{\R^d\setminus B_j}\frac{\Big(\Big(\frac{|x|}{d_jR_j}\Big)^{\frac{\beta}{\Gamma}}-C_0^{-\beta}\Big)^{p-1}}{|x|^{d+sp}}\,dx\\
&=\frac{C_0^{ \beta (1-j)(p-1)}L^{p-1}}{(d_jR_j)^{sp}}\int_{\{|y|\ge1\}}\frac{(|y|^{\frac{\beta}{\Gamma}}-C_0^{-\beta})^{p-1}}{|y|^{d+sp}}\,dy\\
&=\frac{C_1C_0^{\beta (1-j)(p-1)}L^{p-1}}{(d_jR_j)^{sp}}\int_1^\infty\frac{(\rho^{\frac{\beta}{\Gamma}}-C_0^{-\beta})^{p-1}}{\rho^{1+sp}}\,d\rho
\end{align*}
for some $C_1>0$, where
in the second equality we used the
change of variable: $y:=\frac{x}{d_jR_j}$. Thus,
\begin{align*}
J_1&\le C_1\delta^{1+sp} (2 \eta)^{p-2} d_j^{-sp}C_0^{\beta (1- j)(p-1)}L^{p-1}\int_1^\infty\frac{(\rho^{\frac{\beta}{\Gamma}}-C_0^{-\beta})^{p-1}}{\rho^{1+sp}}\,d\rho\\
&=C_1C_0^{\beta(p-1)}2^{p-2}\delta\left(\int_1^\infty\frac{(\rho^{\frac{\beta}{\Gamma}}-C_0^{-\beta})^{p-1}}{\rho^{1+sp}}\,d\rho\right) C_0^{-\beta j}L.
\end{align*} In particular, one can choose $\beta\in (0,1)$ small enough so that
$$C_1
C_0^{\beta(p-1)}2^{p-2}\delta\left(\int_1^\infty\frac{(\rho^{\frac{\beta}{\Gamma}}-C_0^{-\beta})^{p-1}}{\rho^{1+sp}}\,d\rho\right)\le \frac{\eta}{6}$$ and so
$$J_1\le \frac{\eta}{3} \frac{C_0^{-\beta j}L}{2}.$$

Using the monotonicity of $\{m_{j}\}_{j\ge0}$ and $\{M_{j}\}_{j\ge0}$, we have for all $j>j_0$,
$$(M_j-m_j)+2m_{j_0}\le M_j+m_j\le2M_{j_0}-(M_j-m_j).$$
Recalling $m_{j_0}=-\frac{C_0^{-\beta j_0}L}{2}$ and $M_{j_0}=\frac{C_0^{-\beta j_0}L}{2}$ gives
$$-(C_0^{-\beta j_0}-C_0^{-\beta j})L\le M_j+m_j\le(C_0^{-\beta j_0}-C_0^{-\beta j})L,$$
and so
\begin{align*}
\Big(\frac{C_0^{-\beta j}L}{2}+v(t,x)\Big)_-^{p-1}&=\Big(\frac{C_0^{-\beta j}L}{2}+u(t,x)-\frac{M_j+m_j}{2}\Big)_-^{p-1}\\
&\le\Big(u_-(t,x)+\frac{C_0^{-\beta j_0}L}{2}\Big)^{p-1}\le2^{p-1}u_-(t,x)^{p-1}+ C_0^{-(p-1)\beta j_0} L^{p-1}.
\end{align*}
Thus,
\begin{align*}
J_2&\le \int^0_{-\delta^{1+sp} (2 \eta)^{p-2} R_j^{sp}} \int_{\R^d\setminus B_0} \left(\frac{C_0^{-(p-1)\beta j_0} L^{p-1}}{|x|^{d+sp}}+\frac{2^{p-1}u_-(t,x)^{p-1}}{|x|^{d+sp}}\right)\,dx\,dt\\
&\le C_2\delta^{1+sp} (2 \eta)^{p-2}C_0^{-(p-1)\beta j_0}C_0^{-j(sp-\beta)}L^{p-2}
C_0^{-\beta j}L \\
&\quad +2^{p-1}(\delta^{1+sp} (2 \eta)^{p-2} R_j^{sp})^{\varepsilon/(1+\varepsilon)} \left(\int^0_{-\delta^{1+sp} (2 \eta)^{p-2} R_j^{sp}} \left(\int_{\R^d\setminus B_0} \frac{u_-(t,x)^{p-1}}{|x|^{d+sp}} \,dx\right)^{1+\varepsilon}\,dt\right)^{1/(1+\varepsilon)}\\
&\le
C_22^{p-2}\delta C_0^{-(p-1)\beta j_0-\beta(2-p)+sp}C_0^{-j(sp-\beta)}C_0^{-\beta j}L\\
&\quad+
C_3C_0^{-j(sp\varepsilon/(1+\varepsilon)-\beta)} C_0^{-\beta j}(C_0R)^{sp\varepsilon/(1+\varepsilon)}
\left(\int^0_{-(C_0R)^{sp}} \left(\int_{\R^d\setminus B_0} \frac{u_-(t,x)^{p-1}}{|x|^{d+sp}} \,dx\right)^{1+\varepsilon}\,dt\right)^{1/(1+\varepsilon)}\\
&=:J_{21}+J_{22},
\end{align*}
where the second inequality follows from the H\"{o}lder inequality, and  in the last inequality we used the fact that
$L\ge\eta^{-1}(C_0^{\beta(2-p)-sp}\delta^{sp})^{1/(2-p)}$ as well as that we take $C_0>1$ large enough.  By taking $j_0$ large enough and $\beta\in(0,1)$ small enough so that $\beta<sp$, we can see that for all $j>j_0$,
$$J_{21}\le \frac{\eta}{3} \frac{C_0^{-\beta j}L}{2}.$$ On the other hand, let $\varepsilon>0$ and $\beta\in (0,1)$ small enough so that
$$\beta<sp\varepsilon/(1+\varepsilon).$$ Then, by the definition of $L$,
we get
that for $j_0$ large enough and for all $j>j_0$,
$$J_{22}\le \frac{\eta}{3} \frac{C_0^{-\beta j}L}{2}.$$

Putting all the inequalities together, we can find $\beta\in (0,1)$ small enough and $j_0$ large enough so that for all $j>j_0$,
\begin{align*}
&{\rm \tilde{Tail}}^{p-1}_{p-1}
\big(\omega_-;0,d_jR_j, -\delta^{1+sp} (2 \eta)^{p-2} R_j^{sp},0\big)
\le J_1+J_{21}+J_{22}=\eta\frac{C_0^{-\beta j}L}{2}=:\eta M.
\end{align*}
The proof is complete.
\end{proof}

\begin{proof}[Proof of Theorem $\ref{main1}$
{\rm(ii):} $1<p<2$]
The following proof is similar to that of Theorem $\ref{main1}$ (i);
we include details for the reader's convenience.
We use the notation in Proposition \ref{pro4}. As a direct  consequence of Proposition \ref{pro4}, we have the following
decay of oscillations:
\begin{equation}\label{oscp}
\mathop{\ess\, \text{osc}}_{\mathcal{Q}_{R_i}(t_0,x_0,d_i)}u
:=\mathop{\ess\sup}_{\mathcal{Q}_{R_i}(t_0,x_0,d_i)}u-\mathop{\ess\inf}_{\mathcal{Q}_{R_i}(t_0,x_0,d_i)}u\le C_0^{-\beta i}L,
\end{equation}
where $d_i$ (for $i\ge1$) depends on the constant $L$ that
satisfies
$L\ge L(t_0,x_0)$ with
\begin{align*}L(t_0,x_0):=&2C_0^{\beta j_0}\|u\|_{L^\infty(\mathcal{Q}_{R_0}(d_0))}+\left(\oint^{t_0}_{t_0-(C_0R)^{sp}}\left((C_0R)^{sp}\int_{\R^d\setminus B_{ C_0R }(x_0)}\frac{u_-(t,x)^{p-1
}}{|x-x_0|^{d+sp}}\,dx\right)^{1+\varepsilon}\,dt\right)^{1/(1+\varepsilon)}\\
&+\eta^{-1}(C_0^{\beta(2-p)-sp}\delta^{sp})^{1/(2-p)}.\end{align*}

To prove the desired assertion, without loss of generality we will assume that $(t_0, x_0)=(0,0)$. Consider any two points $(t_1,x_1),(t_2,x_2)\in\mathcal{Q}_{R_0}(d_0)/2:=(-(C_0R)^{sp}/2,0)\times B_{C_0R/2}$
with $t_1\neq t_2$ and $x_1\neq x_2$. Then there exist non-negative integers $n$ and $m$ such that
\begin{equation}\label{Rnp}
d_{n+1}R_{n+1}<|x_1-x_2|\le d_nR_n \qquad \text{and} \qquad R_{m+1}^{sp}<|t_1-t_2|\le R_m^{sp}.
\end{equation} Here we understand that $d_0=1$ and $R_0=C_0R$, and emphasize again that $d_i$ (for all $i\ge1$) depends on the constant $L$ which is fixed later.
From the first inequality in \eqref{Rnp}, we get
$$|x_1-x_2|>(\eta C_0^{-\beta(n+1)}L)^{(p-2)/(sp)}\delta C_0^{-n}R,$$
which implies that
\begin{equation}\label{firp}
C_*L^{\frac{2-p}{sp}}\frac{|x_1-x_2|}{R}>
C_0^{-\beta(n+1)(p-2)/(sp)}C_0^{-n}\ge C_0^{-n},
\end{equation}
where $C_*=\eta^{\frac{2-p}{sp}}\delta^{-1},$
and the last inequality follows from $ C_0^{-\beta(n+1)({p-2})/({sp})}\ge1$ since $p<2$
and $C_0>1$.
On the other hand, from the second inequality in \eqref{Rnp}, we deduce
\begin{equation}\label{secp}
\frac{|t_1-t_2|^{\frac{1}{sp}}}{R}>C_0^{-m}.
\end{equation}

Furthermore, for any $x\in {B_{C_0R}(x_1)}^c$ and $x_1\in B_{C_0R/2}$,
$$\frac{|x|}{|x-x_1|}\le 1+\frac{|x_1|}{|x-x_1|}\le 1+\frac{C_0R/2}{C_0R}\le 2 \quad\text{and}\quad |x|\ge |x-x_1|-|x_1|\ge C_0R-C_0R/2=C_0R/2,$$
which implies that
\begin{equation}\label{ee:add}
\begin{split}
L(t_1,x_1)&:=2C_0^{\beta j_0}\|u\|_{L^\infty ((t_1-(C_0R)^{sp},t_1)\times B_{C_0R}(x_1))}+\eta^{-1}(C_0^{\beta(2-p)-sp}\delta^{sp})^{1/(2-p)}\\
&\quad+\left(\oint_{t_1-(C_0R)^{sp}}^{t_1}\left((C_0R)^{sp}\int_{\R^d\setminus B_{C_0R}(x_1)}\frac{u_-(t,x)^{p-1}}{|x-x_1|^{d+sp}}\,dx\right)^{1+\varepsilon}\,dt\right)^{1/(1+\varepsilon)}\\
&\le 2C_0^{\beta j_0}\|u\|_{L^\infty ((-2(C_0R)^{sp},0)\times B_{2C_0R})}+\eta^{-1}(C_0^{\beta(2-p)-sp}\delta^{sp})^{1/(2-p)}\\
&\quad+\left(\oint_{-2(C_0R)^{sp}}^{0}\left((C_0R)^{sp}\int_{\R^d\setminus B_{C_0R/2}}\frac{
u_-(t,x)^{p-1}}{|x|^{d+sp}}\,dx\right)^{1+\varepsilon}\,dt\right)^{1/(1+\varepsilon)}=:L.
\end{split}
\end{equation}

Below we
take the constant $L$ as that defined in the right hand side of the inequality above. Let $k:=\min\{n,m\}$.
Putting \eqref{Rnp}, \eqref{firp}, \eqref{secp} and \eqref{ee:add} together, we obtain
\begin{align*}
|u(t_1,x_1)-u(t_2,x_2)|&\le\mathop{\ess\, \text{osc}}_{\mathcal{Q}_{R_k}(t_1,x_1,d_k)}u\le\max\{C_0^{-\beta n},C_0^{-\beta m}\}L\le(C_0^{-n}+C_0^{-m})^\beta L\\
&\le \frac{L}{R^\beta}\left(C_*L^{\frac{2-p}{sp}}|x_1-x_2|+|t_1-t_2|^{\frac{1}{sp}}\right)^\beta,
\end{align*}
where in the second inequality we used \eqref{oscp}.
The proof is complete. \end{proof}

\section{Appendix}\label{section4}
Fix $p>1$. For a domain $D\subset\R^d$, the fractional Sobolev space $W^{s,p}(D)$ consists of all functions $f\in L^p(D)$ such that
$$[f]_{W^{s,p}(D)}:=\left(\int_D\int_D\frac{|f(x)-f(y)|^p}{|x-y|^{d+sp}}\,dx\,dy\right)^{{1}/{p}}<\infty.$$
The norm of $f\in W^{s,p}(D)$ is given by
$$\|f\|_{W^{s,p}(D)}=[f]_{W^{s,p}(D)}+\|f\|_{L^p(D)},$$
where $\|f\|_{L^p(D)}=\left(\displaystyle\int_D |f(x)|^p\,d x\right)^{1/p}.$
For $q\in (0,\infty]$, the parabolic Sobolev space $L^q(0,T;W^{s,p}(D))$ is the set of measurable functions on $(0,T)\times D$ such that
$$\|f\|_{L^q(0,T;W^{s,p}(D))}=\left(\int_0^T\|f(t,\cdot)\|^q _{W^{s,p}(D)}\,dt\right)^{{1}/{(q\vee1)}}<\infty.$$
For $q,\delta>0$,
$$L_{q,\delta}(\R^d):=\left\{f\in L^q_{{\rm loc}}(\R^d): \int_{\R^d}\frac{|f(x)|^q}{(1+|x|)^{d+\delta}}\,dx<\infty\right\}.$$

For $q>0$, any interval $I\subset\R$ and subset $D\subset\R^d$,  we define the following function spaces:
\begin{align*}
&L^p_{{\rm loc}}(I;W^{s,p}_{{\rm loc}}(D)):=\Big\{f: \text{for any sub-interval $[t_1,t_2]\subset I$ and any compact set $D'\subset D$},\\
&\qquad\qquad\qquad\qquad\qquad\quad\int_{t_1}^{t_2}\|f(t,\cdot)\|^p _{W^{s,p}(D')}\,dt<\infty \Big\},
\\
&L^2_{{\rm loc}}(I;L^2_{{\rm loc}}(D)):=\Big\{f: \text{for any sub-interval $[t_1,t_2]\subset I$ and any compact set $D'\subset D$}, \\
&\qquad\qquad\qquad\qquad\qquad\quad
\int_{t_1}^{t_2}\|f(t,\cdot)\|^2_{L^2(D')}\,dt <\infty\Big\},
\\
&L^\infty_{{\rm loc}}(I;L_{p-1,sp}(\R^d)):=\Big\{f: \text{for any sub-interval $[t_1,t_2]\subset I$},
\,\,\mathop{\ess\sup}_{t_1<t<t_2}\int_{\R^d}\frac{|f(t,x)|^{p-1}}{(1+|x|)^{d+sp}}\,dx<\infty \Big\},
\\
&L^q_{{\rm loc}}(I;L_{p-1,sp}(\R^d)):=\Big\{f: \text{for any sub-interval $[t_1,t_2]\subset I$},
\,\,\int_{t_1}^{t_2}\left(\int_{\R^d}\frac{|f(t,x)|^{p-1}}{(1+|x|)^{d+sp}}\,dx\right)^{q/(p-1)}\,dt<\infty \Big\},\\
&W^{1,2}_{{\rm loc}}(I;L_{{\rm loc}}^2(D)):=\Big\{f: \text{for any sub-interval $[t_1,t_2]\subset I$ and any compact set $D'\subset D$},\\
&\qquad\qquad\qquad\qquad\qquad\quad\int_{t_1}^{t_2}\|f(t,\cdot)\|^2 _{L^2(D')}\,dt<\infty\,\,\text{and}\,\,\int_{t_1}^{t_2}\|\partial_tf(t,\cdot)\|^2 _{L^2(D')}\,dt<\infty\Big\},
\\
&L^p_{{\rm loc}}(I;W^{s,p}_0(D)):=\Big\{f: \text{for any sub-interval $[t_1,t_2]\subset I$ and $t\in[t_1,t_2]$},\\
&\qquad\qquad\qquad\qquad\qquad\quad f(t,\cdot)=0\,\,\text{in}\,\,\R^d\setminus D \,\,\text{and}\,\,\int_{t_1}^{t_2}\|f(t,\cdot)\|^p _{W^{s,p}(D)}\,dt<\infty \Big\},
\end{align*}
where
$$W^{s,p}_0(D):=\Big\{f\in W^{s,p}(\R^d):f=0\,\,\text{in}\,\,\R^d\setminus D\Big\}.$$

The following Sobolev-type inequality is taken from \cite[Lemma 2.3]{DZZ}.
\begin{lemma}\label{Sob} Let $s\in (0,1)$ and $p>1$. Then there is a constant $C>0$ such that for all $0<t_1<t_2$, $x_0\in \R^d$,  $r>0$ and $f\in L^p(t_1,t_2;W^{s,p}(B_r(x_0)))$,
\begin{align*}
&\int_{t_1}^{t_2}\oint_{B_r(x_0)}|f(t,x)|^{p(1+2s/d)}\,dx\,dt\\
&\le C\left(r^{sp}\int_{t_1}^{t_2}\int_{B_r(x_0)}\oint_{B_r(x_0)} \frac{|f(t,x)-f(t,y)|^p}{|x-y|^{d+sp}}\,dx\,dy\,dt+\int_{t_1}^{t_2}\oint_{B_r(x_0)}|f(t,x)|^p\,dx\,dt\right)\\
&\quad\quad\times \left(\sup_{t_1\le t\le t_2}\oint_{B_r(x_0)}|f(t,x)|^2\,dx\right)^{sp/d}.
\end{align*}
\end{lemma}

We need the following elementary iteration lemma, see, e.g., \cite[Lemma 7.1, p.\ 220]{Gi}.
\begin{lemma}\label{iteration}
Let $\{Y_j\}_{j\ge0}$ be a sequence of real positive numbers satisfying
$$Y_{j+1}\le c_0b^jY_j^{1+\beta},\quad j\ge0,$$
for some constants $c_0>0$, $b>1$ and $\beta>0$. If
$$Y_0\le c_0^{-1/\beta}b^{-1/\beta^2},$$
then
\begin{equation}\label{a-}
Y_j\le b^{-j/\beta}Y_0,\quad j\ge0,
\end{equation}
which in particular yields $\lim_{j\rightarrow\infty}Y_j=0$.
\end{lemma}

For any $\epsilon>0$, let $\zeta_\epsilon(s)$ be a standard mollifier with support in
$[-\epsilon,\epsilon]$.  Given $f:\R\times \R^d\rightarrow\R$, we define
\begin{equation}\label{epsilon}
f_\epsilon(t,x)=\int_{\R}f(s,x)\zeta_\epsilon(t-s)\,ds,\quad (t,x)\in \R\times \R^d.
\end{equation}
Now, we prove the so-called Caccioppoli-type inequality.
\begin{lemma}\label{2}
Let $p\in(1,\infty)$ and $\xi>0$. Let  $x_0\in \R^d$, $r>0$, $\tau_1<\tau_2$ and $h>0$, and set $B_r:=B_r(x_0)$. Assume that $u$ is a
locally bounded weak
sub-solution $($resp.\ super-solution$)$
to \eqref{EQ}
in $(\tau_1-h,\tau_2)\times B_r$. Let $k$ be a non-negative and differentiable function on $(\tau_1-h,\tau_2)$, $v(t,x)=u(t,x)-k(t)$ and
$w_\pm=v_\pm^{(p-1+\xi)/p}$. There exist constants $C_1, C_2, C_3>0$, such that for every non-negative and bounded
function $\phi(t,x)$ which is compactly supported in $B_r$ for any fixed $t\in(\tau_1-h,\tau_2)$ and is differentiable with respect to $t$ for any fixed $x\in B_r$, it holds that for all $t_2\in (\tau_1-h, \tau_2)$,
\begin{align*}
& \int_{\tau_1-h}^{t_2}\int_{B_r\times B_r}|w_\pm(t,x)-w_\pm(t,y)|^p\min\{\phi(t,x)^p,\phi(t,y)^p\}J(t;x,y)\,dx\,dy\,dt\\
&\quad+ \int_{\tau_1-h}^{t_2}\int_{B_r}v_\pm(t,x)^{\xi}\phi(t,x)^p\left(\int_{B_r}v_\mp(t,y)^{p-1} J(t;x,y)\,dy\right)\,dx\,dt+C_1\int_{B_{r}} v_\pm(t,x)^{1+\xi}\phi(t,x)^p\,dx\bigg|_{\tau_1-h}^{t_2}\\
&\le C_2\int_{\tau_1-h}^{t_2}\int_{B_r}v_\pm(t,x)^{\xi} \phi(t,x)^p\left(\int_{B_r^c}v_\pm(t,y)^{p-1}J(t;x,y)\,dy\right)\,dx\,dt\\
&\quad+C_3\int_{\tau_1-h}^{t_2}\int_{B_r\times B_r}\max\{w_\pm(t,x)^p,w_\pm(t,y )^p\}|\phi(t,x)-\phi(t,y)|^pJ(t;x,y)\,dx\,dy\,dt\\
&\quad\mp C_2\int_{\tau_1-h}^{t_2}\int_{B_{r}}
v_\pm(t,x)^\xi\phi(t,x)^pk'(t)\,dx\,dt+C_1\int_{\tau_1-h}^{t_2}\int_{B_{r}}|\partial_t\phi(t,x)^{p}|v_\pm(t,x)^{1+\xi}\,dx\,dt.
\end{align*}
Here, the notation $\pm$ means that when $u$ is
a weak
sub-solution $($resp.\ super-solution$)$, the statement above holds with $w_+$ $($resp.\ $w_-$$)$, and
double signs correspond in the whole terms.
\end{lemma}

\begin{proof}
(1) Suppose that $u$ is a locally bounded weak sub-solution in $(\tau_1-h,\tau_2)\times B_r$. For $\tau_1-h<t_1<t_2<\tau_2$,  let $\theta_j(t)\in C^\infty_c(\tau_1-h,\tau_2)$ be a smooth approximation of $\I_{(t_1,t_2)}(t)$ as $j\rightarrow\infty$. Let $q=1-\xi<1$. Then the following is an admissible test function
$$\eta_{j,\epsilon}(t,x):=((v_+)_\epsilon^{1-q}\phi^p\theta_j)_\epsilon(t,x),$$
where the subscript $\epsilon$ on the right hand side denotes the mollification with respect to the time variable in the sense of \eqref{epsilon}. Indeed, on the one hand,
\begin{align*}
&|(v_+)_\epsilon^{1-q}\phi^p(t,x)-(v_+)_\epsilon^{1-q}\phi^p(t,y)|^p\\
&\le\Big(|(v_+)_\epsilon^{1-q}(t,x)|
\cdot
|\phi^p(t,x)-\phi^p(t,y)|+|\phi^p(t,y)|
\cdot
|(v_+)_\epsilon^{1-q}(t,x)-(v_+)_\epsilon^{1-q}(t,y)|\Big)^p\\
&\le2^{p-1}|(v_+)_\epsilon^{1-q}(t,x)|^p|\phi^p(t,x)-\phi^p(t,y)|^p+2^{p-1}|\phi^p(t,y)|^p|(v_+)_\epsilon^{1-q}(t,x)-(v_+)_\epsilon^{1-q}(t,y)|^p\\
&\le2^{p-1}|(v_+)_\epsilon^{1-q}(t,x)|^p(p\|\phi(t,\cdot)\|_\infty^{p-1})^p|\phi(t,x)-\phi(t,y)|^p\\
&\quad+2^{p-1}|\phi^p(t,y)|^p((1-q)C_t^{-q})^p|(v_+)_\epsilon(t,x)-(v_+)_\epsilon(t,y)|^p,
\end{align*}
where in the last inequality we used the mean value theorem twice and $\min\{(v_+)_\epsilon
(t,x),(v_+)_\epsilon(t,y)\}\le C_t\le \max\{(v_+)_\epsilon(t,x),(v_+)_\epsilon(t,y)\}$. Thanks to the local boundedness of $(v_+)_\epsilon$ and the compact-supported property of $\phi$, we derive that
$$[\eta_{j,\epsilon}]_{W^{s,p}(B_r)}=\left(\int_{B_r}\int_{B_r}\frac{|\eta_{j,\epsilon}(t,x)-\eta_{j,\epsilon}(t,y)|^p}{|x-y|^{d+sp}}\,dx\,dy\right)^{{1}/{p}}<\infty,$$
and so, by \eqref{epsilon} and the fact that $\eta_{j,\epsilon}(t,x)$ is smooth with respect to $t$,  $\eta_{j,\epsilon}\in L^p_{{\rm loc}}((\tau_1-h,\tau_2);W^{s,p}_0(B_r))$.  On the other hand, since $\eta_{j,\epsilon}(t,x)$ is compactly supported in $B_r$ (due to $\supp[\phi(t,\cdot)
]\subset B_r$) and is smooth with respect to $t$,
$$\int_{t_1}^{t_2}\|\eta_{j,\epsilon}(t,\cdot)\|^2 _{L^2(B_r)}\,dt<\infty \quad \text{and} \quad \int_{t_1}^{t_2}\|\partial_t\eta_{j,\epsilon}(t,\cdot)\|^2 _{L^2(B_r)}\,dt<\infty.$$
Thus, $\eta_{j,\epsilon}\in W^{1,2}_{{\rm loc}}((\tau_1-h,\tau_2);L_{{\rm loc}}^2(B_r))$.

Using the fact that $u(t,x)$ is a weak sub-solution in
$(\tau_1-h,\tau_2)\times B_{r}$, we find
\begin{align*}0
&\ge \int_{\tau_1-h}^{\tau_2}
\mathcal{E}_t(u(t,x),\eta_{j,\epsilon}(t,x))\,dt+\int_{B_r} u(t,x)\eta_{j,\epsilon}(t,x)\,dx\Big|^{\tau_2}_{\tau_1-h}-\int_{\tau_1-h}^{\tau_2}\int_{B_r} u(t,x)\partial_t\eta_{j,\epsilon}(t,x)\,dx\,dt\\
&=\frac{1}{2}\int_{\tau_1-h}^{\tau_2}\int_{B_{r}\times B_{r}} |u(t,x)-u(t,y)|^{p-2} (u(t,x)-u(t,y))\\
&\qquad\qquad\qquad\qquad\times(((v_+)_\epsilon^{1-q}\phi^p\theta_j)_\epsilon(t,x)-((v_+)_\epsilon^{1-q}\phi^p\theta_j)_\epsilon(t,y)) J(t;x,y)\,dx\,dy\,dt\\
&\quad +\int_{\tau_1-h}^{\tau_2}\int_{B_{r}\times B^c_{r}}|u(t,x)-u(t,y)|^{p-2} (u(t,x)-u(t,y))((v_+)_\epsilon^{1-q}\phi^p\theta_j)_\epsilon(t,x)J(t;x,y)\,dx\,dy\,dt\\
&\quad+\int_{B_r} u(t,x)\eta_{j,\epsilon}(t,x)\,dx\Big|^{\tau_2}_{\tau_1-h}
-\int_{\tau_1-h}^{\tau_2}\int_{B_r} u(t,x)\partial_t\eta_{j,\epsilon}(t,x)\,dx\,dt\\
&=:I^{j,\epsilon}_1+I^{j,\epsilon}_2+I^{j,\epsilon}_3+I^{j,\epsilon}_4.
\end{align*}

For $I^{j,\epsilon}_4$,
\begin{align*}
I^{j,\epsilon}_4&=-\int_{\tau_1-h}^{\tau_2}\int_{B_r} u(t,x)\partial_t\left(\int_{\R}(v_+)_\epsilon(s,x)^{1-q}\phi(s,x)^p\theta_j(s)\zeta_\epsilon(t-s)\,ds\right)\,dx\,dt\\
&=-\int_{\tau_1-h}^{\tau_2}\int_{B_r} u(t,x)\int_{\R}(v_+)_\epsilon(s,x)^{1-q}\phi(s,x)^p\theta_j(s)\partial_t\zeta_\epsilon(t-s)\,ds\,dx\,dt\\
&=\int_{\tau_1-h}^{\tau_2}\int_{B_r}\int_{\R}(v_+)_\epsilon(s,x)^{1-q}\phi(s,x)^p\theta_j(s)u(t,x)\partial_s\zeta_\epsilon(s-t)\,ds\,dx\,dt\\
&=\int_{B_r}\int_{\R}(v_+)_\epsilon(s,x)^{1-q}\phi(s,x)^p\theta_j(s)\left(\int_{\tau_1-h}^{\tau_2}u(t,x)\partial_s\zeta_\epsilon(s-t)\,dt\right)\,ds\,dx\\
&=\int_{B_r}\int_{\R}(v_+)_\epsilon(s,x)^{1-q}\phi(s,x)^p\theta_j(s)\partial_s\left(\int_{\R}u(t,x)\zeta_\epsilon(s-t)\,dt\right)\,ds\,dx\\
&=\int_{B_r}\int_{\R}(v_+)_\epsilon(s,x)^{1-q}\phi(s,x)^p\theta_j(s)\partial_su_\epsilon(s,x)\,ds\,dx\\
&=\int_{\tau_1-h}^{\tau_2}\int_{B_r}(v_+)_\epsilon(t,x)^{1-q}\phi(t,x)^p\theta_j(t)\partial_tu_\epsilon(t,x)\,dx\,dt,
\end{align*}
where in the third equality we used the fact that the mollifier $\zeta_\epsilon$ is an even function, in the fifth equality we used the facts that $\supp[\zeta_\epsilon]=[-\epsilon,\epsilon]$, $\supp[\theta_j]\subseteq[t_1,t_2]$, and so $u(t,x)\zeta_\epsilon(s-t)=0$ when $t\in(\tau_1-h,\tau_2)^c\subset(t_1-\epsilon,t_2-\epsilon)^c$
with $\epsilon$ small enough, and in the second to last equality we used the fact that $\supp[\theta_j]\subseteq[t_1,t_2]\subset(\tau_1-h,\tau_2)$. We also note that the dominated convergence theorem was used in the second and fifth equalities as well. As
$j\rightarrow \infty$, we have
\begin{align*}
I^{j,\epsilon}_4
&\rightarrow\int_{t_1}^{t_2}\int_{B_r}(v_+)_\epsilon(t,x)^{1-q}\phi(t,x)^p\partial_tu_\epsilon(t,x)\,dx\,dt\\
&=\int_{t_1}^{t_2}\int_{B_r}(v_+)_\epsilon(t,x)^{1-q}\phi(t,x)^p\partial_t(v_+)_\epsilon(t,x)\,dx\,dt+\int_{t_1}^{t_2}\int_{B_r}(v_+)_\epsilon(t,x)^{1-q}\phi(t,x)^pk'_\epsilon(t)\,dx\,dt\\
&=:I^{\epsilon}_4,
\end{align*}
where in the first equality we used the fact that for $u(t,x)\ge k(t)$, $\partial_tu_\epsilon(t,x)=\partial_t[u(t,x)-k(t)+k(t)]_\epsilon=\partial_t[v_+(t,x)+k(t)]_\epsilon=\partial_t(v_+)_\epsilon(t,x)+k_\epsilon'(t)$.
Then, the integration by parts yields
\begin{align*}
I^{\epsilon}_4
&=\frac{1}{2-q}\int_{t_1}^{t_2}\int_{B_r}\partial_t(v_+)_\epsilon(t,x)^{2-q}\phi(t,x)^p\,dx\,dt+\int_{t_1}^{t_2}\int_{B_r}(v_+)_\epsilon(t,x)^{1-q}\phi(t,x)^pk'_\epsilon(t)\,dx\,dt\\
&=\frac{1}{2-q}\int_{B_r}(v_+)_\epsilon(t,x)^{2-q}\phi(t,x)^p\,dx\Big|^{t_2}_{t_1}-\frac{1}{2-q}\int_{t_1}^{t_2}\int_{B_r}\partial_t\phi(t,x)^p(v_+)_\epsilon(t,x)^{2-q}\,dx\,dt\\
&\quad+\int_{t_1}^{t_2}\int_{B_r}(v_+)_\epsilon(t,x)^{1-q}\phi(t,x)^pk'_\epsilon(t)\,dx\,dt\\
&\rightarrow\frac{1}{2-q}\int_{B_r}v_+(t,x)^{2-q}\phi(t,x)^p\,dx\Big|^{t_2}_{t_1}-\frac{1}{2-q}\int_{t_1}^{t_2}\int_{B_r}\partial_t\phi(t,x)^pv_+(t,x)^{2-q}\,dx\,dt\\
&\quad+\int_{t_1}^{t_2}\int_{B_r}v_+(t,x)^{1-q}\phi(t,x)^pk'(t)\,dx\,dt,
\end{align*}
as $\epsilon\rightarrow0$.

For $I^{j,\epsilon}_3$, since
$$\eta_{j,\epsilon}(t,x)=\int_{\R}(v_+)_\epsilon(s,x)^{1-q}\phi(s,x)^p\theta_j(s)\zeta_\epsilon(t-s)\,ds,$$ $\supp[\zeta_\epsilon]=[-\epsilon,\epsilon]$ and $\supp[\theta_j]\subseteq[t_1,t_2]$,  $\supp[\eta_{j,\epsilon}(\cdot,x)
]\subseteq[t_1-\epsilon,t_2+\epsilon]\subset(\tau_1-h,\tau_2)$ for small $\epsilon>0$. Thus, $I^{j,\epsilon}_3=0$.

For $I^{j,\epsilon}_2$, as $j\rightarrow \infty$ and $\epsilon\rightarrow 0$, we have
\begin{align*}
I^{j,\epsilon}_2&\rightarrow\int_{t_1}^{t_2}\int_{B_{r}\times B^c_{r}}|u(t,x)-u(t,y)|^{p-2} (u(t,x)-u(t,y))v_+(t,x)^{1-q}\phi(t,x)^pJ(t;x,y)\,dx\,dy\,dt\\
&\ge-\int_{t_1}^{t_2}\int_{B_{r}\times B^c_{r}}(u(t,y)-u(t,x))_+^{p-1}v_+(t,x)^{1-q}\phi(t,x)^pJ(t;x,y)\,dx\,dy\,dt\\
&\ge-\int_{t_1}^{t_2}\int_{B_{r}\times B^c_{r}}(u(t,y)-k(t))_+^{p-1}v_+ (t,x)^{1-q}\phi(t,x)^pJ(t;x,y)\,dx\,dy\,dt\\
&=-\int_{t_1}^{t_2}\int_{B_r}v_+(t,x)^{1-q} \phi(t,x)^p\left(\int_{B_r^c}v_+(t,y)^{p-1}J(t;x,y)\,dy\right)\,dx\,dt.
\end{align*}

For $I^{j,\epsilon}_1$, as $j\rightarrow \infty$ and $\epsilon\rightarrow 0$, we have
\begin{align*}
I^{j,\epsilon}_1
&\rightarrow\frac{1}{2}\int_{t_1}^{t_2}\int_{B_{r}\times B_{r}}|u(t,x)-u(t,y)|^{p-2} (u(t,x)-u(t,y))\\
&\qquad\qquad\qquad\quad\,\,\,\times\left(v_+(t,x)^{1-q}\phi(t,x)^p-v_+(t,y)^{1-q}\phi(t,y)^p\right)J(t;x,y)\,dx\,dy\,dt\\
&=:I_1.
\end{align*}
Now we consider the integral involved in. Without loss of generality, we assume that $u(t,x)\ge u(t,y)$. To this end, we denote $A_k:=A_{k,t,r}=\{u(t,\cdot)>k(t)\}\cap B_r$ for fixed $t\in(t_1,t_2)$ and $r>0$. We claim that
\begin{itemize}
  \item If $x\in B_r\setminus A_k$, then $u(t,x)\le k(t)$,  and so $u(t,y)\le k(t)$. In particular,
  $v_+(t,x)=v_+(t,y)=0$. Thus,
      $$|u(t,x)-u(t,y)|^{p-2} (u(t,x)-u(t,y))(v_+(t,x)^{1-q} \phi(t,x)^p- v_+(t,y)^{1-q} \phi(t,y)^p)=0.$$
  \item If $x\in A_k$ and $y\in B_r\setminus A_k$, then $u(t,x)>k(t)\ge u(t,y)$, and so
     \begin{align*}
     &|u(t,x)-u(t,y)|^{p-2} (u(t,x)-u(t,y))(v_+(t,x)^{1-q} \phi(t,x)^p- v_+(t,y)^{1-q} \phi(t,y)^p)\\
     &=\big(u(t,x)-k(t)+(k(t)-u(t,y))\big)^{p-1}v_+(t,x)^{1-q}\phi(t,x)^p\\
     &=\big(v_+(t,x)+v_-(t,y)\big)^{p-1}v_+(t,x)^{1-q}\phi(t,x)^p\\
     &\ge\frac{1}{2}\big(v_+(t,x)^{p-1}+v_-(t,y)^{p-1}\big)v_+(t,x)^{1-q}\phi(t,x)^p\\
     &=\frac{1}{2}v_+(t,x)^{p-q}\phi(t,x)^p+\frac{1}{2} v_-(t,y)^{p-1}v_+(t,x)^{1-q}\phi(t,x)^p\\
     &=\frac{1}{2}|w_+(t,x)-w_+(t,y)|^p\phi(t,x)^p+\frac{1}{2}v_-(t,y)^{p-1}v_+(t,x)^{1-q}\phi(t,x)^p\\
     &\ge\frac{1}{2} |w_+(t,x)-w_+(t,y)|^p\min\{\phi(t,x)^p,\phi(t,y)^p\}+\frac{1}{2}v_-(t,y)^{p-1}v_+(t,x)^{1-q}\phi(t,x)^p,
     \end{align*}
     where $w_+=v_+^{(p-q)/p}$ and in the last equality we used the fact that $y\in B_r\setminus A_k$, so $w_+(t,y)=0$.
  \item If $x\in A_k$ and $y\in A_k$, then $u(t,x)>k(t)$ and $u(t,y)>k(t)$. Thus, by \cite[Lemma 2.1]{KWZ},
    \begin{align*}
     &|u(t,x)-u(t,y)|^{p-2} (u(t,x)-u(t,y))(v_+(t,x)^{1-q} \phi(t,x)^p- v_+(t,y)^{1-q} \phi(t,y)^p)\\
     &=(v_+(t,x)-v_+(t,y))^{p-1}(v_+(t,x)^{1-q}\phi(t,x)^p-v_+(t,y)^{1-q}\phi(t,y)^p)\\
     &\ge c_1|w_+(t,x)-w_+(t,y)|^p\phi(t,x)^p-c_2w_+(t,x)^p|\phi(t,x)-\phi(t,y)|^p\\
     &\ge c_1|w_+(t,x)-w_+(t,y)|^p\min\{\phi(t,x)^p,\phi(t,y)^p\}-c_2\max\{w_+(t,x)^p,w_+(t,y)^p\}|\phi(t,x)-\phi(t,y)|^p.
     \end{align*}
\end{itemize}
Therefore, there are positive constants $c_3$ and $c_4$ such that
\begin{align*}
I_1\ge&2c_3\int_{t_1}^{t_2}\int_{A_k\times (B_r\setminus A_k)}|w_+(t,x)-w_+(t,y)|^p\min\{\phi(t,x)^p,\phi(t,y)^p\}J(t;x,y)\,dx\,dy\,dt\\
&+c_3\int_{t_1}^{t_2}\int_{A_k\times A_k}|w_+(t,x)-w_+(t,y)|^p\min\{\phi(t,x)^p,\phi(t,y)^p\}J(t;x,y)\,dx\,dy\,dt\\
&+c_3\int_{t_1}^{t_2}\int_{A_k\times(B_r\setminus A_k)}v_-(t,y)^{p-1}v_+(t,x)^{1-q}\phi(t,x)^pJ(t;x,y)\,dx\,dy\,dt\\
&-c_4\int_{t_1}^{t_2}\int_{A_k\times A_k}\max\{w_+(t,x)^p,w_+(t,y )^p\}|\phi(t,x)-\phi(t,y)|^pJ(t;x,y)\,dx\,dy\,dt\\
=&c_3\int_{t_1}^{t_2}\bigg(2\int_{A_k\times (B_r\setminus A_k)}+\int_{A_k\times A_k}
+\int_{(B_r\setminus A_k)\times (B_r\setminus A_k)}\bigg)
|w_+(t,x)-w_+(t,y)|^p\\
&\qquad\qquad\qquad\qquad\quad \times \min\{\phi(t,x)^p,\phi(t,y)^p\}J(t;x,y)\,dx\,dy\,dt\\
&+c_3\int_{t_1}^{t_2}
\Big(\int_{A_k\times(B_r\setminus A_k)}+\int_{(B_r\setminus A_k)\times(B_r\setminus A_k)}+
\int_{B_r\times A_k}\Big)v_-(t,y)^{p-1}v_+(t,x)^{1-q}\phi(t,x)^pJ(t;x,y)\,dx\,dy\,dt\\
&-c_4\int_{t_1}^{t_2}\int_{A_k\times A_k}\max\{w_+(t,x)^p,w_+(t,y )^p\}|\phi(t,x)-\phi(t,y)|^pJ(t;x,y)\,dx\,dy\,dt\\
\ge&
c_3\int_{t_1}^{t_2}\int_{B_r\times B_r}|w_+(t,x)-w_+(t,y)|^p\min\{\phi(t,x)^p,\phi(t,y)^p\}J(t;x,y)\,dx\,dy\,dt\\
&+c_3\int_{t_1}^{t_2}\int_{B_r\times B_r}v_-(t,y)^{p-1}v_+(t,x)^{1-q}\phi(t,x)^pJ(t;x,y)\,dx\,dy\,dt\\
&-c_4\int_{t_1}^{t_2}\int_{B_r\times B_r}\max\{w_+(t,x)^p,w_+(t,y )^p\}|\phi(t,x)-\phi(t,y)|^pJ(t;x,y)\,dx\,dy\,dt,
\end{align*}
where, on the right-hand side of the above equality, the third and the fifth terms are equal to $0$ due to the facts that $w_+(t,\cdot)=0$ and $v_+(t,\cdot)=0$ on $B_r\setminus A_k$, and the sixth term is also equal to $0$ due to the fact that $v_-(t,\cdot)=0$ on $A_k$; and on the right-hand side of the last inequality above, the first and the last terms are based on the symmetry and the fact that $A_k\subset B_r$, respectively.

According to all the estimates above, we find that
\begin{align*}
&c_3\int_{t_1}^{t_2}\int_{B_r\times B_r}|w_+(t,x)-w_+(t,y)|^p\min\{\phi(t,x)^p,\phi(t,y)^p\}J(t;x,y)\,dx\,dy\,dt\\
&\quad+c_3\int_{t_1}^{t_2}\int_{B_r}v_+(t,x)^{1-q}\phi(t,x)^p\left(\int_{B_r}v_-(t,y)^{p-1} J(t;x,y)\,dy\right)\,dx\,dt+\frac{1}{2-q}\int_{B_{r}} v_+(t,x)^{2-q}\phi(t,x)^p\,dx\bigg|^{t_2}_{t_1}\\
&\le  \int_{t_1}^{t_2}\int_{B_r}v_+(t,x)^{1-q} \phi(t,x)^p\left(\int_{B_r^c}v_+(t,y)^{p-1}J(t;x,y)\,dy\right)\,dx\,dt\\
&\quad+c_4\int_{t_1}^{t_2}\int_{B_r\times B_r}\max\{w_+(t,x)^p,w_+(t,y )^p\}|\phi(t,x)-\phi(t,y)|^pJ(t;x,y)\,dx\,dy\,dt\\
&\quad-\int_{t_1}^{t_2}\int_{B_{r}}
v_+(t,x)^{1-q}\phi(t,x)^pk'(t)\,dx\,dt+\frac{1}{2-q}\int_{t_1}^{t_2}\int_{B_{r}}|\partial_t\phi(t,x)^{p}|v_+(t,x)^{2-q }\,dx\,dt.
\end{align*}
This completes the proof of the assertion for the weak sub-solution by $q=1-\xi$ and $t_1=\tau_1-h$.

(2) Suppose that $u(t,x)$ is a locally bounded weak super-solution
in $(\tau_1-h,\tau_2)\times B_{r}$. Then, $-u(t,x)$ is a locally bounded weak sub-solution in $(\tau_1-h,\tau_2)\times B_{r}$. For $\tau_1-h<t_1<t_2<\tau_2$,  let $\theta_j(t)\in C^\infty_c(\tau_1-h,\tau_2)$ be a smooth approximation of
$\I_{(t_1,t_2)}(t)$ as $j\rightarrow\infty$. Let $q=1-\xi<1$. Then, the following is an admissible test function
$$\eta_{j,\epsilon}(t,x):=((v_-)_\epsilon^{1-q}\phi^p\theta_j)_\epsilon(t,x),$$
where the subscript $\epsilon$ on the right hand side denotes the mollification with respect to $t$ in the sense of \eqref{epsilon}. We find that
\begin{align*}
0&\ge
\int_{\tau_1-h}^{\tau_2}
\mathcal{E}_t(-u(t,x),\eta_{j,\epsilon}(t,x))\,dt+\int_{B_r} (-u(t,x))\eta_{j,\epsilon}(t,x)\,dx\Big|^{\tau_2}_{\tau_1-h}-\int_{\tau_1-h}^{\tau_2}\int_{B_r} (-u(t,x))\partial_t\eta_{j,\epsilon}(t,x)\,dx\,dt\\
&=-\frac{1}{2}\int_{\tau_1-h}^{\tau_2}\int_{B_{r}\times B_{r}}|u(t,x)-u(t,y)|^{p-2} (u(t,x)-u(t,y))\\
&\qquad\qquad\qquad\qquad\,\,\,\,\times(((v_-)_\epsilon^{1-q}\phi^p\theta_j)_\epsilon(t,x)-((v_-)_\epsilon^{1-q}\phi^p\theta_j)_\epsilon(t,y)) J(t;x,y)\,dx\,dy\,dt\\
&\quad-\int_{\tau_1-h}^{\tau_2}\int_{B_{r}\times B^c_{r}}|u(t,x)-u(t,y)|^{p-2} (u(t,x)-u(t,y))((v_-)_\epsilon^{1-q}\phi^p\theta_j)_\epsilon(t,x)J(t;x,y)\,dx\,dy\,dt\\
&\quad-\int_{B_r}u(t,x)\eta_{j,\epsilon}(t,x)\,dx\Big|^{\tau_2}_{\tau_1-h}\\
&\quad+\int_{\tau_1-h}^{\tau_2}\int_{B_r} u(t,x)\partial_t\eta_{j,\epsilon}(t,x)\,dx\,dt\\
&=:I^{j,\epsilon}_1+I^{j,\epsilon}_2+I^{j,\epsilon}_3+I^{j,\epsilon}_4.
\end{align*}

We can now proceed similarly
to that in (1). For the sake of completeness, we write the details here.
For $I^{j,\epsilon}_4$,
\begin{align*}
I^{j,\epsilon}_4&=\int_{\tau_1-h}^{\tau_2}\int_{B_r} u(t,x)\partial_t\left(\int_{\R}(v_-)_\epsilon(s,x)^{1-q}\phi(s,x)^p\theta_j(s)\zeta_\epsilon(t-s)\,ds\right)\,dx\,dt\\
&=\int_{\tau_1-h}^{\tau_2}\int_{B_r} u(t,x)\int_{\R}(v_-)_\epsilon(s,x)^{1-q}\phi(s,x)^p\theta_j(s)\partial_t\zeta_\epsilon(t-s)\,ds\,dx\,dt\\
&=-\int_{\tau_1-h}^{\tau_2}\int_{B_r}\int_{\R}(v_-)_\epsilon(s,x)^{1-q}\phi(s,x)^p\theta_j(s)u(t,x)\partial_s\zeta_\epsilon(s-t)\,ds\,dx\,dt\\
&=-\int_{B_r}\int_{\R}(v_-)_\epsilon(s,x)^{1-q}\phi(s,x)^p\theta_j(s)\left(\int_{\tau_1-h}^{\tau_2}u(t,x)\partial_s\zeta_\epsilon(s-t)\,dt\right)\,ds\,dx\\
&=-\int_{B_r}\int_{\R}(v_-)_\epsilon(s,x)^{1-q}\phi(s,x)^p\theta_j(s)\partial_s\left(\int_{\R}u(t,x)\zeta_\epsilon(s-t)\,dt\right)\,ds\,dx\\
&=-\int_{B_r}\int_{\R}(v_-)_\epsilon(s,x)^{1-q}\phi(s,x)^p\theta_j(s)\partial_su_\epsilon(s,x)\,ds\,dx\\
&=-\int_{\tau_1-h}^{\tau_2}\int_{B_r}(v_-)_\epsilon(t,x)^{1-q}\phi(t,x)^p\theta_j(t)\partial_tu_\epsilon(t,x)\,dx\,dt,
\end{align*}
where in the third equality we used the fact that the mollifier $\zeta_\epsilon$ is an even function, in the fifth equality we used the facts that $\supp[\zeta_\epsilon]=[-\epsilon,\epsilon]$, $\supp[\theta_j]\subseteq[t_1,t_2]$, and so $u(t,x)\zeta_\epsilon(s-t)=0$ when $t\in(\tau_1-h,\tau_2)^c\subset(t_1-\epsilon,t_2-\epsilon)^c$ with $\epsilon$ small enough, and in the second to last equality we used the fact that $\supp[\theta_j]\subseteq[t_1,t_2]\subset(\tau_1-h,\tau_2)$.  We also note that the dominated convergence theorem was used in the second and fifth equalities as well. As $j\rightarrow \infty$, we have
\begin{align*}
I^{j,\epsilon}_4
&\rightarrow-\int_{t_1}^{t_2}\int_{B_r}(v_-)_\epsilon(t,x)^{1-q}\phi(t,x)^p\partial_tu_\epsilon(t,x)\,dx\,dt\\
&=\int_{t_1}^{t_2}\int_{B_r}(v_-)_\epsilon(t,x)^{1-q}\phi(t,x)^p\partial_t(v_-)_\epsilon(t,x)\,dx\,dt-\int_{t_1}^{t_2}\int_{B_r}(v_-)_\epsilon(t,x)^{1-q}\phi(t,x)^pk'_\epsilon(t)\,dx\,dt\\
&=:I^{\epsilon}_4,
\end{align*}
where in the first equality we used the fact that for $u(t,x)\le k(t)$, $\partial_tu_\epsilon(t,x)=\partial_t[u(t,x)-k(t)+k(t)]_\epsilon=\partial_t[-v_-(t,x)+k(t)]_\epsilon=-\partial_t(v_-)_\epsilon(t,x)+k_\epsilon'(t)$.
Then, the integration by parts yields
\begin{align*}
I^{\epsilon}_4
&=\frac{1}{2-q}\int_{t_1}^{t_2}\int_{B_r}\partial_t(v_-)_\epsilon(t,x)^{2-q}\phi(t,x)^p\,dx\,dt-\int_{t_1}^{t_2}\int_{B_r}(v_-)_\epsilon(t,x)^{1-q}\phi(t,x)^pk'_\epsilon(t)\,dx\,dt\\
&=\frac{1}{2-q}\int_{B_r}(v_-)_\epsilon(t,x)^{2-q}\phi(t,x)^p\,dx\Big|^{t_2}_{t_1}-\frac{1}{2-q}\int_{t_1}^{t_2}\int_{B_r}\partial_t\phi(t,x)^p(v_-)_\epsilon(t,x)^{2-q}\,dx\,dt\\
&\quad-\int_{t_1}^{t_2}\int_{B_r}(v_-)_\epsilon(t,x)^{1-q}\phi(t,x)^pk'_\epsilon(t)\,dx\,dt\\
&\rightarrow\frac{1}{2-q}\int_{B_r}v_-(t,x)^{2-q}\phi(t,x)^p\,dx\Big|^{t_2}_{t_1}-\frac{1}{2-q}\int_{t_1}^{t_2}\int_{B_r}\partial_t\phi(t,x)^pv_-(t,x)^{2-q}\,dx\,dt\\
&\quad-\int_{t_1}^{t_2}\int_{B_r}v_-(t,x)^{1-q}\phi(t,x)^pk'(t)\,dx\,dt,
\end{align*}
as $\epsilon\rightarrow0$.

For $I^{j,\epsilon}_3$, since
$$\eta_{j,\epsilon}(t,x)=\int_{\R}(v_-)_\epsilon(s,x)^{1-q}\phi(s,x)^p\theta_j(s)\zeta_\epsilon(t-s)\,ds,$$ $\supp[\zeta_\epsilon]=[-\epsilon,\epsilon]$ and $\supp[\theta_j]\subseteq[t_1,t_2]$, $\supp[\eta_{j,\epsilon}(\cdot,x)
]\subseteq(t_1-\epsilon,t_2+\epsilon)\subset(\tau_1-h,\tau_2)$ for small $\epsilon>0$. Thus, $I^{j,\epsilon}_3=0$.

For $I^{j,\epsilon}_2$, as $j\rightarrow \infty$ and $\epsilon\rightarrow0$, we have
\begin{align*}
I^{j,\epsilon}_2&\rightarrow-\int_{t_1}^{t_2}\int_{B_{r}\times B^c_{r}}|u(t,x)-u(t,y)|^{p-2} (u(t,x)-u(t,y))v_-(t,x)^{1-q}\phi(t,x)^pJ(t;x,y)\,dx\,dy\,dt\\
&\ge-\int_{t_1}^{t_2}\int_{B_{r}\times B^c_{r}}(u(t,y)-u(t,x))_-^{p-1}v_-(t,x)^{1-q}\phi(t,x)^pJ(t;x,y)\,dx\,dy\,dt\\
&\ge-\int_{t_1}^{t_2}\int_{B_{r}\times B^c_{r}}(u(t,y)-k(t))_-^{p-1}v_- (t,x)^{1-q}\phi(t,x)^pJ(t;x,y)\,dx\,dy\,dt\\
&=-\int_{t_1}^{t_2}\int_{B_r}v_-(t,x)^{1-q} \phi(t,x)^p\left(\int_{B_r^c}v_-(t,y)^{p-1}J(t;x,y)\,dy\right)\,dx\,dt.
\end{align*}

For $I^{j,\epsilon}_1$, as $j\rightarrow \infty$ and $\epsilon\rightarrow0$, we have
\begin{align*}
I^{j,\epsilon}_1
&\rightarrow\frac{1}{2}\int_{t_1}^{t_2}\int_{B_{r}\times B_{r}}-|u(t,x)-u(t,y)|^{p-2} (u(t,x)-u(t,y))\\
&\qquad\qquad\qquad\quad\,\,\,\times\left(v_-(t,x)^{1-q}\phi(t,x)^p-v_-(t,y)^{1-q}\phi(t,y)^p\right)J(t;x,y)\,dx\,dy\,dt\\
&=:I_1.
\end{align*}
Now we consider the integral involved in. Without loss of generality, we assume that $u(t,x)\ge u(t,y)$. To this end, we denote $A_k:=A_{k,t,r}=\{u(t,\cdot)\le k(t)\}\cap B_r$ for fixed $t\in(t_1,t_2)$ and $r>0$. We claim that
\begin{itemize}
  \item If $y\in B_r\setminus A_k$, then $u(t,y)>k(t)$,  and so $u(t,x)>k(t)$. In particular,
  $v_-(t,x)=v_-(t,y)=0$. Thus,
      $$-|u(t,x)-u(t,y)|^{p-2} (u(t,x)-u(t,y))(v_-(t,x)^{1-q}\phi(t,x)^p- v_-(t,y)^{1-q}\phi(t,y)^p)=0.$$
  \item If $y\in A_k$ and $x\in B_r\setminus A_k$, then $u(t,x)>k(t)\ge u(t,y)$, and so
     \begin{align*}
     &-|u(t,x)-u(t,y)|^{p-2} (u(t,x)-u(t,y))(v_-(t,x)^{1-q} \phi(t,x)^p- v_-(t,y)^{1-q} \phi(t,y)^p)\\
     &=\big(u(t,x)-k(t)+(k(t)-u(t,y))\big)^{p-1}v_-(t,y)^{1-q}\phi(t,y)^p\\
     &=\big(v_+(t,x)+v_-(t,y)\big)^{p-1}v_-(t,y)^{1-q}\phi(t,y)^p\\
     &\ge\frac{1}{2}\big(v_+(t,x)^{p-1}+v_-(t,y)^{p-1}\big)v_-(t,y)^{1-q}\phi(t,y)^p\\
     &=\frac{1}{2}v_-(t,y)^{p-q}\phi(t,y)^p+\frac{1}{2}v_+(t,x)^{p-1}v_-(t,y)^{1-q}\phi(t,y)^p\\
     &=\frac{1}{2}|w_-(t,y)-w_-(t,x)|^p\phi(t,y)^p+\frac{1}{2}v_+(t,x)^{p-1}v_-(t,y)^{1-q}\phi(t,y)^p\\
     &\ge\frac{1}{2}|w_-(t,y)-w_-(t,x)|^p\min\{\phi(t,x)^p,\phi(t,y)^p\}+\frac{1}{2}v_+(t,x)^{p-1}v_-(t,y)^{1-q}\phi(t,y)^p,
     \end{align*}
     where $w_-=v_-^{(p-q)/p}$ and in the last equality we used the fact that $x\in B_r\setminus A_k$, so $w_-(t,x)=0$.
  \item If $x\in A_k$ and $y\in A_k$, then $u(t,x)\le k(t)$ and $u(t,y)\le k(t)$. Thus, by \cite[Lemma 2.1]{KWZ},
    \begin{align*}
     &-|u(t,x)-u(t,y)|^{p-2}(u(t,x)-u(t,y))(v_-(t,x)^{1-q}\phi(t,x)^p-v_-(t,y)^{1-q}\phi(t,y)^p)\\
     &=(k(t)-u(t,y)-(k(t)-u(t,x)))^{p-1}(v_-(t,y)^{1-q}\phi(t,y)^p-v_-(t,x)^{1-q}\phi(t,x)^p)\\
     &=(v_-(t,y)-v_-(t,x))^{p-1}(v_-(t,y)^{1-q}\phi(t,y)^p-v_-(t,x)^{1-q}\phi(t,x)^p)\\
     &\ge c_1|w_-(t,x)-w_-(t,y)|^p\phi(t,y)^p-c_2w_-(t,y)^p|\phi(t,x)-\phi(t,y)|^p\\
     &\ge c_1|w_-(t,x)-w_-(t,y)|^p\min\{\phi(t,x)^p,\phi(t,y)^p\}-c_2\max\{w_-(t,x)^p,w_-(t,y)^p\}|\phi(t,x)-\phi(t,y)|^p.
     \end{align*}
\end{itemize}
Therefore, there are positive constants $c_3$ and $c_4$ such that
\begin{align*}
I_1\ge&2c_3\int_{t_1}^{t_2}\int_{(B_r\setminus A_k)\times A_k}|w_-(t,x)-w_-(t,y)|^p\min\{\phi(t,x)^p,\phi(t,y)^p\}J(t;x,y)\,dx\,dy\,dt\\
&+c_3\int_{t_1}^{t_2}\int_{A_k\times A_k}|w_-(t,x)-w_-(t,y)|^p\min\{\phi(t,x)^p,\phi(t,y)^p\}J(t;x,y)\,dx\,dy\,dt\\
&+c_3\int_{t_1}^{t_2}\int_{(B_r\setminus A_k)\times A_k}v_+(t,x)^{p-1}v_-(t,y)^{1-q}\phi(t,y)^pJ(t;x,y)\,dx\,dy\,dt\\
&-c_4\int_{t_1}^{t_2}\int_{A_k\times A_k}\max\{w_-(t,x)^p,w_-(t,y )^p\}|\phi(t,x)-\phi(t,y)|^pJ(t;x,y)\,dx\,dy\,dt\\
=&c_3\int_{t_1}^{t_2}\bigg(2\int_{(B_r\setminus A_k)\times A_k}+\int_{A_k\times A_k}
+\int_{(B_r\setminus A_k)\times (B_r\setminus A_k)}\bigg)
|w_-(t,x)-w_-(t,y)|^p\\
&\qquad\qquad\qquad\qquad\quad \times \min\{\phi(t,x)^p,\phi(t,y)^p\}J(t;x,y)\,dx\,dy\,dt\\
&+c_3\int_{t_1}^{t_2}
\bigg(\int_{(B_r\setminus A_k)\times A_k}+\int_{(B_r\setminus A_k)\times(B_r\setminus A_k)}+
\int_{A_k\times B_r}\bigg)v_+(t,x)^{p-1}v_-(t,y)^{1-q}\phi(t,y)^pJ(t;x,y)\,dx\,dy\,dt\\
&-c_4\int_{t_1}^{t_2}\int_{A_k\times A_k}\max\{w_-(t,x)^p,w_-(t,y )^p\}|\phi(t,x)-\phi(t,y)|^pJ(t;x,y)\,dx\,dy\,dt\\
\ge&
c_3\int_{t_1}^{t_2}\int_{B_r\times B_r}|w_-(t,x)-w_-(t,y)|^p\min\{\phi(t,x)^p,\phi(t,y)^p\}J(t;x,y)\,dx\,dy\,dt\\
&+c_3\int_{t_1}^{t_2}\int_{B_r\times B_r}v_+(t,y)^{p-1}v_-(t,x)^{1-q}\phi(t,x)^pJ(t;x,y)\,dx\,dy\,dt\\
&-c_4\int_{t_1}^{t_2}\int_{B_r\times B_r}\max\{w_-(t,x)^p,w_-(t,y )^p\}|\phi(t,x)-\phi(t,y)|^pJ(t;x,y)\,dx\,dy\,dt,
\end{align*}
where, on the right-hand side of the above equality, the third and the fifth terms are equal to $0$ due to the facts that $w_-(t,\cdot)=0$ and $v_-(t,\cdot)=0$ on $B_r\setminus A_k$, and the sixth term is also equal to $0$ due to the fact that $v_+(t,\cdot)=0$ on $A_k$; and the last inequality is based on the symmetry and the fact that $A_k\subset B_r$.

According to all the estimates above, we find that
\begin{align*}
&c_3\int_{t_1}^{t_2}\int_{B_r\times B_r}|w_-(t,x)-w_-(t,y)|^p\min\{\phi(t,x)^p,\phi(t,y)^p\}J(t;x,y)\,dx\,dy\,dt\\
&\quad+c_3\int_{t_1}^{t_2}\int_{B_r}v_-(t,x)^{1-q}\phi(t,x)^p\left(\int_{B_r}v_+(t,y)^{p-1} J(t;x,y)\,dy\right)\,dx\,dt+\frac{1}{2-q}\int_{B_{r}} v_-(t,x)^{2-q}\phi(t,x)^p\,dx\bigg|^{t_2}_{t_1}\\
&\le\int_{t_1}^{t_2}\int_{B_r}v_-(t,x)^{1-q} \phi(t,x)^p\left(\int_{B_r^c}v_-(t,y)^{p-1}J(t;x,y)\,dy\right)\,dx\,dt\\
&\quad+c_4\int_{t_1}^{t_2}\int_{B_r\times B_r}\max\{w_-(t,x)^p,w_-(t,y)^p\}|\phi(t,x)-\phi(t,y)|^pJ(t;x,y)\,dx\,dy\,dt\\
&\quad+\int_{t_1}^{t_2}\int_{B_{r}}
v_-(t,x)^{1-q}\phi(t,x)^pk'(t)\,dx\,dt+\frac{1}{2-q}\int_{t_1}^{t_2}\int_{B_{r}}|\partial_t\phi(t,x)^{p}|v_-(t,x)^{2-q }\,dx\,dt.
\end{align*}
This completes the proof of the assertion for the weak super-solution by $q=1-\xi$ and $t_1=\tau_1-h$.
\end{proof}

\bigskip

\noindent {\bf Acknowledgements}\,\,
 The research of Takashi Kumagai is supported
by JSPS KAKENHI Grant Number JP22H00099 and 23KK0050.
The research of Jian Wang is supported by the National Key R\&D Program of China (2022YFA1006003) and  the National Natural Science Foundation of China (Nos.\ 12071076 and 12225104).

\ \ 

\noindent {\bf AI Disclosure}\,\, The AI tool was not used
in any capacity, in particular, the arguments, of the paper.

\vskip 0.3truein
{\bf Takashi Kumagai:}
 Department of Mathematics,
Waseda University, Shinjuku, Tokyo, 169-8555, Japan.
\newline \texttt{t-kumagai@waseda.jp}

\bigskip

{\bf Jian Wang:}
    School of Mathematics and Statistics \& Key Laboratory of Analytical Mathematics and Applications (Ministry of Education) \& Fujian Provincial Key Laboratory
of Statistics and Artificial Intelligence, Fujian Normal University, 350117, Fuzhou, P.R. China.
     \newline \texttt{jianwang@fjnu.edu.cn}

\bigskip

{\bf Meng-ge Zhang:}
    School of Mathematics and Statistics, Ningbo University, 315211, Ningbo, P.R. China.
     \newline \texttt{mgzhangmath@163.com}
\end{document}